\documentclass[a4paper,reqno,10pt]{amsart}

\usepackage{amsfonts,amssymb,amsthm,amsmath,amscd,mathtools,mathrsfs}
\usepackage{graphicx,color,xcolor,cite,leftidx,comment,enumitem}
\usepackage{fancybox,multirow,makecell,mathdots,caption,subcaption}
\usepackage{extarrows,tasks,thmtools}
\usepackage[all,2cell,cmtip]{xy}
\xyoption{curve}
\usepackage{tikz,tikz-cd,float}
\usetikzlibrary{arrows,decorations.pathmorphing,decorations.markings,backgrounds,positioning,fit,petri,patterns,matrix,shapes.geometric}
\usepackage[hypertexnames=false,draft]{hyperref}
\usepackage[nameinlink]{cleveref}

\let\shorttwoheadrightarrow\twoheadrightarrow
\renewcommand{\hookrightarrow}{\lhook\joinrel\longrightarrow}
\renewcommand{\twoheadrightarrow}{\relbar\joinrel\relbar\joinrel\shorttwoheadrightarrow}
\makeatletter
\newcommand{\xdashrightarrow}[2][]{\ext@arrow 0359\tofill@@{#1}{#2}}
\newcommand{\shortdash}{
  \mathord{\rule[0.54ex]{0.75ex}{0.37pt}}
}
\def\tofill@@{\arrowfill@@\relax\shortdash\dashrightarrow}
\def\arrowfill@@#1#2#3#4{
$\m@th\thickmuskip0mu\medmuskip\thickmuskip\thinmuskip\thickmuskip
   \relax#4#1
   \xleaders\hbox{$#4\mkern1.6mu#2\mkern1.6mu$}\hfill
   #3$
}
\def\Hom{\operatorname{Hom}}
\declaretheoremstyle[headfont=\bfseries,bodyfont=\itshape]{plainstyle}
\declaretheoremstyle[headfont=\bfseries,bodyfont=\normalfont]{defstyle}
\declaretheoremstyle[headfont=\bfseries,bodyfont=\normalfont]{remarkstyle}
\declaretheorem[style=plainstyle,numberwithin=section,name=Theorem]{theorem}
\declaretheorem[style=plainstyle,sibling=theorem,name=Lemma]{lemma}
\declaretheorem[style=plainstyle,sibling=theorem,name=Proposition]{proposition}
\declaretheorem[style=plainstyle,sibling=theorem,name=Definition-Proposition]{definition-proposition}
\declaretheorem[style=plainstyle,sibling=theorem,name=Corollary]{corollary}
\declaretheorem[style=defstyle,sibling=theorem,name=Definition]{definition}
\declaretheorem[style=defstyle,sibling=theorem,name=Example]{example}
\declaretheorem[style=remarkstyle,sibling=theorem,name=Remark]{remark}

\declaretheorem[style=defstyle,sibling=theorem,name=Fact]{fact}
\declaretheorem[style=defstyle,sibling=theorem,name=Definition-Lemma]{defnlem}
\crefname{equation}{}{}
\Crefname{equation}{}{}
\numberwithin{equation}{section}
\title[Exceptional model structures]{Exceptional model structures and \\ the induced extriangulated categories}
\author{Chencheng Zhang, \ \ Pu Zhang$^\ast$}
\dedicatory{Dedicated to the memory of Idun Reiten}
\thanks{$^*$ Corresponding author}
\thanks{zhangchencheng@sjtu.edu.cn \ \ pzhang@sjtu.edu.cn}
\thanks{Supported by National Natural Science Foundation of China with Grant No. 12131015.}
\address{Chencheng Zhang, School of Mathematical Sciences, Shanghai Jiao Tong University, Shanghai 200240, PR China}
\address{Pu Zhang, School of Mathematical Sciences, Shanghai Jiao Tong University, Shanghai 200240, PR China}

\begin{document}
\begin{abstract} A Hovey triple $(\mathcal{C}, \mathcal{F}, \mathcal{W})$ is  {\it exceptional}, if
$(\mathcal{C} \cap \mathcal{F}, \ \mathcal{C} \cap \mathcal{F} \cap \mathcal{W})$ is not a Frobenius pair. One has a disjoint union
$\{\text{Hovey triple}\} = \{\text{Hereditary Hovey triple}\} \ \dot\bigcup \ \{\text{Exceptional Hovey triple}\}$ \ $\dot\bigcup \ \{\text{non-hereditary and non-exceptional Hovey triple}\}.$
Exceptional Hovey triples appear widely. For a selfinjective Nakayama algebra $A= kC_n/J^t$,
$A\mbox{-}{\rm mod}$ admits exceptional Hovey triples if and only if $\gcd (n, t) \ge 2$ and $t \geq 3$. Their homotopy categories reveal new
phenomena. Nakaoka-Palu's Theorem implies that there is a triangulation on $\frac{\mathcal C\cap\mathcal F}{\mathcal C\cap\mathcal F\cap\mathcal W}.$
It is proved that a Hovey triple in a weakly idempotent complete extriangulated category $(\mathcal A, \mathbb E, \mathfrak s)$ is exceptional if and only if
the induced extriangulated structure $\bigl(\frac{\mathcal C\cap\mathcal F}{\mathcal C\cap\mathcal F\cap\mathcal W}, \ \overline{\mathbb E}, \ \overline{\mathfrak s}\bigr)$ is not {\it canonically triangulated}.
Between this extriangulated structure and the one arising from the triangulation, the intermediate extriangulated structures are in one-to-one correspondence with Serre subcategories of  $\mathrm{fp}((\frac{\mathcal P(\mathcal C\cap\mathcal F)}{\mathcal C\cap\mathcal F\cap\mathcal W})^{\mathrm{op}}, \mathrm{Ab})$. Moreover, the homotopy relation and the homotopy category of an arbitrary model structure on a weakly idempotent complete additive category are explicitly described.

  \vskip5pt
  {\it Keywords and phrases.} Frobenius pair, exceptional Hovey triple, cotorsion pair, Nakayama algebra, extriangulated category.
  \vskip5pt
  2020 Mathematics Subject Classification. 18N40,  16G20, 16G70, 18G80, 18G65, 18E35.
\end{abstract}
\maketitle
\section{\bf Introduction}

\vskip5pt

\subsection{Exceptional exact model structures} $\,$

\vskip5pt

From the viewpoint of representation theory, exact model structures are the most important model structures: they are uniquely determined by Hovey triples (\cite{Hov02}), and their homotopy categories are triangulated (\cite{NP19}).
Let $(\mathcal C,\mathcal F,\mathcal W)$ be a Hovey triple in a weakly idempotent complete extriangulated category $(\mathcal A,\mathbb E,\mathfrak s)$.
Then the homotopy category $\mathsf{Ho}(\mathcal{A})$ is equivalent to the additive quotient $\frac{\mathcal{C} \cap \mathcal{F}}{\mathcal{C} \cap \mathcal{F} \cap \mathcal{W}}$ as an additive category.

\vskip5pt

Note that $\mathcal{C} \cap \mathcal{F}$ is an extension-closed subcategory, it follows that it inherits an extriangulated structure from $\mathcal{A}$,
with the extension bifunctor given by the restriction of $\mathbb E$ to $(\mathcal C\cap\mathcal F)^{\mathrm{op}}\times(\mathcal C\cap\mathcal F)$.
If the Hovey triple is hereditary, then $(\mathcal{C} \cap \mathcal{F}, \ \mathcal{C} \cap \mathcal{F} \cap \mathcal{W})$ is a Frobenius pair (i.e.,
$\mathcal{C} \cap \mathcal{F}$ is a Frobenius  extriangulated category, and $\mathcal{C} \cap \mathcal{F} \cap \mathcal{W}$ is precisely the full subcategory of projective-injective objects), and $\mathsf{Ho}(\mathcal{A})\cong \frac{\mathcal{C} \cap \mathcal{F}}{\mathcal{C} \cap \mathcal{F} \cap \mathcal{W}}$ is triangulated.
However, as H. Nakaoka and Y. Palu \cite[Theorem 6.20]{NP19} have shown, for an arbitrary exact model structure which is not necessarily hereditary, $\mathsf{Ho}(\mathcal{A})$ is always triangulated.

\vskip5pt

Recently, J. Gillespie \cite[Theorems~3.2 and~3.4]{Gil26} gives a new method to construct Hovey triples. With this approach, non-hereditary Hovey triples could be obtained by combining a non-hereditary cotorsion pair with an injective cotorsion pair
(cf. Theorem \ref{theorem_merge}). The first example of (non-trivial) non-hereditary Hovey triple is given in \cite[Example 4.2]{Gil26}, and even in this example,
the pair $(\mathcal{C} \cap \mathcal{F}, \ \mathcal{C} \cap \mathcal{F} \cap \mathcal{W})$ is still a Frobenius pair.

\vskip5pt

Thus, it is natural to ask whether there exists {\it exceptional exact model structures}, defined as follows.

\vskip5pt

\noindent {\bf Definition} {\rm (Definition \ref{exceptional})} \ {\it $(1)$ \ A pair $(\mathcal A, \mathcal P)$ is {\it a Frobenius pair}, if $\mathcal A$ is a Frobenius  extriangulated category, and $\mathcal P = \mathcal P(\mathcal A),$  where $\mathcal{P}(\mathcal{A})$ is the full subcategory of projective-injective objects.

\vskip5pt

$(2)$ \ A \textup{Hovey} triple $(\mathcal{C}, \mathcal{F}, \mathcal{W})$ in an extriangulated category $(\mathcal A,\mathbb E,\mathfrak s)$ is \emph{exceptional},
if the pair $(\mathcal{C} \cap \mathcal{F}, \ \mathcal{C} \cap \mathcal{F} \cap \mathcal{W})$ is not a Frobenius pair $($as we have remarked that
$(\mathcal{C} \cap \mathcal{F}, \mathbb E|_{(\mathcal C\cap\mathcal F)^{\mathrm{op}}\times(\mathcal C\cap\mathcal F)}, \mathfrak s|_{(\mathcal C\cap\mathcal F)^{\mathrm{op}}\times(\mathcal C\cap\mathcal F)})$ is an extriangulated category$)$. More explicitly,

$\quad \textup{(i)}$ \ An exceptional  \textup{Hovey} triple $(\mathcal{C}, \mathcal{F}, \mathcal{W})$ is of \emph{Type} {\rm I},  if $\mathcal{C} \cap \mathcal{F}$ is a Frobenius extriangulated category,
but $\mathcal P(\mathcal{C} \cap \mathcal{F})\ne \mathcal{C} \cap \mathcal{F} \cap \mathcal{W}$ $($or equivalently, $\mathcal{C} \cap \mathcal{F} \cap \mathcal{W} \subsetneqq \mathcal P(\mathcal{C} \cap \mathcal{F}))$.

$\quad \textup{(ii)}$ \ An exceptional  \textup{Hovey} triple $(\mathcal{C}, \mathcal{F}, \mathcal{W})$ is of \emph{Type} {\rm II},  if $\mathcal{C} \cap \mathcal{F}$ is not a Frobenius extriangulated category.

\vskip5pt

$(3)$ \ An exact model structure on an extriangulated category is \emph{an exceptional exact model structure}, or simply, \emph{an exceptional model structure}, provided that the corresponding \textup{Hovey} triple under the \textup{Hovey} correspondence is exceptional.}

\vskip5pt

It is clear that an exceptional model structure is not hereditary; however, the converse is not true:
Gillespie's example \cite[Example 4.2]{Gil26} gives a non-hereditary and non-exceptional Hovey triple.
One has the partition

$$\{\text{Hovey triple}\} = \left \{\substack{\text{Hereditary} \\ \text {Hovey triple}}\right\} \ \dot\bigcup \ \left \{\substack{\text{Exceptional} \\ \text {Hovey triple}}\right\}
  \ \dot\bigcup \ \left \{\substack{\text{non-hereditary and} \\ \text{non-exceptional} \\ \text {Hovey triple}}\right\}.$$

\vskip5pt
\noindent
As we will see, the two types of exceptional model structures appear widely.

\subsection{\bf Exceptional \textup{Hovey} triples by selfinjective Nakayama algebras}$\,$

\vskip5pt

To construct exceptional \textup{Hovey} triples, consider the category $A\mbox{-}{\rm mod}$  of finitely generated left $A$-modules, where $A$ is a Nakayama algebra without simple projective module.
Thus $A\cong k C_n/I$, where $C_n \ (n\ge 1)$ is the quiver
\begin{equation*}
  \xymatrix{
  1 \ar@{->}[r]^{a_2} & 2 \ar@{->}[r]^{a_3} & 3 \ar@{->}[r]^{a_4} & \cdots \ar@{->}[r]^{a_n} & n \ar@/^1pc/@{->}[llll]^{a_1}
  }
\end{equation*}

\vskip10pt
\noindent $I$ is an ideal of the path algebra $kC_n$ generated by
some paths with $J^m \subseteq I \subseteq J^2$, and $J$ is the ideal generated by arrows. Let $S_1, \dots, S_n$ be the pairwise non-isomorphic simple $A$-modules,
$P_i$ (respectively, $I_i$) the indecomposable projective (respectively, injective) module with top (respectively, socle) $S_i$,
$(c_1, \ldots, c_n)$ the Kupisch series of $A$. Denote by
$M_{i, l}$ the indecomposable $A$-module with top $S_i$ and length $l$.
Then $M_{i, l}$ \ ($1\le l \le c_i$, $i\in \mathbb Z/n\mathbb Z$) give all the indecomposable $A$-modules.

\vskip5pt

For $M\in A\mbox{-}{\rm mod}$, denote by $\operatorname{add}M$ the full subcategory consisting of direct summands of finite direct sums of copies of $M$.

\vskip5pt

We first consider selfinjective Nakayama algebras $A$. This is precisely the case $A = kC_n/J^t$ with $t\ge 2.$
Then \ $\mathcal{P}(A) = \mathcal{I}(A) = \operatorname{add} (P)$ where $P  = \bigoplus\limits _{1\le i\le n} P_i  = \bigoplus\limits_{1\le i \le n} M_{i, t}.$

\vskip5pt

\noindent {\bf Theorem A.} \label{introthm1} {\rm (Theorem \ref{thm_classification}, Theorem \ref{type})} \ {\it Let $A= kC_n/J^t \ (t\ge 2)$ be a selfinjective Nakayama algebra. Then

\vskip5pt

$(1)$ \ $A\mbox{-}{\rm mod}$ admits an exceptional \textup{Hovey} triple if and only if $d \ge 2$ and $t \geq 3$, where $d = \gcd(n, t)$.

\vskip5pt

$(2)$ \ Assume that $d \ge 2$ and $t\ge 3$. Put $R = d\mathbb Z/n\mathbb Z \subseteq \mathbb Z/n\mathbb Z$.
Then $(\mathcal{C}, \mathcal{F}, \mathcal{W})$ is an exceptional \textup{Hovey} triple in $A\mbox{-}{\rm mod}$, where
  {\small \begin{align*}\mathcal{C} & =\operatorname{add}(P\oplus\bigoplus\limits_{\substack{i\notin R, \  1\leq l<t}}M_{i,l}),  \ \  \ \
              \mathcal{F} =\operatorname{add}(P\oplus\bigoplus\limits_{\substack{i+l-1\notin R,                                 \\ 1\leq l<t}}M_{i,l}) \\
              \mathcal{W} & =\operatorname{add}(P\oplus\bigoplus\limits_{i\in R}(M_{i,1}\oplus M_{i+1,t-1})).\end{align*}}
Moreover, it is of \textup{Type I} if and only if $d=2$, and of \textup{Type II} if and only if $d\ge 3$.}

\vskip5pt

Theorem A and Theorem B below are inspired by \cite[Theorems~3.2 and~3.4]{Gil26},
i.e., the construction of Hovey triples by combining a cotorsion pair with an injective cotorsion pair (cf. Theorem \ref{theorem_merge}).

\vskip5pt

The general construction in Theorem A are concluded from examples in Subsection~3.7, and the examples are found as follows.
We first seek a thick subcategory $\mathcal W$ in $A\mbox{-mod}$ with $\mathcal P(A) \cup \mathcal I(A) \subsetneqq \mathcal W \subsetneqq A\mbox{-mod}$.
For small selfinjective Nakayama algebra $A$, easy candidates are the thick closure of $\mathcal P(A)$ together with some $X \notin \mathcal P(A)$.
Then we compute  the table of $\mathrm{Ext}^1_A(M,N)$ for all pairs of indecomposable modules,
where very few cotorsion pairs $(\mathcal C,\mathcal X)$ satisfied the required constraints, and this led to the examples.
The constructions of $\mathcal C$ and $\mathcal W$ subject to simple congruence conditions, and these are generalized to the main construction in Theorem A.

\subsection{\bf Exceptional \textup{Hovey} triples by Nakayama algebras with a unique maximal module} $\,$ \vskip5pt
 Next, we consider non-selfinjective Nakayama algebras without simple projective module, which admit a unique indecomposable maximal module.
These are precisely the Nakayama algebras $A$ having a unique zero relation.
Thus $A=kC_n/\langle a_t\cdots a_2a_1\rangle$ ($n\ge 2$). The Kupisch series is $(c_1, \ldots, c_n)$ with $c_i=t+n-i$ for $1\le i\le n$.
The indecomposable $A$-modules are $M_{i,l}$, where $i\in \mathbb Z/n\mathbb Z$ and $1 \leq l \leq c_i$.
Then $P_i=M_{i,c_i}$, $1\le i\le n$,   and  $I_i=M_{1,t+i}$ for $1\leq i\leq n-1$ and $I_n=M_{1,t}$.
The unique maximal module is the unique indecomposable projective-injective module $P_1 = I_{n-1} = M_{1,t+n-1}$.

\vskip5pt

\noindent {\bf Theorem B.} \label{introthm2} {\rm (Theorem \ref{thm_m_nak_classification}, Theorem \ref{prop_m_nak_injective_type})} \ {\it Let  $A=kC_n/\langle a_t\cdots a_2a_1\rangle \ (n\ge 2)$.
Then

\vskip5pt

$(1)$ \  $A\mbox{-}{\rm mod}$ admits an exceptional \textup{Hovey} triple if and only if $t=dn$ with  $d\geq2$.

\vskip5pt

$(2)$ \  Assume that $t=dn$ with $d\geq2$.   Then $(\mathcal C,\mathcal F,\mathcal W)$ is an exceptional \textup{Hovey} triple of \textup{Type I} in $A\mbox{-}{\rm mod}$, where
  {\small \begin{align*}
      \mathcal C
      = & \operatorname{add}(\bigoplus\limits_{1 \leq l \leq t}M_{n,l}\oplus\bigoplus_{\substack {1 \leq i \leq n-1 \\ 1\leq l \leq n-i-1}} M_{i,l}
      \oplus
      \bigoplus_{\substack {1 \leq i \leq n-1                                                                       \\ t+1-i\leq l\leq t+n-i}} M_{i,l}
      )                                                                                                             \\
      \mathcal F
      = & \operatorname{add}(
      I
      \oplus
      \bigoplus_{1\leq r\leq d-1}M_{1,rn}
      \oplus
      \bigoplus_{0\leq r\leq d-1}M_{n,rn+1}
      )                                                                                                             \\
      \mathcal W
      = & \operatorname{add}(
      \bigoplus_{1\le i\le n}
      (
      \bigoplus_{1\le l\le n-i}M_{i,l}
      \oplus
      \bigoplus_{t-i+1\le l\le t+n-i}M_{i,l}
      )).
    \end{align*}}}

\vskip5pt

In Theorem B,  $\mathcal W$ is the thick closure of $\mathcal P(A)$, $\mathcal I(A)$, and all but one simple module.
Motivated by the exact sequence $0\longrightarrow P_n
\longrightarrow P_1\oplus M_{n,t-n+1}
\longrightarrow I_n
\longrightarrow0$, where $P_1$ is the unique indecomposable projective-injective module, we construct $\mathcal X$ as the additive closure of $\mathcal I(A)$ together with $M_{n,t-n+1}$.

\vskip10pt

\subsection{Homotopy in additive categories} $\,$

\vskip5pt

For a category $\mathcal A$ equipped with a model structure $(\mathsf{Cofib}, \mathsf{Fib}, \mathsf{Weq})$, let $\mathcal C$,  $\mathcal F$ and $\mathcal W$ denote the class of cofibrant objects, the class of fibrant objects, and the class of trivial objects, respectively. Quillen introduced the left homotopy relation $\overset{l}{\sim}$ and the right homotopy relation $\overset{r}{\sim}$.
It is known that $\overset{l}{\sim}$ and $\overset{r}{\sim}$ coincide on $\mathcal C\cap\mathcal F$. Denote this equivalence relation ideal on $\mathcal C\cap\mathcal F$ by $\sim$ and the corresponding quotient category by $(\mathcal C\cap\mathcal F)/_{\sim}$.
Quillen's Fundamental Theorem of Model Structures claims that if a category $\mathcal A$ has a zero object, finite coproducts, finite products, and a model structure, then there is an equivalence
$(\mathcal{C}\cap \mathcal F)/_\sim\cong\mathsf{Ho}(\mathcal{A})$ as categories. For details see Theorem \ref{thm}.

\vskip5pt

The following result describes the homotopy relation $\sim$ for an arbitrary model structure on a weakly idempotent complete additive category. The proof is a direct application of the Factorization axiom of a model structure.
For exact model structures on weakly idempotent complete exact categories, this result is known in \cite[Proposition 4.4]{Gil11}.

\vskip5pt

\noindent  {\bf Theorem C} \label {introthm3} \ {\rm(Theorem \ref{htcatisaddq})}  \ {\it Let $(\mathsf{Cofib}, \mathsf{Fib}, \mathsf{Weq})$ be an arbitrary model structure on a weakly idempotent complete additive category  $\mathcal{A}$,
and $f, g : A \longrightarrow B$ be morphisms with $A,B \in \mathcal{C}\cap \mathcal{F}$. Then $f \sim g$ if and only if $f-g$ factors through an object in $\mathcal{C}\cap \mathcal{F}\cap \mathcal{W}$.}

\vskip5pt

As a consequence, one recovers the following description for the homotopy category of an arbitrary  model structure on a weakly idempotent complete additive category via the additive quotient  $\frac{\mathcal{C}\cap \mathcal{F}}{\mathcal{C}\cap \mathcal{F}\cap \mathcal{W}}$.
It has been proved in \cite[Theorem 1.1]{LZ}, but here it is a direct consequence of Theorem C.

\vskip5pt

\noindent  {\bf Corollary D}  \label {introthm4} \ {\rm(Corollary \ref{hcatonadditve})} \ {\it Let $(\mathsf{Cofib}, \mathsf{Fib}, \mathsf{Weq})$ be an arbitrary model structure on a weakly idempotent complete additive category  $\mathcal{A}$.
Then the homotopy category $\mathsf{Ho}(\mathcal{A})$ is an additive category, and $\mathsf{Ho}(\mathcal{A}) \cong \frac{\mathcal{C}\cap \mathcal{F}}{\mathcal{C}\cap \mathcal{F}\cap \mathcal{W}}$ as additive categories.}

\vskip5pt

Corollary D is known for exact model structures on weakly idempotent complete exact categories in Gillespie {\rm \cite[Proposition 4.4]{Gil11}};
for the $\omega$-model structures on weakly idempotent complete exact categories in A. Beligiannis and I. Reiten {\rm \cite[VIII, Theorem 4.2]{BR}}  and {\rm \cite[Theorem 1.1]{CLZ}};
and for exact model structures on triangulated categories in Nakaoka {\rm \cite[Proposition 6.10]{N}}.

\subsection{The homotopy category of exceptional model structures} \ Exceptional model structures can be characterized by their homotopy categories.

\vskip5pt

Let $(\mathcal C,\mathcal F,\mathcal W)$ be a Hovey triple in a weakly idempotent complete extriangulated category $(\mathcal A, \mathbb E, \mathfrak s)$.
Then one obtains two extriangulated structures in different ways,  on the same underlying category $\frac{\mathcal C\cap\mathcal F}{\mathcal C\cap\mathcal F\cap\mathcal W}$.
The first one is the induced extriangulated structure $(\frac{\mathcal C\cap\mathcal F}{\mathcal C\cap\mathcal F\cap\mathcal W}, \ \overline{\mathbb E}, \ \overline{\mathfrak s})$;
and on the other hand, by \cite[Theorem 6.20]{NP19},
$\mathsf{Ho}(\mathcal A)\cong \frac{\mathcal C\cap\mathcal F}{\mathcal C\cap\mathcal F\cap\mathcal W}$ is a triangulated category, and hence one gets another extriangulated structure
\[\xymatrix{(\frac{\mathcal C\cap\mathcal F}{\mathcal C\cap\mathcal F\cap\mathcal W}, \ \Hom_{\frac{\mathcal C\cap\mathcal F}{\mathcal C\cap\mathcal F\cap\mathcal W}}(-, \Sigma -), \ \mathfrak s_\triangle)}\]
determined by Nakaoka-Palu's triangulated structure $(\frac{\mathcal C\cap\mathcal F}{\mathcal C\cap\mathcal F\cap\mathcal W}, \ \Sigma, \ \triangle)$.
However, the two
extriangulated structures on $\frac{\mathcal C\cap\mathcal F}{\mathcal C\cap\mathcal F\cap\mathcal W}$ are not  necessarily {\it extriangle-equivalent} (see Definition  \ref{extrifunctor}, and Examples in Subsection 6.3); or equivalently,
the extriangulated structure $(\frac{\mathcal C\cap\mathcal F}{\mathcal C\cap\mathcal F\cap\mathcal W}, \ \overline{\mathbb E}, \ \overline{\mathfrak s})$ is not necessarily {\it canonically triangulated} in the sense of Definition \ref{cantri}.  The differences between the two extriangulated structures remain mysterious.
As shown in the following theorem, the notion of exceptional model structure gives an explanation of  the differences.

\vskip10pt

\noindent  {\bf Theorem E} \label{introthm5}  {\rm(Lemma \ref {lem_extriangulated_functor}, Theorem \ref{thm_Frobenius_pair}, Proposition \ref{prop_ghosts})} \ {\it Let $(\mathcal C,\mathcal F,\mathcal W)$ be a Hovey triple in a weakly idempotent complete extriangulated category $(\mathcal A,\mathbb E,\mathfrak s)$.
Then

\vskip5pt

$(1)$ \ There is an extriangle functor
\[\xymatrix{({\rm Id}_{\frac{\mathcal C\cap\mathcal F}{\mathcal C\cap\mathcal F\cap\mathcal W}}, \eta): (\frac{\mathcal C\cap\mathcal F}{\mathcal C\cap\mathcal F\cap\mathcal W}, \ \overline{\mathbb E}, \ \overline{\mathfrak s})\longrightarrow (\frac{\mathcal C\cap\mathcal F}{\mathcal C\cap\mathcal F\cap\mathcal W},  \ \Hom_{\frac{\mathcal C\cap\mathcal F}{\mathcal C\cap\mathcal F\cap\mathcal W}}(-, \Sigma -), \ \mathfrak s_\triangle)}\] such that each group homomorphism
$\eta_{Z,X}: \overline{\mathbb E}(Z,X)\longrightarrow \operatorname{Hom}_{\frac{\mathcal C\cap\mathcal F}{\mathcal C\cap\mathcal F\cap\mathcal W}}(Z,\Sigma X)
$ is injective.

\vskip5pt

Moreover, $\overline{\mathbb E}$ is a closed subbifunctor of $\Hom_{\frac{\mathcal C\cap\mathcal F}{\mathcal C\cap\mathcal F\cap\mathcal W}}(-, \Sigma -)$.

\vskip5pt

Moreover, if $\mathcal C\cap\mathcal F$ has enough projective objects, then for $X, Z\in\mathcal C\cap\mathcal F$, the image
$\eta_{Z,X}(\overline{\mathbb E}(Z,X))$ is precisely the set $\mathrm{Gh}_{\frac{\mathcal P(\mathcal C\cap\mathcal F)}{\mathcal C\cap\mathcal F\cap\mathcal W}}(Z, \Sigma X)$ of ghost maps
$($in the sense of {\rm \cite{Chr98} and \cite{Bel15}}$)$.

\vskip10pt

$(2)$ \ The following are equivalent$:$

\vskip5pt

\hskip17pt   $\textup{(i)}$ \ $(\mathcal C,\mathcal F,\mathcal W)$ is exceptional, i.e., $(\mathcal C\cap\mathcal F, \ \mathcal C\cap\mathcal F\cap\mathcal W)$ is not a Frobenius pair.

\vskip5pt

\hskip17pt  $\textup{(ii)}$ \ Extriangulated categories  $(\frac{\mathcal C\cap\mathcal F}{\mathcal C\cap\mathcal F\cap\mathcal W}, \ \overline{\mathbb E}, \ \overline{\mathfrak s})$ and $(\frac{\mathcal C\cap\mathcal F}{\mathcal C\cap\mathcal F\cap\mathcal W},  \ \Hom_{\frac{\mathcal C\cap\mathcal F}{\mathcal C\cap\mathcal F\cap\mathcal W}}(-, \Sigma -), \ \mathfrak s_\triangle)$  are not extriangle-equivalent.

\vskip5pt

\hskip17pt  $\textup{(iii)}$ \ Functors  $\overline{\mathbb E}$ and $\Hom_{\frac{\mathcal C\cap\mathcal F}{\mathcal C\cap\mathcal F\cap\mathcal W}}(-, \Sigma -)$  are
not naturally isomorphic, or equivalently, $\overline{\mathbb E}$ is a proper closed subbifunctor of $\Hom_{\frac{\mathcal C\cap\mathcal F}{\mathcal C\cap\mathcal F\cap\mathcal W}}(-, \Sigma -)$.

\vskip5pt

\hskip17pt $\textup{(iv)}$ \  $\bigl(\frac{\mathcal C\cap\mathcal F}{\mathcal C\cap\mathcal F\cap\mathcal W}, \ \overline{\mathbb E}, \ \overline{\mathfrak s}\bigr)$ is not canonically triangulated.}

\vskip10pt

\subsection{Extriangulated structures arising from an exceptional model structure} $\,$

\vskip5pt

Let $\mathcal A$ be an essentially small additive category.
Denote by $\mathrm{fp}(\mathcal A^{\mathrm{op}}, \mathrm{Ab})$ the category of finitely presented contravariant additive functors from $\mathcal A$ to $\mathrm{Ab}$.

\vskip5pt

Let $(\mathcal C, \mathcal F,\mathcal W)$ be a Hovey triple in a weakly idempotent complete extriangulated category. As seen  from Theorem E,
if this Hovey triple is exceptional, then $\overline{\mathbb E}$ is a proper closed subbifunctor of $\Hom_{\frac{\mathcal C\cap\mathcal F}{\mathcal C\cap\mathcal F\cap\mathcal W}}(-, \Sigma -)$.
Using Y. Ogawa's work \cite{Oga21} on Auslander's defect theory in extriangulated categories and H. Enomoto's work  \cite{Eno21} on closed additive subbifunctors,
the closed additive subbifunctors $\mathbb F$ of $\operatorname{Hom}_{\frac{\mathcal C\cap\mathcal F}{\mathcal C\cap\mathcal F\cap\mathcal W}}(-, \Sigma-)$ with
$\overline{\mathbb E}\subseteq\mathbb F$
are in one-to-one correspondence with Serre subcategories of  $\mathrm{fp}((\frac{\mathcal P(\mathcal C\cap\mathcal F)}{\mathcal C\cap\mathcal F\cap\mathcal W})^{\mathrm{op}}, \mathrm{Ab})$, in case $\mathcal C\cap\mathcal F$ is essentially small and  has enough projective objects.

\vskip5pt

\noindent  {\bf Theorem F} \label{introthm6}  {\rm(Theorem \ref{thm_enomoto_ogawa})} \ {\it \ Let $(\mathcal C,\mathcal F,\mathcal W)$ be a  Hovey triple in a weakly idempotent complete extriangulated category $(\mathcal A,\mathbb E,\mathfrak s)$.
Assume that $\mathcal C\cap\mathcal F$ is essentially small.

\vskip5pt

$(1)$  \ If $\mathcal C\cap\mathcal F$ has enough projective objects, then the map $\mathbb F\mapsto {\rm def}\mathbb F$ gives an isomorphism of posets$:$

\[\xymatrix{\left\{\mathbb F \ \middle| \ \overline{\mathbb E}\subseteq\mathbb F,
  \ \mathbb F \ \textup{is a closed additive subbifunctor of} \operatorname{Hom}_{\frac{\mathcal C\cap\mathcal F}{\mathcal C\cap\mathcal F\cap\mathcal W}}(-, \Sigma-)\right\}\ar[d] \\
  \left\{\textup{Serre subcategories of} \ \mathrm{fp}((\frac{\mathcal P(\mathcal C\cap\mathcal F)}{\mathcal C\cap\mathcal F\cap\mathcal W})^{\mathrm{op}}, \mathrm{Ab}). \right\}}\]

\vskip5pt

$(2)$ \ If $\mathcal C\cap\mathcal F$ has enough projective objects and enough injective objects, then there is an isomorphism of posets
from $\{\textup{Serre subcategories of }
  \ \mathrm{fp}((\frac{\mathcal P(\mathcal C\cap\mathcal F)}{\mathcal C\cap\mathcal F\cap\mathcal W})^{\mathrm{op}}, \mathrm{Ab})\}$ to $\{
  \textup{Serre subcategories of }
  \ \mathrm{fp}(\frac{\mathcal I(\mathcal C\cap\mathcal F)}{\mathcal C\cap\mathcal F\cap\mathcal W}, \mathrm{Ab})\}$.}

  \vskip10pt

If $(\mathcal C,\mathcal F,\mathcal W)$ is exceptional, then
$\overline{\mathbb E} \subsetneqq \operatorname{Hom}_{\frac{\mathcal C\cap\mathcal F}{\mathcal C\cap\mathcal F\cap\mathcal W}}(-, \Sigma-)$, and Theorem F indeed  gives intermediate  extriangulated structures $(\frac{\mathcal C\cap\mathcal F}{\mathcal C\cap\mathcal F\cap\mathcal W}, \ \mathbb F, \ (\mathfrak s_\triangle)|_\mathbb F)$.

\vskip5pt  As an application of Theorem F and Proposition \ref{prop_D_enough_PI} one gets

  \vskip5pt

  \noindent  {\bf Corollary G} \ {\rm(Corollary  \ref{number})} \ {\it Let $A$ be a representation-finite finite-dimensional algebra over a field $k$, and $(\mathcal C,\mathcal F,\mathcal W)$ a {\rm Hovey} triple in $A\text{-}\mathrm{mod}$. Then
$\mathcal C\cap\mathcal F$ has enough projective and enough injective objects, with $\left|\operatorname{Ind}\mathcal P(\mathcal C\cap\mathcal F)\right| = \left|\operatorname{Ind}\mathcal I(\mathcal C\cap\mathcal F)\right|;$ there are $2^n$ closed $k$-linear subbifunctors $\mathbb F$ satisfying
$\overline{\mathbb E}\subseteq\mathbb F\subseteq
\Hom_{\frac{\mathcal C\cap\mathcal F}{\mathcal C\cap\mathcal F\cap\mathcal W}}(-, \Sigma -),$
  and there are $2^n$ Serre subcategories of \ $\mathrm{fp}((\frac{\mathcal P(\mathcal C\cap\mathcal F)}{\mathcal C\cap\mathcal F\cap\mathcal W})^{\mathrm{op}}, k\mbox{-}{\rm mod})$, where $n = \left|
\operatorname{Ind}\mathcal P(\mathcal C\cap\mathcal F) \right| - \left|
\operatorname{Ind}(\mathcal C\cap\mathcal F\cap\mathcal W)
\right|$.}

\vskip5pt

\subsection{Organisation} Preliminaries are recalled in $\S$2, in particular, all results are stated as general as possible in extriangulated categories. Proposition \ref{prop_ctp_1} (which claims that any cotorsion pair in  $A\mbox{-}{\rm mod}$ is complete, where $A$ is a representation-finite algebra)  and
Proposition \ref{prop_D_enough_PI} (which claims that for any \textup{Hovey} triple $(\mathcal{C} , \mathcal{F}, \mathcal{W})$ in  $A\text{-}\mathrm{mod}$,
$\mathcal{C} \cap \mathcal{F}$ has enough projective objects and enough injective objects, where $A$ is a representation-finite algebra) seem to be new; and Proposition \ref{prop_Gillespie} and  Example \ref{nonherednonexcep} give Hovey triples which are neither hereditary nor exceptional.

\vskip5pt

Using the Auslander-Reiten quiver of Nakayama algebras and Gillespie's  new construction  of  Hovey triples (Theorem \ref{theorem_merge}),
Theorems A and B will be proved in $\S 3$ and $\S 4$, respectively; and some technical lemmas in the proofs are put in $\S8$, as an appendix.

\vskip5pt

Theorem C and Corollary D are proved in $\S$5;  Theorem E is proved in $\S$6; and Theorem F  and Corollary G are proved in $\S$7.

\section{\bf Preliminaries}\label{sec_Pre}

\subsection{Additive quotient} \ Let $\mathcal{A}$ be an additive category, and $\mathcal{B}$ a full additive subcategory of $\mathcal{A}$.
The \emph{additive quotient} $\mathcal{A}/\mathcal{B}$ is an additive category with the same objects as $\mathcal{A}$, and for any $X, Y \in \mathcal{A}$, $\mathrm{Hom}_{\mathcal{A} / \mathcal{B}}(X, Y)$ is quotient group 
$\mathrm{Hom}_{\mathcal{A}} (X, Y) / \{f : X \longrightarrow Y \mid \text{$f$ factors through some $B \in \mathcal B$}\}.$

\vskip5pt

Let $A$ be a finite-dimensional algebra.
We denote by $A\mbox{-}{\rm mod}$ the category of finitely generated left $A$-modules, and by $\mathcal{P}(A)$ and $\mathcal{I}(A)$ the full subcategories of projective modules and injective modules, respectively.
Then one has the additive quotients $A\mbox{-}\underline{\rm mod} = A\mbox{-}{\rm mod} / \mathcal{P}(A)$ and 
$A\mbox{-}\overline{\rm mod} = A\mbox{-}{\rm mod} / \mathcal{I}(A)$, and one writes $\underline{\mathrm{Hom}}_A = \mathrm{Hom}_{A\mbox{-}\underline{\rm mod}}$ and $\overline{\mathrm{Hom}}_A = \mathrm{Hom}_{A\mbox{-}\overline{\rm mod}}$.

\subsection{Extriangulated categories} \ The notion of an extriangulated category, introduced by Nakaoka and Palu \cite{NP19}, is a common generalization of exact categories and triangulated categories.
Let $\mathcal{A}$ be an additive category, and $\mathbb E: \mathcal{A}^{\mathrm{op}} \times \mathcal{A} \longrightarrow {\rm Ab}$ an additive bifunctor.
For $f: X\longrightarrow X'$, let $f^\ast = \mathbb E(f, -): \mathbb E(X', -)\longrightarrow \mathbb E(X, -)$ and $f_\ast = \mathbb E(-, f): \mathbb E(-, X)\longrightarrow \mathbb E(-, X')$
be the natural transformations. An element $\delta\in \mathbb{E}(Z, X)$
is called {\it an $\mathbb{E}$-extension}, and $0\in\mathbb{E}(Z, X)$ {\it the split $\mathbb{E}$-extension}.
Since $\mathbb E$ is a bifunctor, for  $h : Z' \longrightarrow Z$ and  $\delta \in \mathbb E(Z, X)$, one has
\[
  (f_\ast)_{Z'} (h^\ast)_{X} (\delta) = (h^\ast)_{X'} (f_\ast)_{Z} (\delta)\in \mathbb E(Z', X').
\]

For $\mathbb{E}$-extensions $\delta \in \mathbb E(Z,X)$ and $\delta' \in \mathbb E(Z', X')$, {\it a morphism of $\mathbb E$-extensions}
$\delta \longrightarrow \delta'$ is a pair of morphisms $(\alpha, \gamma)$ with $\alpha: X\longrightarrow X'$ and $\gamma: Z\longrightarrow Z'$, such that $\alpha_\ast \delta = \gamma^\ast \delta'\in \mathbb E(Z, X')$.
Denote by $\delta \oplus \delta' \in \mathbb E(Z \oplus Z', X \oplus X')$ the image of $(\delta, \delta') \in \mathbb E(Z,X) \oplus \mathbb E(Z',X')$ under the natural inclusion
\[
  \mathbb E(Z,X) \oplus \mathbb E(Z',X') \hookrightarrow  \mathbb E(Z,X) \oplus \mathbb E(Z',X')\oplus \mathbb E(Z',X) \oplus \mathbb E(Z,X') \cong  \mathbb E(Z \oplus Z', X \oplus X').
\]
In particular, if $X=X'$ and $Z=Z'$, then
\[
  \left(\begin{smallmatrix}\mathrm{Id}_Z\\ \mathrm{Id}_Z\end{smallmatrix}\right)^\ast
  \left(\begin{smallmatrix}\mathrm{Id}_X,&\mathrm{Id}_X\end{smallmatrix}\right)_\ast
  (\delta \oplus \delta')
  =
  \left(\begin{smallmatrix}\mathrm{Id}_X,&\mathrm{Id}_X\end{smallmatrix}\right)
  \left(\begin{smallmatrix}\delta&0\\0&\delta'\end{smallmatrix}\right)
  \left(\begin{smallmatrix}\mathrm{Id}_Z\\ \mathrm{Id}_Z\end{smallmatrix}\right)
  = \delta + \delta' =
  \left(\begin{smallmatrix}\mathrm{Id}_X,&\mathrm{Id}_X\end{smallmatrix}\right)_\ast \left(\begin{smallmatrix}\mathrm{Id}_Z\\ \mathrm{Id}_Z\end{smallmatrix}\right)^\ast
  (\delta \oplus \delta').
\]

By definition, two sequences $X \xlongrightarrow{f} Y \xlongrightarrow{g} Z$ and $X \xlongrightarrow{f'} Y' \xlongrightarrow{g'} Z$ are \emph{equivalent} if there exists an isomorphism $\varphi: Y \longrightarrow Y'$ such that the following diagram commutes:
\begin{equation*}
  \xymatrix@R=0.4cm{
  X\ar[r]^{f}\ar@{=}[d] & Y\ar[r]^{g}\ar[d]^{{\varphi }} & Z\ar@{=}[d] \\
  X\ar[r]^{{f'}} & Y'\ar[r]^{{g'}} & Z
  }
\end{equation*}

A {\it realization} $\mathfrak s$ of $\mathbb E$ is a collection of ``mappings'', sending each $\delta \in \mathbb E(Z,X)$  to an equivalence class of sequences of the form $[X \xlongrightarrow{f} Y \xlongrightarrow{g} Z]$, satisfying that for any morphism of $\mathbb E$-extensions $(\alpha, \gamma): \delta \longrightarrow \delta'$ and any representative $X \xlongrightarrow{f} Y \xlongrightarrow{g} Z$ of $\mathfrak s(\delta)$ and $X' \xlongrightarrow{f'} Y' \xlongrightarrow{g'} Z'$ of $\mathfrak s(\delta')$, there exists $\beta: Y \longrightarrow Y'$ such that the following diagram commutes:
\begin{equation*}
  \xymatrix@R=0.4cm{
  X\ar[r]^{f}\ar[d]_-{\alpha} & Y\ar[r]^{g}\ar@{..>}[d]^-{\beta} & Z\ar[d]^-{\gamma} \\
  X'\ar[r]^{{f'}} & Y'\ar[r]^{{g'}} & Z'
  }
\end{equation*}
In this case, one says that an $\mathbb E$-extension $\delta$ is {\it realized} by $X \xlongrightarrow f Y \xlongrightarrow g Z$, if
$\mathfrak s (\delta) = \left[X \xlongrightarrow f Y \xlongrightarrow g Z\right].$

\vskip5pt

A realization $\mathfrak s$ of $\mathbb E$ is {\it additive}, if any split $\mathbb E$-extension $0\in \mathbb E(Y,X)$ is realized by
$X \xlongrightarrow{\binom{\mathrm{Id}_X}{0}} X \oplus Y \xlongrightarrow{(0,\ \mathrm{Id}_Y)} Y$; and if
$\mathfrak s(\delta_i)=[X_i \xlongrightarrow {f_i}Y_i\xlongrightarrow {g_i}Z_i]$, $i =1, 2$, then $\delta_1 \oplus \delta_2$ is realized by
\begin{equation*}			X_1 \oplus X_2 \xlongrightarrow{\left(\begin{smallmatrix}f_1&0\\0&f_2\end{smallmatrix}\right)} Y_1 \oplus Y_2 \xlongrightarrow{\left(\begin{smallmatrix}g_1&0\\0&g_2\end{smallmatrix}\right)} Z_1 \oplus Z_2.
\end{equation*}

\vskip5pt

In the following, if an $\mathbb E$-extension $\delta$ is realized by $X \xlongrightarrow{f} Y \xlongrightarrow{g} Z$,
then it will be denoted by $X \xlongrightarrow{f} Y \xlongrightarrow{g} Z \xdashrightarrow{\delta}{}$, and we call it an \textit{$\mathbb E$-triangle}, $f$ an \textit{$\mathbb E$-inflation} and $g$ an \textit{$\mathbb E$-deflation}.
{\it A morphism} of $\mathbb E$-triangles is a triple $(\alpha, \beta, \gamma)$ such that $\alpha_\ast \delta = \gamma ^\ast \delta '$ and the following diagram commutes:
\begin{equation*}\label{morphism}
  \xymatrix@R=0.4cm{
  X\ar[r]^{f}\ar[d]_-{\alpha} & Y\ar[r]^{g}\ar[d]^-{\beta} & Z\ar@{-->}[r]^{\delta}\ar[d]^-{\gamma} & {} \\
  X'\ar[r]^{{f'}} & Y'\ar[r]^{{g'}} & Z'\ar@{-->}[r]^{{\delta'}} & {}
  }
\end{equation*}

\vskip5pt

\begin{definition} \ {\rm (\cite[Definition 2.12]{NP19})} \ An {\it extriangulated category} is a triplet $(\mathcal{A}, \mathbb E, \mathfrak s)$, satisfying the following axioms:

  \vskip5pt

  {\bf ET1.} \ $\mathcal{A}$ is an additive category, and $\mathbb E: \mathcal{A}^{\mathrm{op}}\times \mathcal{A} \longrightarrow {\rm Ab}$ is an additive bifunctor.

  \vskip5pt

  {\bf ET2.} \ $\mathfrak s$ is an additive realization of $\mathbb E$.

  \vskip5pt

  {\bf ET3.} \ Let $X \xlongrightarrow{f} Y \xlongrightarrow{g} Z \xdashrightarrow{\delta}{}$ and $X' \xlongrightarrow{f'} Y' \xlongrightarrow{g'} Z' \xdashrightarrow{\delta'}{}$ be $\mathbb E$-triangles. If there are $\alpha : X \longrightarrow X'$ and $\beta : Y \longrightarrow Y'$ such that $\beta \circ f = f' \circ \alpha$,
  then there exists $\gamma: Z\longrightarrow Z'$ such that $(\alpha, \beta, \gamma)$ is a morphism of $\mathbb E$-triangles.

  \vskip5pt

  {\bf ET3$^{\mathrm{op}}$.} \ Let $X \xlongrightarrow{f} Y \xlongrightarrow{g} Z \xdashrightarrow{\delta}{}$ and $X' \xlongrightarrow{f'} Y' \xlongrightarrow{g'} Z' \xdashrightarrow{\delta'}{}$ be $\mathbb E$-triangles. If there are $\beta : Y \longrightarrow Y'$ and $\gamma : Z \longrightarrow Z'$ such that $\gamma \circ g = g' \circ \beta$,
  then there exists $\alpha:X\longrightarrow X'$ such that $(\alpha, \beta, \gamma)$ is a morphism of $\mathbb E$-triangles.

  \vskip5pt

  {\bf ET4.} \ Let $A \xlongrightarrow{f} B \xlongrightarrow{g} D \xdashrightarrow{\delta}{}$ and $B \xlongrightarrow{u} C \xlongrightarrow{v} E \xdashrightarrow{\varepsilon}{}$ be $\mathbb E$-triangles. Then there exists a commutative diagram
  \begin{equation*}
    \xymatrix@R=0.4cm{
    A\ar[r]^{f}\ar@{=}[d] & B\ar[r]^{g}\ar[d]^{u} & D\ar@{-->}[r]^{\delta}\ar@{..>}[d]^{w} & {} \\
    A\ar@{..>}[r]^{m} & C\ar@{..>}[r]^{h}\ar[d]^{v} & F\ar@{..>}[r]^{\theta}\ar@{..>}[d]^{q} & {} \\
    {} & E\ar@{=}[r]\ar@{-->}[d]^{{\varepsilon }} & E\ar@{..>}[d]^{\eta} & {} \\
    {} & {} & {} & {}
    }
  \end{equation*}
  such that $A \xlongrightarrow{m} C \xlongrightarrow{h} F \xdashrightarrow{\theta}{}$ and $D \xlongrightarrow{w} F \xlongrightarrow{q} E \xdashrightarrow{\eta}{}$ are $\mathbb E$-triangles and $w^\ast \theta=\delta$, $g_\ast \varepsilon=\eta$ and $f_\ast \theta= q^\ast\varepsilon$. In particular, $\mathbb E$-inflations are closed under compositions.

  \vskip5pt

  {\bf ET4$^{\mathrm{op}}$.} \ Let $A \xlongrightarrow{m} C \xlongrightarrow{h} F \xdashrightarrow{\theta}{}$ and $D \xlongrightarrow{w} F \xlongrightarrow{q} E \xdashrightarrow{\eta}{}$ be $\mathbb E$-triangles. Then there exists a commutative diagram
  \begin{equation*}
    \xymatrix@R=0.4cm{
    A\ar@{..>}[r]^{f}\ar@{=}[d] & B\ar@{..>}[r]^{g}\ar@{..>}[d]^{u} & D\ar@{..>}[r]^{\delta}\ar[d]^{w} & {} \\
    A\ar[r]^{m} & C\ar[r]^{h}\ar@{..>}[d]^{v} & F\ar@{-->}[r]^{\theta}\ar[d]^{q} & {} \\
    {} & E\ar@{=}[r]\ar@{..>}[d]^{{\varepsilon }} & E\ar@{-->}[d]^{\eta} & {} \\
    {} & {} & {} & {}
    }
  \end{equation*}
  such that $A \xlongrightarrow{f} B \xlongrightarrow{g} D \xdashrightarrow{\delta}{}$ and $B \xlongrightarrow{u} C \xlongrightarrow{v} E \xdashrightarrow{\varepsilon}{}$ are $\mathbb E$-triangles and $w^\ast \theta=\delta$, $g_\ast \varepsilon=\eta$ and $f_\ast \theta= q^\ast\varepsilon$. In particular, $\mathbb E$-deflations are closed under compositions.
\end{definition}

\vskip5pt

For an $\mathbb E$-triangle $X \xlongrightarrow{f} Y \xlongrightarrow{g} Z \xdashrightarrow{\delta},$
write $\mathrm{Cone}(f) = Z$ and $\mathrm{CoCone}(g) = X$. By \cite[Remark 3.10]{NP19},  $\mathrm{Cone}(f)$ and $\mathrm{CoCone}(g)$ are unique up to isomorphism.

\vskip5pt

Exact categories and triangulated categories are important classes of extriangulated categories (\cite[Example 2.13, Proposition 3.22]{NP19}). For an exact category $\mathcal{A}$, the bifunctor $\mathbb E$ is the Yoneda extension bifunctor $\operatorname{Ext}^1_{\mathcal{A}}$, and $\mathfrak s$ assigns to an extension class the corresponding equivalence class of conflations.
For a triangulated category $(\mathcal{T},[1],\triangle)$, one takes
$\mathbb E(Z,X)=\operatorname{Hom}_{\mathcal{T}}(Z,X[1])$, with $\mathfrak s(\delta)$ represented by a distinguished triangle
$X\longrightarrow Y\longrightarrow Z\xlongrightarrow{\delta}X[1].$
There are many extriangulated categories which are neither exact categories nor triangulated categories (see \cite[Proposition 3.30, Remark 2.18]{NP19}).

\vskip5pt

An object $P$ of an extriangulated category $(\mathcal{A},\mathbb E,\mathfrak s)$ is \emph{projective} if $\mathrm{Hom}_{\mathcal A}(P,-)$ sends every $\mathbb E$-deflation to a surjection, or equivalently, $\mathbb E(P,-)=0$ (cf. \cite[Propositions 3.23 and 3.24]{NP19}).
An object $I$ is \emph{injective} if $\mathrm{Hom}_{\mathcal A}(-,I)$ sends every $\mathbb E$-inflation to a surjection; equivalently, $\mathbb E(-,I)=0$.
Denote by $\mathcal P(\mathcal A)$ and $\mathcal I(\mathcal A)$ respectively the full subcategories of projective objects and injective objects. If $\mathcal A$ is the category $A$-mod of finitely generated left $A$-modules, where $A$ is a finite-dimensional algebra, then $\mathcal P(\mathcal A)$ and $\mathcal I(\mathcal A)$ are written as $\mathcal P(A)$ and $\mathcal I(A)$, respectively.

\vskip5pt

The category $\mathcal A$ has \emph{enough projective objects} if every object $X$ occurs in an $\mathbb E$-triangle $K \longrightarrow P\longrightarrow X\dashrightarrow$ with $P\in\mathcal P(\mathcal A)$,
and it has \emph{enough injective objects} if every $X$ occurs in an $\mathbb E$-triangle $X\longrightarrow I\longrightarrow C \dashrightarrow$ with $I\in\mathcal I(\mathcal A)$.
Let $\mathcal P(X,Y)$ (respectively, $\mathcal I(X,Y)$) be the subgroups of $\mathrm{Hom}_{\mathcal A}(X,Y)$ consisting of morphisms factoring through a projective (respectively, an injective) object, and put
\[
  \underline{\operatorname{Hom}}_{\mathcal A}(X,Y)
  =\mathrm{Hom}_{\mathcal A}(X,Y)/\mathcal P(X,Y),
  \qquad
  \overline{\operatorname{Hom}}_{\mathcal A}(X,Y)
  =\mathrm{Hom}_{\mathcal A}(X,Y)/\mathcal I(X,Y).
\]
Then one has the additive quotients $\mathcal A/\mathcal P(\mathcal A)$ and $\mathcal A/\mathcal I(\mathcal A)$.

\vskip5pt

\begin{definition} \ {\rm(\cite[Definition 7.1]{NP19})} \label{frobextri} \  An extriangulated category $(\mathcal A, \mathbb E, \mathfrak s)$ is \emph{Frobenius},
  if it has enough projective objects and enough injective objects, and $\mathcal P(\mathcal A)=\mathcal I(\mathcal A)$.
\end{definition}

\vskip5pt

If $\mathcal A$ is a Frobenius extriangulated category, then the \emph{stable category} $\mathcal A/\mathcal P(\mathcal A)$ of $\mathcal A$ is a triangulated category (cf. \cite[7.5]{NP19} and \cite[3.3]{Msa25}).

\vskip5pt

\subsection{Cotorsion pairs in extriangulated categories} \ Cotorsion pairs were introduced by Salce \cite{Sal79}, discussed systematically by Enochs and Jenda \cite{EJ00}, and extended to extriangulated categories by Nakaoka and Palu \cite[Section~4]{NP19}.
Let $(\mathcal{A},\mathbb E,\mathfrak s)$ be an extriangulated category.
Write $X\perp Y$ for $\mathbb E(X,Y)=0$.
For a class of objects $\mathcal{L}$, write $\mathcal{L}^\perp$ (respectively, ${}^\perp\mathcal{L}$) for the class of objects $Z$ such that $L\perp Z$ (respectively, $Z\perp L$) for all $L\in\mathcal{L}$.
A \emph{cotorsion pair} $(\mathcal{C},\mathcal{X})$ in $(\mathcal{A},\mathbb E,\mathfrak s)$ is a pair of classes of objects such that $\mathcal{C}^\perp=\mathcal{X}$ and ${}^\perp\mathcal{X}=\mathcal{C}$.
A cotorsion pair is \emph{complete} if for any object $M\in\mathcal{A}$, there are $\mathbb E$-triangles
\begin{equation*}
  X^M \longrightarrow C^M \longrightarrow M \dashrightarrow,
  \qquad
  M \longrightarrow X_M \longrightarrow C_M \dashrightarrow
\end{equation*}
with $C^M,C_M\in\mathcal{C}$ and $X^M,X_M\in\mathcal{X}$.

\vskip5pt

\begin{lemma}\label{lem_ctp_1} \ {\rm(\cite {EJ00})} \ Let $(\mathcal{A},\mathbb E,\mathfrak s)$ be an extriangulated category with enough projective and injective objects.
  Then the following are equivalent for a cotorsion pair $(\mathcal{C},\mathcal{X})$ in $(\mathcal{A},\mathbb E,\mathfrak s)$.
  \vskip5pt
  $(1)$ \ $(\mathcal{C},\mathcal{X})$ is complete$;$
  \vskip5pt
  $(2)$ \ For any $M\in\mathcal{A}$, there is an $\mathbb E$-triangle
  $X^M \longrightarrow C^M \longrightarrow M \dashrightarrow$
  with $C^M\in\mathcal{C}$ and $X^M\in\mathcal{X};$
  \vskip5pt
  $(2')$ \ For any $M\in\mathcal{A}$, there is an $\mathbb E$-triangle
  $M \longrightarrow X_M \longrightarrow C_M \dashrightarrow$
  with $C_M\in\mathcal{C}$ and $X_M\in\mathcal{X}.$
\end{lemma}

\begin{proof} \ This is well-known for the module category of a ring (see \cite [Proposition 7.1.7]{EJ00}).

  \vskip5pt We only justify $(2)\Longrightarrow (2')$. For $M\in\mathcal{A}$, there is an $\mathbb E$-triangle $M\longrightarrow I\longrightarrow N\xdashrightarrow{\xi}$ with $I$ injective.
  By $(2)$, there is an $\mathbb E$-triangle $X^N\longrightarrow C^N\longrightarrow N\xdashrightarrow{\delta}$ with $C^N\in\mathcal C$ and $X^N\in\mathcal X$.
  By \cite[Proposition 3.15]{NP19} there is a  commutative diagram of $\mathbb E$-triangles:
  \begin{equation*}
    \xymatrix@R=15pt{
    {} & X^N \ar@{..>}[d] & X^N \ar@{=}[l] \ar[d] & {} \\
    M \ar@{=}[d] \ar@{..>}[r] & X_M \ar@{..>}[r] \ar@{..>}[d] & C^N \ar[d] \ar@{-->}[r]^{\lambda} & {} \\
    M \ar[r] & I \ar[r] \ar@{-->}[d]_{\nu} & N \ar@{-->}[r]^{\xi} \ar@{-->}[d]^{\delta} & {} \\
    {} & {} & {} & {}
    }
  \end{equation*}
  The middle column gives $X_M\in\mathcal X$, and hence the middle row is the required $\mathbb E$-triangle.
\end{proof}

\vskip5pt

A finite-dimensional algebra $A$ is {\it representation-finite} if the category $A\mbox{-}{\rm mod}$  of finitely generated left $A$-modules contains only finitely many isomorphism classes of indecomposable modules.

\begin{proposition}\label{prop_ctp_1} \ Let $A$ be a representation-finite algebra, and  $(\mathcal{C}, \mathcal{X})$ a cotorsion pair in $A\mbox{-}{\rm mod}$.
  Then $(\mathcal{C}, \mathcal{X})$ is always  complete.
\end{proposition}
\begin{proof} \ Since $A$ is representation-finite, $\mathcal{C} = \operatorname{add}N$ for some $N\in A$-mod. Thus every $A$-module $M$ admits a right $\mathcal{C}$-approximation, which is necessarily surjective, and hence a right minimal $\mathcal{C}$-approximation (see, e.g., \cite[Proposition 1.1]{AR91}). By Wakamatsu's Lemma (cf. \cite[Proposition 1]{Wak90}; see also \cite [Lemma 1.3]{AR91}, or \cite[Corollary 7.2.3]{EJ00}), the kernel of a right minimal $\mathcal{C}$-approximation is in $\mathcal{X}$.
  By \Cref{lem_ctp_1},  $(\mathcal{C}, \mathcal{X})$ is complete.
\end{proof}

A cotorsion pair $(\mathcal{C},\mathcal{X})$ is \emph{hereditary} if, for every $\mathbb E$-triangle $A\longrightarrow B\longrightarrow C\dashrightarrow$, one has $A\in\mathcal C$ whenever $B,C\in\mathcal C$, and $C\in\mathcal X$ whenever $A,B\in\mathcal X$.

\vskip5pt

\begin{lemma}\label{lem_ctp_2} \label{cor_ctp_2} \ Let $(\mathcal{C},\mathcal{X})$ be a complete cotorsion pair in an extriangulated category $(\mathcal A,\mathbb E,\mathfrak s)$.

  \vskip5pt

  $(1)$ \ {\rm (\cite{Gil25}, \cite{S}, \cite{HZZZ})} \ The following are equivalent$:$

  \hskip15pt $\textup{(i)}$ \ $(\mathcal{C},\mathcal{X})$ is hereditary$;$

  \hskip15pt
  $\textup{(ii)}$ \ For every $\mathbb E$-triangle $A\longrightarrow B\longrightarrow C\dashrightarrow$ with $B,C\in\mathcal C$, one has $A\in\mathcal C$.

  \hskip15pt
  $\textup{(iii)}$ \ For every $\mathbb E$-triangle $A\longrightarrow B\longrightarrow C\dashrightarrow$ with $A,B\in\mathcal X$, one has $C\in\mathcal X;$

  \vskip5pt
  $(2)$ \  {\rm(\cite{GR99})} \  If $\mathcal{A}$ has enough projective objects, then the conditions in $(1)$ are also equivalent to the following$:$

  \hskip15pt
  $\textup{(i)}$ \ For every $\mathbb E$-triangle $M\longrightarrow P\longrightarrow C\dashrightarrow$ with $C\in\mathcal C$ and $P$ projective, one has $M\in\mathcal C;$

  \hskip15pt
  $\textup{(ii)}$ \ For every $C\in\mathcal C$, there is an $\mathbb E$-triangle $M\longrightarrow P\longrightarrow C\dashrightarrow$ with $P$ projective and $M\in\mathcal C.$

  \vskip5pt
  $(2')$ \ {\rm (\cite{GR99})} \ If $\mathcal{A}$ has enough injective objects, then the conditions in $(1)$ are also equivalent to the following$:$

  \hskip15pt
  $\textup{(i)}$ \ For every $\mathbb E$-triangle $X\longrightarrow I\longrightarrow M\dashrightarrow$ with $X\in\mathcal X$ and $I$ injective, one has $M\in\mathcal X;$

  \hskip15pt
  $\textup{(ii)}$ \ For every $X\in\mathcal X$, there is an $\mathbb E$-triangle $X\longrightarrow I\longrightarrow M\dashrightarrow$ with $I$ injective and $M\in\mathcal X.$
\end{lemma}
\begin{proof} \ (1) \ This is well-known in \cite[Theorem 2.16]{Gil25} for exact categories, see also \cite[Lemma 6.17]{S}; and in \cite[Proposition 2.18]{HZZZ} for extriangulated categories.

  \vskip5pt

  (2) \  This is well-known for the module category of a ring, see \cite[Theorem 2.10]{GR99}.

  \vskip5pt

  Assume that $\mathcal{A}$ has enough projective objects.
  The implications $\textup{(1-ii)} \Longrightarrow \textup{(i)} \Longrightarrow \textup{(ii)}$ are clear.
  It remains to show the implication $\textup{(ii)} \Longrightarrow \textup{(1-ii)}$. Consider any $\mathbb E$-triangle
  $A \longrightarrow B \longrightarrow C \xdashrightarrow{\delta}$
  with $B,C\in\mathcal C$.
  By $\textup{(ii)}$ there exists an $\mathbb E$-triangle
  $M\longrightarrow P\longrightarrow C\xdashrightarrow{\xi}$
  with $P$ projective and $M\in\mathcal C$.
  Applying \cite[Proposition 3.15]{NP19} one obtains a commutative diagram of $\mathbb E$-triangles:
  \begin{equation*}
    \xymatrix@R=15pt{
    {} & A \ar@{=}[r] \ar[d] & A \ar[d] & {} \\
    M \ar@{=}[d] \ar[r] & D \ar[r] \ar[d] & B \ar[d] \ar@{-->}[r]^{\zeta} & {} \\
    M \ar[r] & P \ar[r] \ar@{-->}[d]_{\lambda} & C \ar@{-->}[r]^{\xi} \ar@{-->}[d]^{\delta} & {} \\
    {} & {} & {} & {}
    }
  \end{equation*}
  Since $M, B\in\mathcal C$, it follows that $D\in\mathcal C$.
  Since $P$ is projective, one has $\lambda=0$. Thus the middle column splits and hence $A\in\mathcal C$.
  This proves $\textup{(1-ii)}$.

  \vskip5pt

  (2') is the dual of (2). \end{proof}

Recall that a subcategory $\mathcal{W}\subseteq \mathcal{A}$ is \emph{thick} if it is closed under direct summands and satisfies the \emph{2-out-of-3 property} for $\mathbb E$-triangles, i.e.,
for any $\mathbb E$-triangle $A \longrightarrow B \longrightarrow C \dashrightarrow$ in $\mathcal A$, if two of $A,B,C$ are in $\mathcal W$, then the third one is also in $\mathcal W$.

\vskip5pt

\begin{definition-proposition} \label{def_inj_ctp} \ {\rm(\cite{Gil25})} \ A complete cotorsion pair $(\mathcal{W}, \mathcal{B})$ in an extriangulated category $(\mathcal A,\mathbb E,\mathfrak s)$ with enough injective objects
is called an \emph{injective cotorsion pair} if one of the following equivalent conditions holds$:$
\vskip5pt
$(1)$ \ $\mathcal{W}$ is thick, and $\mathcal{W} \cap \mathcal{B} = \mathcal{I}(\mathcal{A});$
\vskip5pt
$(2)$ \ $\mathcal{W}$ is thick, and $\mathcal{I}(\mathcal{A}) \subseteq \mathcal{W};$
\vskip5pt
$(3)$ \ $(\mathcal{W}, \mathcal{B})$ is hereditary, and $\mathcal{W} \cap \mathcal{B} = \mathcal{I}(\mathcal{A})$.
\end{definition-proposition}

\begin{proof} \ This has been proved for exact categories in  \cite[2.20, 2.21]{Gil25}. The proof for  extriangulated categories is similar.
  The implications $(1)\Longrightarrow (2)$ and $(1)\Longrightarrow (3)$ are clear.

  \vskip5pt

  $(2)\Longrightarrow (1)$: \ It remains to show $\mathcal W\cap\mathcal B\subseteq\mathcal I(\mathcal A)$.
  For any $X\in\mathcal W\cap\mathcal B$, by assumption there is an $\mathbb E$-triangle $X\longrightarrow I\longrightarrow Y\xdashrightarrow{\delta}$ with $I$ injective.
  By $(2)$ one has $I\in\mathcal W$, and  hence $Y\in\mathcal W$.
  Since $X\in\mathcal B=\mathcal W^\perp$, one has $\mathbb E(Y,X)=0$.
  Thus $I\cong X\oplus Y$, and $X$ is injective.

  \vskip5pt

  $(3)\Longrightarrow (1)$: \ Since $\mathcal W={}^\perp\mathcal B$ and $(\mathcal{W}, \mathcal{B})$ is a hereditary cotorsion pair, it remains to show that for every $\mathbb E$-triangle $W'\longrightarrow W\longrightarrow M\xdashrightarrow{\eta}$ with $W',W\in\mathcal W$, one has $M\in\mathcal W$.
  Since $(\mathcal{W}, \mathcal{B})$ is a complete cotorsion pair, there is an $\mathbb E$-triangle $B^M\longrightarrow W^M\longrightarrow M\xdashrightarrow{\delta}$ with $W^M\in\mathcal W$ and $B^M\in\mathcal B$.
  By \cite[Proposition 3.15]{NP19}, there is a commutative diagram of $\mathbb E$-triangles:
  \begin{equation*}
    \xymatrix@R=15pt{
    {} & B^M \ar@{=}[r] \ar@{..>}[d] & B^M \ar[d] & {} \\
    W' \ar@{=}[d] \ar@{..>}[r] & P \ar@{..>}[r] \ar@{..>}[d] & W^M \ar[d] \ar@{-->}[r]^{\theta} & {} \\
    W' \ar[r] & W \ar[r] \ar@{-->}[d]_{\zeta} & M \ar@{-->}[r]^{\eta} \ar@{-->}[d]^{\delta} & {} \\
    {} & {} & {} & {}
    }
  \end{equation*}
  The middle row gives $P\in\mathcal W$.
  Since $(\mathcal W,\mathcal B)$ is hereditary, the middle column then gives $B^M\in\mathcal W$.
  Thus $B^M\in\mathcal W\cap\mathcal B=\mathcal I(\mathcal A)$, and hence
  the $\mathbb E$-triangle $B^M\longrightarrow W^M\longrightarrow M\xdashrightarrow{\delta}{}$ splits.
  Therefore  $M$ is a direct summand of $W^M$, and $M\in\mathcal W$.
\end{proof}

\vskip5pt

\begin{remark} \ A hereditary cotorsion pair $(\mathcal{W}, \mathcal{B})$ with $\mathcal{I}(\mathcal{A}) \subseteq \mathcal{W}$ is not necessarily an injective cotorsion pair, as the following example shows.

  \vskip5pt

  Let $\mathrm{Ch}(k)$ be the category of complexes of vector spaces over field $k$.
  It is a Frobenius abelian category whose projective-injective objects are precisely the exact complexes.
  Put
  \begin{equation*}
    \mathcal{W}=\{X\mid \mathrm H^n(X)=0 \text{ for } n<0\}, \quad \mathcal{B}=\{Y\mid \mathrm H^n(Y)=0 \text{ for } n>0\}.
  \end{equation*}
  Let $K(k)$ denote the homotopy category of $\mathrm{Ch}(k)$.
  Then
  $\mathrm{Ext}_{\mathrm{Ch}(k)}^1(X,Y) \cong \mathrm{Hom}_{K(k)}(X,Y[1])$.
  Hence $\mathrm{Ext}_{\mathrm{Ch}(k)}^1(X,Y)=0$ if and only if, for every $n\in\mathbb Z$, at least one of $\mathrm H^n(X)$ and $\mathrm H^{n+1}(Y)$ is zero.
  It follows that $(\mathcal W,\mathcal B)$ is a complete hereditary cotorsion pair.
  Moreover, $\mathcal I(\mathrm{Ch}(k))\subseteq\mathcal W$, but $\mathcal W$ is not thick. Thus this cotorsion pair is not an injective cotorsion pair.
\end{remark}

\vskip5pt

We need the following result concerning the intersection of two cotorsion pairs. It is first stated for weakly idempotent complete exact categories in \cite{WZ25}.

\vskip5pt

\begin{lemma}  {\rm (\cite[Theorem 1.1]{WZ25})} \label{lem_ctp_3} \ Let $(\mathcal{C}, \mathcal{X})$ and $(\mathcal{W}, \mathcal{B})$ be cotorsion pairs in an extriangulated category $(\mathcal A,\mathbb E,\mathfrak s)$.
  Then  $(\mathcal{C} \cap \mathcal{W}, \ (\mathcal{C} \cap \mathcal{W})^\perp)$ is also a cotorsion pair. Moreover, if $\mathcal{X} \subseteq \mathcal{W}$, and if
  $(\mathcal{C}, \mathcal{X})$ and $(\mathcal{W}, \mathcal{B})$ are complete $($respectively, complete and hereditary$)$,
  then $(\mathcal{C} \cap \mathcal{W}, \ (\mathcal{C} \cap \mathcal{W})^\perp)$ is also a complete $($respectively, complete and hereditary$)$.
\end{lemma}
\begin{proof} \ It is clear that $^\perp ((\mathcal{C} \cap \mathcal{W})^\perp) \subseteq {}^\perp (\mathcal{C}^\perp) = \mathcal{C}$ and $^\perp ((\mathcal{C} \cap \mathcal{W})^\perp) \subseteq \mathcal{W}$.
  Thus $^\perp ((\mathcal{C} \cap \mathcal{W})^\perp) \subseteq \mathcal{C} \cap \mathcal{W}$.
  While $\mathcal{C} \cap \mathcal{W} \subseteq {}^\perp ((\mathcal{C} \cap \mathcal{W})^\perp)$ is clear.  It follows that  $(\mathcal{C} \cap \mathcal{W}, \ (\mathcal{C} \cap \mathcal{W})^\perp)$ is a cotorsion pair.

  \vskip5pt

  Suppose that $(\mathcal{C}, \mathcal{X})$ and $(\mathcal{W}, \mathcal{B})$ are complete with $\mathcal{X} \subseteq \mathcal{W}$.
  Let $M\in\mathcal{A}$. Then there are $\mathbb E$-triangles $M \longrightarrow B_M \longrightarrow W_M \xdashrightarrow{\delta}$ and $X^{W_M} \longrightarrow C^{W_M} \longrightarrow W_M \xdashrightarrow{\eta}$ with $B_M\in\mathcal{B}$, $W_M\in\mathcal{W}$, $C^{W_M}\in\mathcal{C}$ and $X^{W_M}\in\mathcal{X}$.
  Using \cite[Proposition 3.15]{NP19} one obtains a commutative diagram of $\mathbb E$-triangles:
  \[
    \xymatrix@R=15pt{
    & M \ar@{.>}[d] & M \ar@{->}[d] \ar@{=}[l] &  \\
    X^{W_M} \ar@{=}[d] \ar@{.>}[r] & E \ar@{.>}[d] \ar@{.>}[r] & B_M \ar@{->}[d] \ar@{-->}[r]^{\mu} & {} \\
    X^{W_M} \ar@{->}[r] & C^{W_M} \ar@{->}[r] \ar@{-->}[d]_{\xi} & W_M \ar@{-->}[r]^{\eta} \ar@{-->}[d]^{\delta} & {} \\
    & {} & {} &
    }
  \]
  Since $X^{W_M}$ and $B_M$ are in $(\mathcal{C} \cap \mathcal{W})^\perp$, one has $E \in (\mathcal{C} \cap \mathcal{W})^\perp$.
  By assumption $\mathcal{X} \subseteq \mathcal{W}$, one has $X^{W_M} \in \mathcal{W}$, and hence $C^{W_M} \in \mathcal{C} \cap \mathcal{W}$.
  Thus  $M \longrightarrow E \longrightarrow C^{W_M} \xdashrightarrow{\xi}$ is a required $\mathbb E$-triangle.

  \vskip5pt

  Again since $(\mathcal{C}, \mathcal{X})$ and $(\mathcal{W}, \mathcal{B})$ are complete, there are $\mathbb E$-triangles $B^M \longrightarrow W^M \longrightarrow M \xdashrightarrow{\zeta}$ and $X^{W^M} \longrightarrow C^{W^M} \longrightarrow W^M \xdashrightarrow{\theta}$ with $W^M\in\mathcal{W}$, $B^M\in\mathcal{B}$,  $C^{W^M}\in\mathcal{C}$ and $X^{W^M}\in\mathcal{X}$.
  By \textup{ET4}$^{\mathrm{op}}$ there is a commutative diagram of $\mathbb E$-triangles:
  \[
    \xymatrix@R=15pt{
    X^{W^M} \ar@{=}[d] \ar@{.>}[r] & F \ar@{.>}[r] \ar@{.>}[d] & B^M \ar@{->}[d] \ar@{-->}[r]^{\rho} & {} \\
    X^{W^M} \ar@{->}[r] & C^{W^M} \ar@{->}[r] \ar@{.>}[d] & W^M \ar@{->}[d] \ar@{-->}[r]^{\theta} & {} \\
    & M \ar@{=}[r] \ar@{-->}[d]_{\nu} & M \ar@{-->}[d]^{\zeta} &  \\
    & {} & {} &
    }
  \]
  Since $B^M$ and $X^{W^M}$ are in $(\mathcal{C} \cap \mathcal{W})^\perp$, one has $F \in (\mathcal{C} \cap \mathcal{W})^\perp$.
  By $\mathcal{X} \subseteq \mathcal{W}$, one has $X^{W^M} \in \mathcal{W}$ and $C^{W^M} \in \mathcal{C} \cap \mathcal{W}$.
  Thus the $\mathbb E$-triangle in the middle column shows that $(\mathcal{C} \cap \mathcal{W}, \ (\mathcal{C} \cap \mathcal{W})^\perp)$ is complete.

  \vskip5pt

  Assume that  $(\mathcal{C}, \mathcal{X})$ and $(\mathcal{W}, \mathcal{B})$ are complete and hereditary with $\mathcal{X} \subseteq \mathcal{W}$.
  By \Cref{lem_ctp_2} any $\mathbb E$-triangle $A\longrightarrow B\longrightarrow C\dashrightarrow$ with $B,C\in\mathcal{C}\cap\mathcal{W}$ yields $A\in\mathcal{C}\cap\mathcal{W}$.
  So $(\mathcal{C} \cap \mathcal{W}, \ (\mathcal{C} \cap \mathcal{W})^\perp)$ is hereditary.
\end{proof}

\subsection{Model structures  and homotopy categories} $\,$

\vskip5pt

The notion of a model structure was introduced by D. Quillen \cite{Q1}. Let $\mathcal A$ be an arbitrary category. A model structure on $\mathcal A$ is a triple $(\mathsf{Cofib}, \mathsf{Fib}, \mathsf{Weq})$ of classes of morphisms satisfying the Two-out-of-three Axiom, the Retract Axiom, the Lifting Axiom, and the Factorization Axiom.
By definition, the homotopy category $\mathsf{Ho}(\mathcal A)$ of the model structure $(\mathsf{Cofib}, \mathsf{Fib}, \mathsf{Weq})$ on $\mathcal A$ is the localization $\mathcal A[\mathsf{Weq}^{-1}]$ of $\mathcal A$ with respect to $\mathsf{Weq}$.
For details, see \cite{Q1}, \cite[p.~233]{Qui69}, or \cite[Definition 1.2.3]{H1}.

\vskip5pt

Let $(\mathsf{Cofib},  \mathsf{Fib}, \mathsf{Weq})$ be a model structure on a category $\mathcal A$.
The morphisms in $\mathsf{Cofib}$ (respectively, $\mathsf{Fib}, \mathsf{Weq})$ are called \emph{cofibrations} (respectively, \emph{fibrations}, \emph{weak equivalences}).
Set $\mathsf{TCofib}: = \mathsf{Cofib}\cap \mathsf{Weq}$ and $\mathsf{TFib}: = \mathsf{Fib}\cap \mathsf{Weq}$.
The morphisms in $\mathsf{TCofib}$ (respectively, $\mathsf{TFib}$) are called \emph{trivial cofibrations} (respectively, \emph{trivial fibrations}).
Assume that  $\mathcal A$ has a zero object $0$. An object $C$ is \emph{cofibrant} if
the morphism $0\longrightarrow C$ is a cofibration; an object $F$ is \emph{fibrant} if $F\longrightarrow 0$ is a fibration;
and an object $W$ is {\it trivial} if $0\longrightarrow W$ is a weak equivalence, or equivalently, $W \longrightarrow 0$ is a weak equivalence. Denote by $\mathcal C$ (respectively, $\mathcal F, \mathcal W$) the class of
cofibrant objects (respectively, fibrant objects, trivial objects) of $\mathcal A$. The objects in $\mathcal C\cap\mathcal W$ (respectively, $\mathcal F\cap\mathcal W$) are said to be
  {\it trivially cofibrant} (respectively, {\it trivially fibrant}).

\subsection{Exact model structures  and \textup{Hovey} triples} $\,$

\begin{definition} \label{def_exact_model_structure} \ {\rm (\cite[Definition 5.5]{NP19})} \   A model structure $(\mathsf{Cofib}, \mathsf{Fib}, \mathsf{Weq})$ on an extriangulated category $(\mathcal A,\mathbb E,\mathfrak s)$ is \emph{exact}, provided that the following conditions are satisfied:
  \vskip5pt
  $(1)$ \ $i\in\mathsf{Cofib}$ if and only if there is an $\mathbb E$-triangle
  $A\xlongrightarrow{i}B\longrightarrow C\dashrightarrow$ with $C\in\mathcal C$;
  \vskip5pt
  $(2)$ \ $p\in\mathsf{Fib}$ if and only if there is an $\mathbb E$-triangle
  $F\longrightarrow A\xlongrightarrow{p}B\dashrightarrow$ with $F\in\mathcal F$;
  \vskip5pt
  $(3)$ \ $j\in\mathsf{TCofib}$ if and only if there is an $\mathbb E$-triangle
  $A\xlongrightarrow{j}B\longrightarrow S\dashrightarrow$ with $S\in\mathcal C\cap\mathcal W$;
  \vskip5pt
  $(4)$ \ $q\in\mathsf{TFib}$ if and only if there is an $\mathbb E$-triangle
  $V\longrightarrow A\xlongrightarrow{q}B\dashrightarrow$ with $V\in\mathcal F\cap\mathcal W.$
\end{definition}

\vskip5pt

\begin{remark} \ The notion of exact model structure on abelian categories has been introduced by Hovey \cite{Hov02}, which is usually called {\it abelian model structure};
  it is extended to exact categories by Gillespie \cite{Gil11} under the name of {\it exact model structure};
  Nakaoka and Palu \cite{NP19} call this \emph{admissible model structure} for extriangulated categories.
\end{remark}

\vskip5pt

\begin{definition} \label{def_Hovey_triple} \ {\rm (\cite{Hov02}, \cite{Gil11}, \cite{Gil16})} \ A \emph{Hovey triple} in an extriangulated category $(\mathcal A,\mathbb E,\mathfrak s)$ is a triple $(\mathcal C,\mathcal F,\mathcal W)$ of classes of objects such that
  $(\mathcal C\cap\mathcal W,\mathcal F)$ and $(\mathcal C,\mathcal F\cap\mathcal W)$ are complete cotorsion pairs in $\mathcal A$, and $\mathcal W$ is a thick subcategory of $\mathcal A$.
  A Hovey triple $(\mathcal C,\mathcal F,\mathcal W)$ is {\it hereditary}, if the cotorsion pairs $(\mathcal C\cap\mathcal W,\mathcal F)$ and $(\mathcal C,\mathcal F\cap\mathcal W)$ are hereditary.
\end{definition}

For classes $\mathcal X$ and $\mathcal Y$ of objects in an extriangulated category $(\mathcal A,\mathbb E,\mathfrak s)$, put
\begin{align*}
  \operatorname{Cone}(\mathcal X,\mathcal Y)
   & :=\{M\mid \text{there is an $\mathbb E$-triangle
    $X\longrightarrow Y\longrightarrow M\dashrightarrow$
  with $X\in\mathcal X$ and $Y\in\mathcal Y$}\},      \\
  \operatorname{CoCone}(\mathcal X,\mathcal Y)
   & :=\{M\mid \text{there is an $\mathbb E$-triangle
    $M\longrightarrow X\longrightarrow Y\dashrightarrow$
    with $X\in\mathcal X$ and $Y\in\mathcal Y$}\}.
\end{align*}

\begin{definition} {\rm (\cite[Definition 5.1]{NP19})} \label{def_Hovey_twin} \ A pair $((\mathcal S, \mathcal T), \ (\mathcal U, \mathcal V))$ of complete cotorsion pairs in an extriangulated category $(\mathcal A,\mathbb E,\mathfrak s)$
  is a \emph{Hovey twin cotorsion pair},  provided that  $\mathbb E(\mathcal S,\mathcal V)=0$ and $\operatorname{Cone}(\mathcal V,\mathcal S) =\operatorname{CoCone}(\mathcal V,\mathcal S)$.
\end{definition}

The following observation identifies Hovey triples with Hovey twin cotorsion pairs.

\vskip5pt

\begin{lemma} \label{lem_Hovey_corr_1} \ {\rm (\cite{NP19})} \ Let $(\mathcal A,\mathbb E,\mathfrak s)$ be an extriangulated category.
  Then a Hovey triple $(\mathcal C,\mathcal F,\mathcal W)$ yields a Hovey twin cotorsion pair $((\mathcal C\cap\mathcal W,\mathcal F),(\mathcal C,\mathcal F\cap\mathcal W));$
  conversely, a Hovey twin cotorsion pair $((\mathcal S,\mathcal T),(\mathcal U,\mathcal V))$ yields a Hovey triple $(\mathcal U,\mathcal T, \mathrm{Cone}(\mathcal{V},\mathcal{S}))$.
\end{lemma}

\begin{proof} \ Suppose that $(\mathcal C,\mathcal F,\mathcal W)$ is a Hovey triple. Then it is clear that $((\mathcal C\cap\mathcal W,\mathcal F),(\mathcal C,\mathcal F\cap\mathcal W))$ is a Hovey twin cotorsion pair.
  The converse statement is proved in \cite[Proposition 5.3]{NP19}.
\end{proof}

The Hovey correspondence establishes a one-to-one correspondence between exact model structures and \textup{Hovey} triples.
It was established by Hovey \cite[Theorem 2.2]{Hov02} for abelian categories and extended by Gillespie \cite{Gil11} to weakly idempotent complete exact categories; see also J. Šťovíček \cite[6.9]{S}.
X. Y. Yang \cite{Yan15} gave its triangulated version.
Nakaoka and Palu \cite[Section 5]{NP19} unified these results in weakly idempotent complete extriangulated categories.
Surprisingly, they proved that the homotopy category
$\mathsf{Ho}(\mathcal A) \cong (\mathcal{C}\cap \mathcal{F})/(\mathcal{C}\cap \mathcal{F}\cap\mathcal W)$
is triangulated \cite[Theorem 6.20]{NP19}, even when the Hovey triple $(\mathcal{C}, \mathcal{F}, \mathcal W)$ is not hereditary.

\vskip5pt

\begin{theorem} \label{lem_Hovey_corr_2} \ {\rm (Hovey correspondence)} \
  Let $(\mathcal A,\mathbb E,\mathfrak s)$ be a weakly idempotent complete extriangulated category.
  Then there is a one-to-one correspondence between exact model structures and Hovey triples in $\mathcal A$, given by
  \ $(\mathsf{Cofib},\mathsf{Fib},\mathsf{Weq})\longmapsto (\mathcal{C}, \ \mathcal{F}, \ \mathcal W)$,  where  $\mathcal C$, $\mathcal F$ and $\mathcal W$ are respectively the classes of cofibrant objects, fibrant objects and trivial objects$;$
  the inverse is given by \ $(\mathcal{C}, \ \mathcal{F}, \ \mathcal W) \longmapsto (\mathsf{Cofib},\mathsf{Fib},\mathsf{Weq}),$  where $\mathsf{Cofib}$  and $\mathsf{Fib}$ are given as in Definition {\rm \ref{def_exact_model_structure}}, and
  $$\mathsf{Weq} = \{q\circ j \ | \ j \ \mbox{is an} \  \mathbb E\mbox{-inflation}, \  \operatorname{Cone}j\in\mathcal C\cap\mathcal W;  \ q \ \mbox{is an} \  \mathbb E\mbox{-deflation}, \ \operatorname{CoCone}q\in\mathcal F\cap\mathcal W\}.$$
  Moreover, in this case, the homotopy category $\mathsf{Ho}(\mathcal A) \cong (\mathcal{C}\cap \mathcal{F})/(\mathcal{C}\cap \mathcal{F}\cap\mathcal W)$ is a triangulated category.\end{theorem}

\subsection{Exceptional \textup{Hovey} triples} $\,$

\vskip5pt

\begin{definition} \label{frobpair} \ A pair $(\mathcal A, \mathcal P)$ is {\it a Frobenius pair}, if $\mathcal A$ is a Frobenius  extriangulated category, and $\mathcal P = \mathcal P(\mathcal A).$
\end{definition}

\vskip5pt

Let $(\mathcal{C}, \mathcal{F}, \mathcal{W})$ be a \textup{Hovey} triple in an extriangulated category $(\mathcal A,\mathbb E,\mathfrak s)$.
The extension-closed subcategory $\mathcal{C} \cap \mathcal{F}$ inherits an extriangulated structure from $\mathcal{A}$.
In particular, its extension bifunctor is the restriction of $\mathbb E$ to $(\mathcal C\cap\mathcal F)^{\mathrm{op}}\times(\mathcal C\cap\mathcal F)$.

\vskip5pt

\begin{definition} \label{exceptional} \ $(1)$ \ A \textup{Hovey} triple $(\mathcal{C}, \mathcal{F}, \mathcal{W})$ in an extriangulated category $(\mathcal A,\mathbb E,\mathfrak s)$ will be said to be \emph{exceptional},
  if the pair $(\mathcal{C} \cap \mathcal{F}, \ \mathcal{C} \cap \mathcal{F} \cap \mathcal{W})$ is not a Frobenius pair. Moreover,

  $\textup{(i)}$ \ An exceptional  \textup{Hovey} triple $(\mathcal{C}, \mathcal{F}, \mathcal{W})$ is of \emph{Type} I,  if $\mathcal{C} \cap \mathcal{F}$ is a Frobenius extriangulated category,
  but $\mathcal P(\mathcal{C} \cap \mathcal{F})\ne \mathcal{C} \cap \mathcal{F} \cap \mathcal{W}$ (or equivalently, $\mathcal{C} \cap \mathcal{F} \cap \mathcal{W} \subsetneqq \mathcal P(\mathcal{C} \cap \mathcal{F})$).

  $\textup{(ii)}$ \ An exceptional  \textup{Hovey} triple $(\mathcal{C}, \mathcal{F}, \mathcal{W})$ is of \emph{Type} II,  if $\mathcal{C} \cap \mathcal{F}$ is not a Frobenius extriangulated category.

  \vskip5pt

  $(2)$ \ An exact model structure on an extriangulated category is \emph{an exceptional exact model structure}, or simply, \emph{exceptional}, provided that the corresponding Hovey  triple under the Hovey correspondence is exceptional.
\end{definition}

\vskip5pt

\begin{fact} \label{lem_Hovey_hereditary} \label{prop_hereditary_Hovey_triple} \label{Wcontainsp} \ Let $(\mathcal{C}, \mathcal{F}, \mathcal{W})$ be a \textup{Hovey} triple in an extriangulated category $(\mathcal A,\mathbb E,\mathfrak s)$.
  Then
  \vskip5pt
  $(1)$ \ $\mathcal P(\mathcal A)\cup \mathcal I(\mathcal A) \subseteq \mathcal{W}.$

  \vskip5pt
  $(2)$ \  $\mathcal{C} \cap \mathcal{F} \cap \mathcal{W}\subseteq \mathcal P(\mathcal{C} \cap \mathcal{F})\cap \mathcal I(\mathcal{C} \cap \mathcal{F})$.
  \vskip5pt
  $(3)$ \ The cotorsion pair $(\mathcal C,\mathcal F\cap\mathcal W)$ is hereditary if and only if $(\mathcal C\cap\mathcal W,\mathcal F)$ is hereditary.
  \vskip5pt
  $(4)$ \ A hereditary Hovey triple $(\mathcal{C}, \mathcal{F}, \mathcal{W})$ is not exceptional, i.e.,
  the pair $(\mathcal{C} \cap \mathcal{F}, \ \mathcal{C} \cap \mathcal{F} \cap \mathcal{W})$ is a Frobenius pair. In other words, no exceptional Hovey triple is hereditary.
\end{fact}

\vskip5pt

We will see in Example \ref{nonherednonexcep} that a Hovey triple which is not hereditary is not necessarily exceptional. That is, there exist Hovey triples which are neither hereditary  nor exceptional.

\vskip5pt

\begin{proposition}\label{prop_D_enough_PI}\label{corollary_D_enough_PI} \ Let $A$ be a finite-dimensional algebra, and $(\mathcal{C} , \mathcal{F}, \mathcal{W})$ a \textup{Hovey} triple in $A\text{-}\mathrm{mod}$.

  If $\mathcal{C} \cap \mathcal{F}$ is covariantly finite in $A\text{-}\mathrm{mod}$, then $\mathcal{C} \cap \mathcal{F}$ has enough projective objects.

  If $\mathcal{C} \cap \mathcal{F}$ is contravariantly finite in $A\text{-}\mathrm{mod}$, then $\mathcal{C} \cap \mathcal{F}$ has  enough injective objects.

  In particular, if $A$ is representation-finite, then  $\mathcal{C} \cap \mathcal{F}$ has enough projective objects and enough injective objects.
\end{proposition}

\begin{proof} \ We only prove the first assertion. Assume that $\mathcal{C} \cap \mathcal{F}$ is covariantly finite in $A\text{-}\mathrm{mod}$.
  For $X\in\mathcal{C} \cap \mathcal{F}$, by the completeness of the cotorsion pair there is an exact sequence $0 \longrightarrow F \xlongrightarrow{i} W \longrightarrow X \longrightarrow 0$ with $W\in \mathcal{C} \cap \mathcal{W}$  and $F\in\mathcal{F}$.
  Clearly $W\in\mathcal{F}$, hence $W\in \mathcal{C} \cap \mathcal{F} \cap \mathcal{W}$, and $W\in {}^\perp (\mathcal{C} \cap \mathcal{F})$.
  By \cite[Proposition 1.1]{AR91}, $F$ admits a minimal left $(\mathcal{C} \cap \mathcal{F})$-approximation $i': F\longrightarrow D$.
  Since $W\in\mathcal{C} \cap \mathcal{F}$, the morphism $i:F \longrightarrow W$ factors through the left $\mathcal{C} \cap \mathcal{F}$-approximation $i^\prime$.
  Since $i$ is a monomorphism, so is $i^\prime$.
  By Wakamatsu's Lemma (cf. \cite[Proposition 1]{Wak90}), $\operatorname{Coker}i^\prime \in {}^\perp (\mathcal{C} \cap \mathcal{F})$.
  Consider the pushout diagram
  \[
    \xymatrix@R=.5cm @C=.5cm{& 0\ar[d] & 0\ar[d] \\
    0\ar[r] & F \ar@{->}[r]^-{i} \ar[d]_-{i^\prime } & W \ar@{->}[r] \ar[d] & X\ar[r] \ar@{=}[d]&0 \\
    0\ar[r] & D \ar@{->}[d] \ar[r] & E \ar[d] \ar[r] & X \ar[r] & 0 \\
    &    \operatorname{Coker}i^\prime \ar@{=}[r]\ar[d] & \operatorname{Coker}i^\prime \ar[d]&
    \\ & 0 & 0  }
  \]
  Since $\operatorname{Coker}i^\prime \in {}^\perp (\mathcal{C} \cap \mathcal{F})$, the middle column splits. Thus $E \cong W \oplus \operatorname{Coker}i^\prime$.
  Since $W$ and $\operatorname{Coker}i^\prime$ are in ${}^\perp (\mathcal{C} \cap \mathcal{F})$, one has $E \in {}^\perp (\mathcal{C} \cap \mathcal{F})$.
  Since $X, D \in \mathcal{C} \cap \mathcal{F}$, one has $E \in \mathcal{C} \cap \mathcal{F}$.
  It follows that $E$ is a projective object of $\mathcal{C} \cap \mathcal{F}$.
  Then the middle row shows  that $\mathcal{C} \cap \mathcal{F}$ has enough projective objects.
\end{proof}

\subsection{Gillespie's new construction of Hovey triples} $\,$

\vskip5pt

It has been a problem to construct non-hereditary \textup{Hovey} triples whose corresponding homotopy categories are nonzero.
Recently, Gillespie \cite{Gil26} introduced a new method for constructing non-hereditary \textup{Hovey} triples by merging a non-hereditary cotorsion pair with an injective cotorsion pair.
This result will play an important role in this paper.

\vskip5pt

\begin{theorem} {\rm (\cite[Theorems 3.2 and 3.4]{Gil26})} \label{theorem_merge} \ Let $\mathcal A$ be an extriangulated category with enough injective objects.
  Suppose that $(\mathcal{C}, \mathcal{X})$ is a complete cotorsion pair in $\mathcal A$, and $(\mathcal{W}, \mathcal{B})$ is an injective cotorsion pair in $\mathcal A$ with $\mathcal{X} \subseteq \mathcal{W}$.
  Then $(\mathcal{C}, \ (\mathcal{C} \cap \mathcal{W})^\perp, \ \mathcal{W})$ is a  \textup{Hovey} triple, and it is
  hereditary if and only if $(\mathcal{C}, \mathcal{X})$ is hereditary.
\end{theorem}

\begin{proof} \ Here is an alternative proof. By \Cref{lem_ctp_3},  $(\mathcal{C} \cap \mathcal{W}, \ (\mathcal{C} \cap \mathcal{W})^\perp)$ is a complete cotorsion pair.
  We claim that
  \[
    \operatorname{Cone}(\mathcal X,\mathcal C\cap\mathcal W)
    =\mathcal W=\operatorname{CoCone}(\mathcal X,\mathcal C\cap\mathcal W).
  \]
  The inclusion $\operatorname{Cone}(\mathcal X,\mathcal C\cap\mathcal W)\subseteq \mathcal W$ follows from the thickness of $\mathcal{W}$.
  Conversely, for any $W \in \mathcal{W}$, there is an $\mathbb E$-triangle $X\longrightarrow C\longrightarrow W\dashrightarrow$ with $X \in \mathcal{X}$ and $C \in \mathcal{C}$.
  Since $\mathcal{W}$ is thick and $\mathcal{X} \subseteq \mathcal{W}$, one has $C\in \mathcal{C} \cap \mathcal{W}$.
  Hence $W \in \operatorname{Cone}(\mathcal X,\mathcal C\cap\mathcal W)$.
  This proves  $\operatorname{Cone}(\mathcal X,\mathcal C\cap\mathcal W)=\mathcal W$.
  Similarly,  $\operatorname{CoCone}(\mathcal X,\mathcal C\cap\mathcal W)=\mathcal W$.  By  \Cref{def_Hovey_twin},  $((\mathcal{C} \cap \mathcal{W}, (\mathcal{C} \cap \mathcal{W})^\perp),(\mathcal{C}, \mathcal{X}))$ is a Hovey twin cotorsion pair, and hence $(\mathcal{C}, (\mathcal{C} \cap \mathcal{W})^\perp, \mathcal{W})$ is a Hovey triple, by Lemma \ref{lem_Hovey_corr_1}.

  \vskip5pt  By Fact \ref{lem_Hovey_hereditary}(3),  $(\mathcal{C}, (\mathcal{C} \cap \mathcal{W})^\perp, \mathcal{W})$ is
  hereditary if and only if so is $(\mathcal{C}, \mathcal{X})$.
\end{proof}

Let $R$ be a ring, $\mathrm{Ch}(R)$ the category of complexes of $R$-modules, and $\mathrm K(R)$ the homotopy category of $\mathrm{Ch}(R)$.
Denote by $\mathcal{E}$ the class of exact complexes in $\mathrm{Ch}(R)$.
For any class of $R$-modules $\mathcal{L}$, let $\mathcal{E}_{\mathcal{L}}$ be the class of exact complexes with cycles (i.e. kernels of differentials) in $\mathcal{L}$.
For a cotorsion pair $(\mathcal{L}, \mathcal{R})$ in $R\text{-}\mathrm{Mod}$, put
\begin{align*}
  \mathrm{dg}\mathcal L  & = \{R\text{-complex} \ X \mid X^n\in \mathcal{L}, \ \mathrm{Hom}^\bullet(X, \mathcal E_\mathcal{R}) \ \text{is exact} \}= \{X \mid \mathrm{Hom}_{\mathrm K(R)}(X, \mathcal E_\mathcal R)= 0 \}  \\
  \mathrm{dg}\mathcal{R} & = \{R\text{-complex} \ X \mid X^n\in \mathcal{R}, \ \mathrm{Hom}^\bullet(\mathcal E_\mathcal{L}, X) \ \text{is exact} \}= \{X \mid \mathrm{Hom}_{\mathrm K(R)}(\mathcal E_\mathcal L, X)= 0 \}.
\end{align*}
It follows from \cite{Gil04} that $\mathrm{dg}{\mathcal{L}} = {}^\perp(\mathcal{E}_{\mathcal{R}})$ and $\mathrm{dg}{\mathcal{R}} = (\mathcal{E}_{\mathcal{L}})^\perp$.

\vskip5pt

As an application of \Cref{theorem_merge}, Gillespie \cite[Theorem 4.1]{Gil26} gives a new construction of Hovey triples for derived categories.
This construction also yields non-hereditary Hovey triples; see \cite[Example 4.2]{Gil26}.
The following proposition shows that all the Hovey triples constructed in \cite[Theorem 4.1]{Gil26} are not exceptional,
	in other words, there exist non-hereditary and non-exceptional Hovey triples.

\vskip5pt

\begin{proposition} \label{prop_Gillespie} \ Let $\mathcal{X}$ be a set of $R$-modules.
  Put $(\mathcal{L}, \mathcal{R}) = ({}^\perp (\mathcal{X}^\perp), \mathcal{X}^\perp)$.
  Then $(\mathrm{dg}{\mathcal{L}}, (\mathrm{dg}{\mathcal{L}} \cap \mathcal{E})^\perp, \mathcal{E})$ is a \textup{Hovey} triple in $\mathrm{Ch}(R)$,
  and it is not exceptional, namely,
  $(\mathrm{dg}{\mathcal{L}}\cap (\mathrm{dg}{\mathcal{L}} \cap \mathcal{E})^\perp, \mathrm{dg}{\mathcal{L}}\cap (\mathrm{dg}{\mathcal{L}} \cap \mathcal{E})^\perp\cap \mathcal{E})$ is a Frobenius pair.
\end{proposition}

\begin{proof} \ By \cite[Theorem 4.1]{Gil26},  $(\mathrm{dg}{\mathcal{L}}, (\mathrm{dg}{\mathcal{L}} \cap \mathcal{E})^\perp, \mathcal{E})$ is a Hovey triple, which is denoted by $(\mathcal{C}, \mathcal{F}, \mathcal{W})$.
  Let $N$ be an object of $\mathcal{C} \cap \mathcal{F}$.
  Then there is a short exact sequence $0 \longrightarrow N \longrightarrow \mathrm{cone}(\mathrm{Id}_N) \longrightarrow N[1] \longrightarrow 0$
  in $\mathcal C\cap\mathcal F$, where $\mathrm{cone}(\mathrm{Id}_N)$ is the mapping cone of the identity map $\mathrm{Id}_N$.
  Since $\mathrm{cone}(\mathrm{Id}_N)$ is contractible, it lies in $\mathcal{E}$.
  It follows that $\mathrm{cone}(\mathrm{Id}_N) \in \mathcal{C} \cap \mathcal{F} \cap \mathcal{W}$.
  In particular, if $N$ is an injective object of $\mathcal{C} \cap \mathcal{F}$, then the short exact sequence splits. Hence $N \in \mathcal{C} \cap \mathcal{F} \cap \mathcal{W}$.
  Since $(\mathcal{C} \cap \mathcal{F} \cap \mathcal{W}) \subseteq \mathcal{I}(\mathcal{C} \cap \mathcal{F})$, one sees that $\mathcal{C} \cap \mathcal{F} \cap \mathcal{W} = \mathcal{I}(\mathcal{C} \cap \mathcal{F})$.
  The above construction shows that $\mathcal{C} \cap \mathcal{F}$ has enough injective objects, which is $\mathcal{C} \cap \mathcal{F} \cap \mathcal{W}$.
  Similarly, $\mathcal{C} \cap \mathcal{F}$ has enough projective objects,  and $\mathcal{C} \cap \mathcal{F} \cap \mathcal{W} = \mathcal{P}(\mathcal{C} \cap \mathcal{F})$.
  Thus $(\mathcal{C} \cap \mathcal{F}, \mathcal{C} \cap \mathcal{F} \cap \mathcal{W})$ is a Frobenius pair.
\end{proof}

The following example indeed gives Hovey triples that are neither hereditary nor exceptional.

\begin{example} \label{nonherednonexcep} {\rm (\cite[Example 4.2]{Gil26})} \
  For a ring $R$, let $\mathcal S$ be the class of finitely presented left $R$-modules, which is essentially a set.
  Let $(\mathcal{FP}, \mathcal{FI}) = ({}^\perp(\mathcal S^\perp), \mathcal S^\perp)$ denote the complete cotorsion pair,
  where $\mathcal{FP}$ and $\mathcal{FI}$ are FP-projective modules and FP-injective modules, respectively (see, e.g., \cite[Remark 2.8]{MD05}).
  By \Cref{prop_Gillespie}, $(\mathrm{dg}\mathcal{FP}, (\mathrm{dg}\mathcal{FP} \cap \mathcal{E})^\perp, \mathcal{E})$ is a non-exceptional Hovey triple.

  \vskip5pt

  By \cite[Proposition~B.3]{Sto14}, $R$ is left coherent if and only if $(\mathcal{FP},\mathcal{FI})$ is hereditary, and if and only if $(\mathrm{dg}\mathcal{FP},\mathcal{E}_{\mathcal{FI}})$ is hereditary, by \cite[Corollary~3.13]{Gil04}.
  Hence for a non-left-coherent ring $R$, $(\mathrm{dg}\mathcal{FP}, (\mathrm{dg}\mathcal{FP} \cap \mathcal{E})^\perp, \mathcal{E})$ is a non-hereditary \textup{Hovey} triple.
\end{example}

\section{\bf Exceptional \textup{Hovey} triples from selfinjective Nakayama algebras}\label{sec_more_examples}

We will give a combinatorial description of selfinjective Nakayama algebras $A$ such that $A$-mod admits an exceptional Hovey triple.
A main tool  is Theorem \ref{theorem_merge}, Gillespie's  new construction  of  Hovey triples.

\subsection{Nakayama algebras and the main result}\label{sec_Nak_alg} Let $A$ be a finite-dimensional algebra over an algebraically closed field $k$.
An $A$-module $M$ is \emph{uniserial} if the set of submodules is totally ordered by inclusion.
The algebra $A$ is \emph{Nakayama} if all the indecomposable projective left $A$-modules and all the indecomposable injective left $A$-modules are uniserial, or equivalently, every indecomposable
left $A$-module is uniserial. The algebra $A$ is Nakayama if and only if $A^{\rm op}$ is Nakayama.
A basic connected Nakayama algebra $A$ is isomorphic to $kQ/I$, where the quiver $Q$ is either $\mathbb{A}_n \ (n\ge 1)$ with linear order :
$1 \longrightarrow \cdots \longrightarrow n$,
or a cycle $C_n \ (n\ge 1)$ with cyclic orientation:
\begin{equation*}
  \xymatrix{
  1 \ar@{->}[r] & 2 \ar@{->}[r] & \cdots \ar@{->}[r] & n-1 \ar@{->}[r] & n \ar@/^1.2pc/@{->}[llll]}
\end{equation*}

\vskip10pt
\noindent $I$ is an ideal of the path algebra $kQ$ generated by
some paths with $J^m \subseteq I \subseteq J^2$, and $J$ is the ideal generated by arrows of $Q$.
A Nakayama  algebra $A$ is representation-finite, and it is selfinjective if and only if $Q = C_n$ and $I = J^t$ for some $t\ge 2$.
We refer to \cite[IV.2]{ARS} and \cite [V.3]{ASS06} for more information.

\vskip5pt

In this paper we only consider  Nakayama algebras $A$ without simple projective modules. This is precisely the case $Q = C_n$.  Let $S_1, \dots, S_n$ be the pairwise non-isomorphic simple $A$-modules,
$P_i$ (respectively, $I_i$)
the indecomposable projective (respectively, injective) module with top (respectively, socle) $S_i$, and $(c_1, \cdots, c_n)$ the Kupisch series of $A$, i.e., $c_i$ is the length of $P_i$.
Denote by $M_{i,l}$ the indecomposable $A$-module with top $S_i$ and length $l$. Then $M_{i,l}$ \ ($1\le l\le c_i$, $i\in\mathbb Z/n\mathbb Z$) give all the indecomposable $A$-modules.
We will often use the formula ${\rm soc}(M_{i,l}) = S_{i+l-1}$.

\vskip5pt

Throughout this section, $A$ is assumed to be a selfinjective Nakayama algebra. Thus $A = kC_n/J^t$ with $t\ge 2$. We will not repeat this assumption each time.
In this case  $P_i = M_{i,t} = I_{i+t-1}$ and
$\mathcal{P}(A) = \mathcal{I}(A) = \operatorname{add} (P)$, where $P = \bigoplus\limits_{1\le i\le n}P_i =  \bigoplus\limits_{1\le i\le n}M_{i, t}.$
Let $P(M_{i, l})$ and $I(M_{i, l})$ be respectively the projective cover and the injective envelope of $M_{i, l}$, $\Omega(M)$ and $\mho(M)$ the syzygy and cosyzygy of $M$, respectively.
Then for $M_{i, l} \notin \mathcal{P}(A)$ one has
\begin{equation*}
  P(M_{i, l}) =  M_{i, t}, \ \ \Omega(M_{i, l}) = M_{i+l, t-l},  \ \ I(M_{i, l}) =  M_{i-t+l, t}, \ \  \mho(M_{i,l}) = M_{i-t+l, t-l}.
\end{equation*}

\vskip5pt

\begin{theorem}\label{thm_classification} \ Let $A= kC_n/J^t \ (t\ge 2)$ be a selfinjective Nakayama algebra. Then $A\mbox{-}{\rm mod}$ admits an exceptional \textup{Hovey} triple if and only if $\gcd(n, t) > 1$ and $t \geq 3$.
\end{theorem}

\vskip5pt

When $\gcd(n, t) > 1$ and $t \geq 3$, we will explicitly construct an exceptional \textup{Hovey} triple in  $A\mbox{-}{\rm mod}$, and determine its type.

\subsection{Thick subcategories of $A$-mod}  For the case $\gcd(n, t) = 1$ or $t = 2$, we can determine all the thick subcategories $\mathcal{W}$ of $A$-mod satisfying $\mathcal{W}\supseteq \mathcal{P}(A)$, 
and then all the Hovey triples $(\mathcal C, \ \mathcal F, \ \mathcal W)$ in $A$-mod can be determined.

\vskip5pt

\begin{lemma}\label{lemgcd} \ If $\gcd(n, t) = 1$ or $t = 2$, then any thick subcategory $\mathcal{W}$ of $A\mbox{-}{\rm mod}$ containing all the projective-injective modules is either $A\mbox{-}{\rm mod}$ or $\mathcal{P}(A)$.
\end{lemma}

\begin{proof} \ First, assume that $\gcd(n, t) = 1$. Suppose that $\mathcal{W} \neq \mathcal{P}(A)$.
  It suffices to prove that every simple module $M_{j,1}$ belongs to $\mathcal{W}$ for all $j\in \mathbb Z/n\mathbb Z$.
  Since $\mathcal{P}(A)\subsetneqq \mathcal{W}$, there is a non-projective indecomposable module $M_{i, l} \in \mathcal{W}$. Thus  $l < t$.
  Since $\mathcal{W}$ is thick, $\Omega^2 (M_{i,l})  = M_{i + t, l} \in \mathcal{W}$. Repeating this process, by $\gcd(n,t) = 1$, one sees that $M_{j, l} \in \mathcal{W}$ for all $j\in \mathbb Z/n\mathbb Z$.
  If $l = 1$, then we are done.
  Otherwise, $l> 1$. By the exact sequence (cf. Fact \ref{pp}) and the thickness of $\mathcal{W}$
  $$0 \longrightarrow M_{i, l} \longrightarrow M_{i, l-1} \oplus M_{i-1, l+1} \longrightarrow M_{i-1, l} \longrightarrow 0$$
  one sees that  $M_{i, l-1} \in \mathcal{W}$.
  By induction $M_{i, 1} \in \mathcal{W}$. Thus $\Omega^2 (M_{i,1})  = M_{i + t, 1} \in \mathcal{W}$.
  Since $\gcd(n,t) = 1$,  one sees that $M_{j, 1}\in \mathcal{W}$ for all $j\in \mathbb Z/n\mathbb Z$.\vskip5pt

  Now, assume that $t = 2$. Suppose that $\mathcal{W} \neq \mathcal{P}(A)$. Then there is a non-projective indecomposable module $M_{i, 1} \in \mathcal{W}$ for some $i\in \mathbb Z/n\mathbb Z$.
  We need to prove $M_{j, 1}\in \mathcal{W}$ for any $j\in \mathbb Z/n\mathbb Z$.
  By the assumption one sees that  $\Omega(M_{i,1}) = M_{i+1, 1} \in \mathcal{W}$.
  Repeating this process we are done. \end{proof}

\vskip5pt

By Fact \ref{Wcontainsp} and Lemma \ref{lemgcd} one gets

\vskip5pt

\begin{corollary}\label{lem_hovey_coprime_or_rad_square_zero} \ If $\gcd(n, t) = 1$ or $t = 2$, then any \textup{Hovey} triple $(\mathcal{C}, \mathcal{F}, \mathcal{W})$ in $A\mbox{-}{\rm mod}$
  is either $(A\mbox{-}{\rm mod}, \  A\mbox{-}{\rm mod}, \ \mathcal{P}(A))$, or of the form $(\mathcal{C}, \ \mathcal{F}, \ A\mbox{-}{\rm mod})$ for some complete cotorsion pair $(\mathcal{C}, \mathcal{F})$ in $A\mbox{-}{\rm mod}$.
\end{corollary}

\subsection{Complete cotorsion pairs in $A\mbox{-}{\rm mod}$} $\,$ \vskip5pt

Throughout this subsection, let $A = kC_n/J^t$ be a selfinjective algebra with $\gcd(n,t)=d\neq 1$ and $t\geq 3$, and set
$R=d\mathbb Z/n\mathbb Z\subseteq\mathbb Z/n\mathbb Z$. Set $P=\bigoplus\limits_{1\le i\le n}P_i=\bigoplus\limits_{1\le i\le n}M_{i,t}$.

\vskip10pt

\begin{lemma}\label{lem_sinj_cx} \ Put \ $\mathcal{C} =\operatorname{add}(P \oplus \bigoplus\limits_{i \notin R, 1\le l< t} M_{i, l}) \quad \text{and} \quad \mathcal{X} = \operatorname{add}(P \oplus \bigoplus\limits_{i \in R} M_{i+1, t-1}).
$
  Then $(\mathcal{C},\mathcal{X})$ is a non-hereditary complete cotorsion pair in $A\mbox{-}{\rm mod}$.
\end{lemma}

\begin{proof} \ Put $\mho \mathcal{X} = \{\mho X \mid X \in \mathcal{X}\}$. Since $\mho(M_{i+1, t-1}) = M_{i, 1}$, one has $\mho \mathcal{X} = \operatorname{add}(\bigoplus\limits_{i \in R} M_{i, 1})$.
  By the formula $\mathrm{Ext}^1_A(M,N)\cong \underline{\mathrm{Hom}}_A(M, \mho N)$ (cf. Lemma \ref{extinfrobenius}), a non-projective indecomposable module $M_{j,l}$ belongs to ${}^\perp\mathcal{X}$ if and only if
  $\underline{\mathrm{Hom}}_A(M_{j,l}, \bigoplus\limits_{i\in R}M_{i,1})=0$, and if and only if $j\notin R$ and $1 \leq l < t$ by \Cref{simphom}.
  Thus $\mathcal{C} = {}^\perp \mathcal{X}$. To see that $\mathcal{C}^\perp \subseteq\mathcal{X}$, it suffices to prove

  \vskip5pt

  {\bf Claim 1} \ For any non-projective  module $M_{j, l}\notin \mathcal{X}$,  there is $L \in \mathcal{C}$ such that
  $\mathrm{Ext}^1_A(L, M_{j, l})\ne 0$.

  \vskip5pt

  Using the formula $\mathrm{Ext}^1_A(M,N)\cong \underline{\mathrm{Hom}}_A(M, \mho N)$ and $\mho (M_{j, l})= M_{j-t+l, t-l}$,  {\bf Claim 1} is equivalent to

  \vskip5pt

  {\bf Claim 2} \ For any non-projective  module $M_{j, l}\notin \mathcal{X}$,  there is $L \in \mathcal{C}$ such that
  $\underline{\mathrm{Hom}}_A(L, M_{j-t+l, t-l}) \neq 0$.

  \vskip5pt

  Since $M_{j, l}$ is not projective, by construction  $M_{j, l}\notin \mathcal{X}$ if and only if
  $M_{j, l}\notin \{M_{i+1, t-1} \ | \ i\in R\}$, or equivalently, $\mho (M_{j, l})= M_{j-t+l, t-l}\notin \{M_{i, 1} \mid i \in R\}$.

  \vskip5pt

  Assume that {\bf Claim 2} is not true, i.e., there is a non-projective module $M_{j, l}\notin \mathcal{X}$,  such that
  $\underline{\mathrm{Hom}}_A(L, M_{j-t+l, t-l})= 0$ for any $L \in \mathcal{C}$. Then there are two cases:

  \vskip5pt

  \textup{Case 1.} \ If $(j+l) \notin R$, then by construction $\mho(M_{j,l}) = M_{j-t+l, t-l} \in \mathcal{C}$.
  Since $M_{j-t+l,t-l}$ is not projective, one has
  $\underline{\mathrm{End}}_A(M_{j-t+l,t-l})\neq0$, contradicting the assumption.
  \vskip5pt
  \textup{Case 2.} \ Assume that $(j+l) \in R$. Then $j-t+l\in R$ and hence $j-t+l+1\notin R$. Since $M_{j-t+l,t-l}\notin\{M_{i,1}\mid i\in R\}$, one has $t-l\ne1$.
  Thus $\mho(M_{j,l}) = M_{j-t+l, t-l}$ is not a simple module.
  Put $L = \mathrm{rad}(M_{j-t+l, t-l}) = M_{j-t+l+1, t-l-1}$. Since $j-t+l+1\notin R$, by construction $L= M_{j-t+l+1, t-l-1}\in \mathcal{C}$.
  However, the radical inclusion $L\longrightarrow M_{j-t+l,t-l}$ does not factor through an injective module (cf. Lemma \ref{radical}). Hence
  $\underline{\mathrm{Hom}}_A(L,M_{j-t+l,t-l})\neq0$, contradicting the assumption.
  \vskip5pt
  This proves {\bf Claim 2}. Thus $(\mathcal{C},\mathcal{X})$ is a cotorsion pair in $A\mbox{-}{\rm mod}$.
  It is complete by \Cref{prop_ctp_1}.
  Since $\mho \mathcal{X} \not\subseteq \mathcal{X}$, it is not hereditary by Lemma \ref{lem_ctp_2}.
\end{proof}

\begin{lemma}\label{lem_sinj_wb} \  Put
  \begin{equation*}
    \mathcal{W} = \operatorname{add}(P \oplus \bigoplus_{i \in R} (M_{i, 1} \oplus M_{i+1, t-1})) \quad \text{and} \quad \mathcal{B} = \operatorname{add}(P \oplus \bigoplus_{\substack{j-1\notin R,\\ j + l - 1 \notin R,\\ 1 \leq l < t}} M_{j,l}).
  \end{equation*}
  Then $(\mathcal{W},\mathcal{B})$ is an injective cotorsion pair in $A\mbox{-}{\rm mod}$.
\end{lemma}

\begin{proof} \ We first show that $\mathcal{W}^\perp = \mathcal{B}$. By definition
  \begin{equation}\label{eqn_wperp}
    \mathcal{W}^\perp = (\bigoplus_{i \in R} (M_{i, 1} \oplus M_{i+1, t-1}))^\perp =\bigcap_{i \in R} (M_{i, 1})^\perp \ \cap \ \bigcap_{i \in R} (M_{i+1, t-1})^\perp.
  \end{equation}
  Let $M_{j,l}$ be a non-projective module.  By \Cref{extinfrobenius}, $M_{j,l}\in (M_{i, 1})^\perp$ if and only if
  $$\underline{\mathrm{Hom}}_A(M_{i,1}, \mho M_{j,l}) = \underline{\mathrm{Hom}}_A(M_{i,1}, M_{j+l-t, t-l}) = 0.$$
  Thus, by Lemma \ref{simphom} one has
  \begin{equation}\label{eqn1wperp} M_{j,l}\in (M_{i, 1})^\perp \ \ \ \text{if and only if} \ \ \ j\ne i + 1.\end{equation}

  \vskip5pt

  \noindent Similarly,    by \Cref{extinfrobenius}, $M_{j,l}\in (M_{i+1, t-1})^\perp$ if and only if
  $$\underline{\mathrm{Hom}}_A(\Omega M_{i+1, t-1}, M_{j,l}) = \underline{\mathrm{Hom}}_A(M_{i+t, 1}, M_{j, l}) = 0$$
  and by Lemma \ref{simphom}  one has
  \begin{equation}\label{eqn2wperp} M_{j,l}\in (M_{i+1, t-1})^\perp \ \ \ \text{if and only if} \ \ \ j\ne i+t-l+1.\end{equation}
  By \Cref{eqn1wperp}  and   \Cref{eqn2wperp} one has
  \begin{equation*}\label{eqn3wperp} M_{j,l}\in \mathcal{W}^\perp \ \ \ \text{if and only if}  \ \ \ j\ne i+1, \  \ \ j\ne i + t - l + 1, \ \ \ \forall \ i\in R.\end{equation*}
  Since $R = d\mathbb Z/n\mathbb Z \subseteq \mathbb Z/n\mathbb Z$ and $d = \gcd(n, t )$, one has $R+t=R$, i.e.,
  $$j\ne i+1, \  \ \ j\ne i + t - l + 1,  \  \forall \ i\in R,  \  \ \text{if and only if}  \ \ \ j-1\notin R, \  \ \ j+l-1\notin R.$$
  Thus  $\mathcal{W}^\perp = \mathcal{B}$, by \Cref{eqn_wperp}.

  \vskip5pt

  To see that $(\mathcal{W}, \mathcal{B})$ is a cotorsion pair, it remains to prove  ${}^\perp \mathcal{B} \subseteq \mathcal{W}$, or equivalently, to prove
  \vskip5pt
  {\bf Claim} \ For any non-projective module $M_{j,l}\notin \mathcal{W}$ with $j\in \mathbb Z/n\mathbb Z$ and $1 \leq l < t$, there is $B \in \mathcal{B}$ such that $\mathrm{Ext}^1_A(M_{j,l}, B)\ne 0$.

  \vskip5pt
  We divide the proof into the following four cases.

  \vskip5pt

  {\bf Case 1} \ Suppose that $j+l-1\in R$.  Put $B=M_{j+l-1,2}$.  Since $d>1$, one has $(j+l-1)-1\notin R$ and $(j+l-1)+2-1\notin R$.  Hence $B\in\mathcal B$ by construction.
  Moreover, $l\geq2$; otherwise, $l=1$, so $j=j+l-1\in R$ and hence $M_{j,l}\in\mathcal W$, contradicting the assumption. Since $\mathrm{top}(M_{j+l,t-l})=S_{j+l}=\mathrm{soc}(B)$, the composition
  \begin{equation*}
    f:M_{j+l,t-l}\twoheadrightarrow S_{j+l}\hookrightarrow B=M_{j+l-1,2}
  \end{equation*}
  is not zero.  We claim that $f$ does not factor through a projective module.  Otherwise, $f$ factors through the injective envelope $\sigma:M_{j+l,t-l}\hookrightarrow M_{j,t}$.  Thus $f=g\circ\sigma$ for some $g:M_{j,t}\longrightarrow M_{j+l-1,2}$.  Since $M_{j+l,t-l}=\operatorname{rad}^lM_{j,t}$ with $l\geq2$ and the length of $B$ is $2$, we obtain
  $$f(M_{j+l,t-l})=g(M_{j+l,t-l})=g(\operatorname{rad}^lM_{j,t})
  \subseteq\operatorname{rad}^lg(M_{j,t})
  \subseteq\operatorname{rad}^lM_{j+l-1,2}=0,$$
  contradicting $f\neq0$.
  Thus $\mathrm{Ext}^1_A(M_{j,l},B)\cong \underline{\mathrm{Hom}}_A(M_{j+l,t-l}, M_{j+l-1,2})\neq0$.

  \vskip5pt

  {\bf Case 2} \ Suppose that $j+l-1\notin R$ and $j-1\notin R$.  Put $B=\Omega M_{j,l}=M_{j+l,t-l}$.
  Since $(j+l)-1 \notin R$ and $(j+l)+(t-l)-1=j+t-1 \notin R$, by construction $B\in\mathcal B$.  Since $B$ is not a projective $A$-module, by \Cref{extinfrobenius} one sees
  \begin{align*}
    \mathrm{Ext}^1_A(M_{j,l},B)
    \cong\underline{\mathrm{Hom}}_A(\Omega M_{j,l},B)\cong\underline{\mathrm{End}}_A(B)\neq0.
  \end{align*}

  \vskip5pt

  {\bf Case 3} \ Suppose that $j+l-1\notin R$, $j-1\in R$, and $j+l\in R$.  Put $B=M_{j+l,2}$.  Since $d>1$, one has $(j+l)-1\notin R$ and $(j+l)+2-1\notin R$. Thus  $B\in\mathcal B$.
  Moreover, $t-l\geq2$; otherwise, $l=t-1$, and $j-1\in R$ implies
  $M_{j,l}=M_{j,t-1}\in\mathcal W$, contradicting the assumption.
  Now since $t-l \geq 2$, there is a canonical epimorphism
  \begin{equation*}
    M_{j+l,t-l}\twoheadrightarrow M_{j+l,2}=B.
  \end{equation*}
  By \Cref{radical}$(2)$, this epimorphism between non-projective indecomposable modules does not factor through a projective module.  Hence
  \begin{equation*}
    \mathrm{Ext}^1_A(M_{j,l},B)
    \cong\underline{\mathrm{Hom}}_A(\Omega M_{j,l},B)\cong\underline{\mathrm{Hom}}_A(M_{j+l,t-l},M_{j+l,2})\neq0.
  \end{equation*}

  \vskip5pt

  {\bf Case 4} \ Suppose that $j+l-1\notin R$, $j-1\in R$, and $j+l\notin R$.  Put $B=M_{j+l,1}$.  Since $(j+l)-1\notin R$ and $(j+l)+1-1\notin R$, one has $B\in\mathcal B$.  By \Cref{simphom} one has
  \begin{equation*}
    \mathrm{Ext}^1_A(M_{j,l},B)
    \cong\underline{\mathrm{Hom}}_A(\Omega M_{j,l},B)\cong\underline{\mathrm{Hom}}_A(M_{j+l,t-l},M_{j+l,1})\neq0.
  \end{equation*}

  \vskip5pt

  These cases exhaust all the possibilities. Thus {\bf Claim} is proved. So $(\mathcal{W}, \mathcal{B})$ is a cotorsion pair. By \Cref{prop_ctp_1}, it is complete. By construction  $\mathcal{W} \cap \mathcal{B} = \mathcal{P}(A)$.
  We claim that $\mathcal B$ is closed under cosyzygies. In fact, for any non-injective $M_{j,l}\in \mathcal B$, one has $j-1 \notin R$, $j+l-1 \notin R$, and $1 \leq l < t$. Then $\mho M_{j,l}= M_{j+l-t, t-l}$.
  Since $(j+l-t)-1 \notin R$ and $(j+l-t)+(t-l)-1=j-1 \notin R$,  $\mho M_{j,l}= M_{j+l-t, t-l}\in\mathcal B$.
  By \Cref{cor_ctp_2},  $(\mathcal{W}, \mathcal{B})$ is hereditary.
  By \Cref{def_inj_ctp}, $(\mathcal{W}, \mathcal{B})$ is an injective cotorsion pair.
\end{proof}

\subsection{Exceptional \textup{Hovey} triple in $A\mbox{-}{\rm mod}$}
$\,$

\vskip5pt

\begin{lemma}\label{lem_sinj_hovey_exceptional} \  Put   $\mathcal{C} =\operatorname{add}(P\oplus\bigoplus\limits_{\substack{i\notin R,\\ 1\leq l<t}}M_{i,l}), \quad
    \mathcal{F} =\operatorname{add}(P\oplus\bigoplus\limits_{\substack{i+l-1\notin R,\\ 1\leq l<t}}M_{i,l})$ and
  \vskip5pt
  \noindent  $\mathcal{W} =\operatorname{add}(P\oplus\bigoplus\limits_{i\in R}(M_{i,1}\oplus M_{i+1,t-1})).$
  Then \begin{equation*}
    \mathcal{C}\cap\mathcal{F}
    =\operatorname{add}(P\oplus\bigoplus_{\substack{i\notin R,\\ i+l-1\notin R, \\ 1\le l< t}}M_{i,l}), \quad \ \ \mathcal{C}\cap\mathcal{F}\cap\mathcal{W}
    =\operatorname{add}(P\oplus\bigoplus_{i\in R}M_{i+1,t-1});
  \end{equation*}
  and $(\mathcal{C}, \mathcal{F}, \mathcal{W})$ is a \textup{Hovey} triple in $A\mbox{-}{\rm mod}$, and it is exceptional,
  i.e., $(\mathcal{C} \cap \mathcal{F}, \mathcal{C} \cap \mathcal{F} \cap \mathcal{W})$ is not a Frobenius pair.
\end{lemma}

\begin{proof} \  Put
  \begin{equation*}
    \mathcal{X}=\operatorname{add}(P\oplus\bigoplus_{i\in R}M_{i+1,t-1}),\quad
    \mathcal{B}=\operatorname{add}(P\oplus\bigoplus_{\substack{i-1\notin R,\\ i+l-1\notin R, \\ 1\le l< t}}M_{i,l}).
  \end{equation*}
  By \Cref{lem_sinj_cx}, $(\mathcal{C},\mathcal{X})$ is a non-hereditary complete cotorsion pair; and by \Cref{lem_sinj_wb}, $(\mathcal{W},\mathcal{B})$ is an injective cotorsion pair;
  moreover, $\mathcal{X}\subseteq\mathcal{W}$.
  Hence, by \Cref{theorem_merge}, $(\mathcal{C}, \ (\mathcal{C}\cap\mathcal{W})^\perp, \ \mathcal{W})$ is a \textup{Hovey} triple.
  We claim that $\mathcal F = (\mathcal C\cap\mathcal W)^\perp$. In fact, by definition $
    \mathcal{C}\cap\mathcal{W}
    =\operatorname{add}(P\oplus\bigoplus\limits_{i\in R} M_{i+1,t-1}).$
  Thus $( \mathcal{C}\cap\mathcal{W})^\perp =  \bigcap\limits_{i \in R} (M_{i+1, t-1})^\perp.$
  Let $M_{j,l}$ be a non-projective module.  As we have already seen in (\ref{eqn2wperp}), $M_{j,l}\in (M_{i+1, t-1})^\perp$ with $i\in R$ if and only if
  $j\ne i+t-l+1.$ Thus
  \begin{equation*}
    (\mathcal C\cap\mathcal W)^\perp
    =\operatorname{add}(P\oplus\bigoplus_{\substack {i+l-1\notin R, \\ 1\le l< t}}M_{i,l})
    =\mathcal F.
  \end{equation*}
  Hence  $(\mathcal C,\mathcal F,\mathcal W)$ is a \textup{Hovey} triple.
  By construction
  \begin{equation*}
    \mathcal{C}\cap\mathcal{F}
    =\operatorname{add}(P\oplus\bigoplus_{\substack{i\notin R,\\ i+l-1\notin R, \\ 1\le l< t}}M_{i,l}), \quad \mathcal{C}\cap\mathcal{F}\cap\mathcal{W}
    =\operatorname{add}(P\oplus\bigoplus_{i\in R}M_{i+1,t-1}).
  \end{equation*}

  We claim that $M_{i+1,1}$ is an injective object in the exact category $\mathcal C\cap\mathcal F$ for any $i\in R$.
  In fact, since $i+1\notin R$, one has $M_{i+1,1}\in\mathcal C\cap \mathcal F$.
  It suffices to show that $\mathrm{Ext}^1_A(M_{j,l},M_{i+1,1})=0$ for every non-projective module $M_{j,l}\in\mathcal C\cap\mathcal F$.
  Using \begin{align*}
    \mathrm{Ext}^1_A(M_{j,l},M_{i+1,1})
    \cong \underline{\mathrm{Hom}}_A(\Omega M_{j,l},M_{i+1,1}) \cong \underline{\mathrm{Hom}}_A(M_{j+l,t-l},M_{i+1,1}),
  \end{align*}
  it suffices to see $\underline{\mathrm{Hom}}_A(M_{j+l,t-l},M_{i+1,1})=0$.
  By the construction of $\mathcal{F}$, one has $j+l-1\notin R$. Since $i\in R$, one has $j+l \neq i+1$. It follows from \Cref{simphom} that $\underline{\mathrm{Hom}}_A(M_{j+l,t-l},M_{i+1,1})=0$.
  This proves the claim.

  \vskip5pt

  However, since $t\ge 3$, $M_{i+1,1}$ with $i\in R$ is not in $\mathcal C\cap\mathcal F\cap\mathcal W$.
  It follows that $(\mathcal{C}\cap\mathcal{F}, \ \mathcal{C}\cap\mathcal{F}\cap\mathcal{W})$ is not a Frobenius pair. This completes the proof.
\end{proof}

\subsection{Proof of Theorem \ref{thm_classification}}  If $\gcd(n, t) \ge 2$ and $t \geq 3$,  then by \Cref{lem_sinj_hovey_exceptional} one sees that $A\mbox{-}{\rm mod}$ admits an exceptional \textup{Hovey} triple.

\vskip5pt

Conversely, if $A\mbox{-}{\rm mod}$ admits an exceptional \textup{Hovey} triple  $(\mathcal{C}, \mathcal{F}, \mathcal{W})$ (in particular, $(\mathcal{C}\cap\mathcal{F}, \ \mathcal{C}\cap\mathcal{F}\cap\mathcal{W})$ is not a Frobenius pair), then  $\gcd(n, t) > 1$ and $t \geq 3$. Otherwise,
$\gcd(n, t) = 1$ or $t = 2$. By \Cref{lem_hovey_coprime_or_rad_square_zero},
$(\mathcal{C}, \mathcal{F}, \mathcal{W})$ is either $(A\mbox{-}{\rm mod}, \  A\mbox{-}{\rm mod}, \ \mathcal{P}(A))$, or of the form $(\mathcal{C}, \ \mathcal{F}, \ A\mbox{-}{\rm mod})$ for some complete cotorsion pair $(\mathcal{C}, \mathcal{F})$.

\vskip5pt

In the first case,  $(\mathcal{C}\cap\mathcal{F}, \ \mathcal{C}\cap\mathcal{F}\cap\mathcal{W}) = (A\mbox{-}{\rm mod}, \mathcal{P}(A))$ is a Frobenius pair, which contradicts the assumption.
In the second case,  $(\mathcal{C}\cap\mathcal{F}, \ \mathcal{C}\cap\mathcal{F}\cap\mathcal{W}) = (\mathcal{C}\cap \mathcal{F}, \mathcal{C}\cap \mathcal{F})$ is also a Frobenius pair, which contradicts the assumption. \hfill $\square$

\subsection{The Type of the exceptional \textup{Hovey} triple in \Cref{lem_sinj_hovey_exceptional}}

$\,$

\vskip5pt

\begin{theorem}\label{type} \ Let $(\mathcal{C}, \mathcal{F}, \mathcal{W})$ be the exceptional \textup{Hovey} triple
  in \Cref{lem_sinj_hovey_exceptional}. Then
  \vskip5pt
  $(1)$ \ $(\mathcal C,\mathcal F,\mathcal W)$ is of \textup{Type I} if and only if $d=\gcd(n,t)=2$,
  \vskip5pt
  $(2)$ \ $(\mathcal C,\mathcal F,\mathcal W)$ is of \textup{Type II} if and only if $d=\gcd(n,t)\ge 3$.
\end{theorem}

\begin{proof} \  By Corollary \ref{corollary_D_enough_PI}, $\mathcal C\cap\mathcal F$ has enough projective objects and  enough injective objects.
  By \Cref{lem_sinj_hovey_exceptional} one has
  \begin{equation*}
    \mathcal{C}\cap\mathcal{F}
    =\operatorname{add}(P\oplus\bigoplus_{\substack{i\notin R,\\ i+l-1\notin R,\\ 1\le l<t}}M_{i,l}), \quad \mathcal{C}\cap\mathcal{F}\cap\mathcal{W}
    =\operatorname{add}(P\oplus\bigoplus_{i\in R}M_{i+1,t-1}).
  \end{equation*}
  We will prove $\mathcal P(\mathcal{C} \cap \mathcal{F})
    =\operatorname{add}\left(P\oplus
    \bigoplus\limits_{i\in R}\big(M_{i+1,t-1}\oplus M_{i-1,1}\big)\right)$.   Since $\mathcal{C} \cap \mathcal{F}$ is extension-closed in $A\mbox{-}{\rm mod}$, one has $\mathrm{Ext}^1_{\mathcal{C} \cap \mathcal{F}}(M,N)=\mathrm{Ext}^1_A(M,N)$ for all $M,N\in\mathcal{C}\cap\mathcal{F}$.
  Hence
  $\mathcal{P}(\mathcal{C} \cap \mathcal{F})
    ={}^\perp(\mathcal{C}\cap\mathcal{F})\cap(\mathcal{C}\cap\mathcal{F})$.
  \vskip5pt
  We first prove that
  \begin{equation*}
    \operatorname{add}\left(P\oplus\bigoplus_{i\in R}
    \big(M_{i+1,t-1}\oplus M_{i-1,1}\big)\right)\subseteq\mathcal P(\mathcal{C} \cap \mathcal{F}).
  \end{equation*}
  Since
  \begin{equation*}
    \operatorname{add}\left(P\oplus\bigoplus_{i\in R}M_{i+1,t-1}\right) = \mathcal{C}\cap\mathcal{F}\cap\mathcal{W}
    \subseteq\mathcal P(\mathcal C\cap\mathcal F),
  \end{equation*}

  \vskip5pt
  \noindent it remains to show that $M_{i-1,1}\in\mathcal P(\mathcal C\cap\mathcal F)$ for  $i\in R$.
  Since $i\in R$ and $d\ge 2$, one has $i-1\notin R$, and hence $M_{i-1,1}\in\mathcal C\cap\mathcal F$.  Let $M_{j,l}$ be a non-projective  module in $\mathcal C\cap\mathcal F$.  Then $j\notin R$. By \Cref{extinfrobenius} one has
  \begin{align*}
    \mathrm{Ext}^1_A(M_{i-1,1},M_{j,l})
    \cong\underline{\mathrm{Hom}}_A(M_{i-1,1},\mho M_{j,l})\cong\underline{\mathrm{Hom}}_A(M_{i-1,1},M_{j+l-t,t-l})
  \end{align*}
  and it is zero if and only if $j+l-t \neq (i-1)-(t-l)+1$ (cf. \Cref{simphom}), i.e., $i \neq j$.  This is true since $i\in R$ and $j\notin R$. Thus, $M_{i-1,1}\in\mathcal P(\mathcal C\cap\mathcal F)$ for  $i\in R$.

  \vskip5pt

  To prove $\mathcal{P}(\mathcal{C} \cap \mathcal{F}) \subseteq \operatorname{add}\left(P\oplus\bigoplus\limits_{i\in R}
    \big(M_{i+1,t-1}\oplus M_{i-1,1}\big)\right)$, it suffices to prove

  \vskip5pt

  {\bf Claim} \ Let $M_{j,l}\in\mathcal C\cap\mathcal F$ be a non-projective $A$-module.  If
  \begin{equation}\label{eqn_sinj_exceptional_type}
    M_{j,l}\notin\{M_{i+1,t-1},M_{i-1,1}\mid i\in R\},
  \end{equation}
  then there exists $B\in\mathcal C\cap\mathcal F$ such that $\mathrm{Ext}^1_A(M_{j,l},B)\neq0$.

  \vskip5pt

  We divide the proof of {\bf Claim} into the following two cases.

  \vskip5pt

  {\bf Case 1} \ Suppose that $j+l\notin R$.  Put $B=M_{j+l,1}$.  Since $j+l\notin R$ and $(j+l)+1-1 \notin R$, one has $B\in\mathcal C\cap\mathcal F$.  There is a canonical epimorphism
  \begin{equation*}
    p:\Omega M_{j,l}=M_{j+l,t-l}\twoheadrightarrow M_{j+l,1}=B
  \end{equation*}
  between non-projective indecomposable modules.
  Thus, by \Cref{radical}, $p$ does not factor through a projective module.  It follows from \Cref{extinfrobenius} that
  \begin{equation*}
    \mathrm{Ext}^1_A(M_{j,l},B)
    \cong\underline{\mathrm{Hom}}_A(M_{j+l,t-l},B)\neq0.
  \end{equation*}

  \vskip5pt

  {\bf Case 2} \ Suppose that $j+l\in R$.  We first claim that $2\leq l\leq t-2$.  Indeed, if $l=1$, then $j+1=j+l\in R$,
  and hence $M_{j,l}=M_{(j+1)-1,1}$ with  $(j+1)\in R$, which contradicts the assumption \Cref{eqn_sinj_exceptional_type};
  if $l=t-1$, then $j+l\in R$ implies $j-1 = j+t-1\in R \pmod n$, and hence $M_{j,l}=M_{(j-1)+1,t-1}$ with  $(j-1)\in R$, which again contradicts \Cref{eqn_sinj_exceptional_type}.
  \vskip5pt
  Put $B=M_{j+l-1,3}$.  Since $j+l\in R$ and $d>1$, one has $(j+l-1)\notin R$ and $(j+l-1)+3-1\notin R$, and hence $B\in\mathcal C\cap\mathcal F$.  Since $t-l\geq2$, the composition
  \begin{equation*}
    f:M_{j+l,t-l}\twoheadrightarrow M_{j+l,2}=\operatorname{rad}B \hookrightarrow B
  \end{equation*}
  is non-zero.
  We claim that $f$ does not factor through a projective module.  Otherwise, $f$ factors through the injective envelope $\sigma:M_{j+l,t-l}\hookrightarrow M_{j,t}$.
  Thus $f=g\circ\sigma$ for some $g:M_{j,t}\longrightarrow B$.  Since $M_{j+l,t-l}=\operatorname{rad}^lM_{j,t}$ with $l\geq2$, it follows that
  \begin{equation*}
    f(M_{j+l,t-l}) = g(M_{j+l,t-l}) =g(\operatorname{rad}^lM_{j,t})
    \subseteq\operatorname{rad}^lg(M_{j,t})
    \subseteq\operatorname{rad}^2B,
  \end{equation*}
  which contradicts $f(M_{j+l,t-l})=\operatorname{rad}B$.  Therefore $f$ does not factor through a projective module.  By \Cref{extinfrobenius},
  \begin{equation*}
    \mathrm{Ext}^1_A(M_{j,l}, B)\cong \underline{\mathrm{Hom}}_A(\Omega M_{j, l},B)
    \cong\underline{\mathrm{Hom}}_A(M_{j+l,t-l},B)\neq0.
  \end{equation*}

  \vskip5pt

  Thus {\bf Claim } is proved.  Therefore
  \begin{equation}\label{eq_projectives_sinj_exceptional}
    \mathcal P(\mathcal C\cap\mathcal F)
    =\operatorname{add}\left(P\oplus\bigoplus_{i\in R}
    \big(M_{i+1,t-1}\oplus M_{i-1,1}\big)\right).
  \end{equation}

  \vskip5pt

  Dually, one can prove
  \begin{equation}\label{eq_injectives_sinj_exceptional}
    \mathcal I(\mathcal C\cap\mathcal F)
    =\operatorname{add}\left(P\oplus\bigoplus\limits_{i\in R}
    \big(M_{i+1,t-1}\oplus M_{i+1,1}\big)\right).
  \end{equation}

  \vskip5pt

  By definition,  $(\mathcal C,\mathcal F,\mathcal W)$ is of \textup{Type I} if and only if these two sets are equal,  if and only if
  $\{i-1 \ | \ i\in R= d\mathbb Z/n\mathbb Z\} = \{i+1 \ | \ i\in R\}$,  i.e.,  $d =2$.
  This completes the proof. \end{proof}

\subsection{Examples} \ We present two simple examples of exceptional exact model structures, via Theorem \ref{type}.

\vskip5pt

\begin{example} \label{exm1} \ Take $n =2$ and $t=4$ in Theorem \ref{type}. Then $d = 2$ and $R = 2\mathbb Z/2\mathbb Z = \{0\} \subseteq \mathbb Z/2\mathbb Z$. Thus $A$ is the selfinjective Nakayama algebra $kC_2/J^4$.
  The Auslander-Reiten quiver of $A$ is

  $$
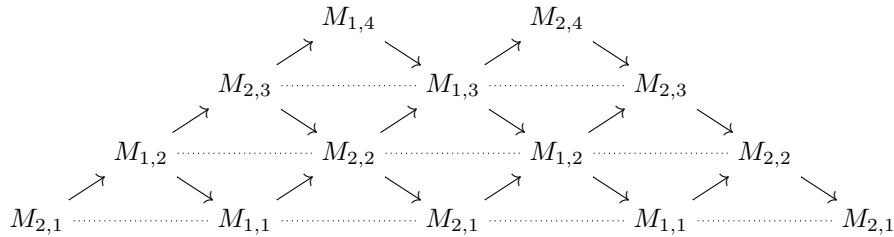
\xymatrix @R= 0.3cm @C=0.4cm {&&& M_{1,4}\ar[dr] && M_{2,4}\ar[dr] & &\\
    &&M_{2,3}\ar[ur]\ar[dr] & & M_{1,3}\ar[dr]\ar[ur]\ar@{.}[ll] & & M_{2,3}\ar[dr]\ar@{.}[ll] & \\
    & M_{1,2}\ar[dr]\ar[ur] & & M_{2,2}\ar[ur]\ar[dr]\ar@{.}[ll] & & M_{1,2}\ar[ur]\ar[dr]\ar@{.}[ll] & & M_{2,2}\ar[dr]\ar@{.}[ll]\\
    M_{2,1}\ar[ur] & &M_{1,1}\ar[ur]\ar@{.}[ll] & &M_{2,1}\ar[ur]\ar@{.}[ll] & & M_{1,1}\ar[ur]\ar@{.}[ll] && M_{2,1}\ar@{.}[ll]}$$

  \vskip5pt

  \noindent By Theorem \ref{type}(1),  $(\mathcal{C}, \mathcal{F}, \mathcal{W})$ is an exceptional \textup{Hovey} triple of \textup{Type I} in $A$-mod, namely,
  $(\mathcal{C}, \mathcal{F}, \mathcal{W})$ is a \textup{Hovey} triple such that $\mathcal{C} \cap \mathcal{F}$ is a Frobenius category, but
  $\mathcal{C} \cap \mathcal{F} \cap \mathcal{W} \subsetneqq \mathcal{P}(\mathcal{C} \cap \mathcal{F}) = \mathcal{C} \cap \mathcal{F}$, where
  \begin{align*}
    \mathcal{C} & = \operatorname{add}(M_{1,1}\oplus M_{1,2}\oplus M_{1,3}\oplus M_{1,4}\oplus M_{2,4}), \\
    \mathcal{F} & = \operatorname{add}(M_{1,1}\oplus M_{1,3}\oplus M_{2,2}\oplus M_{1,4}\oplus M_{2,4}), \\
    \mathcal{W} & = \operatorname{add}(M_{2,1}\oplus M_{1,3}\oplus M_{1,4}\oplus M_{2,4}).
  \end{align*}
\end{example}

\vskip5pt

\begin{example} \label{exm2} \ Take $n =3$ and $t=3$ in Theorem \ref{type}. Then $d  = 3$ and $R = 3\mathbb Z/3\mathbb Z = \{0\} \subseteq \mathbb Z/3\mathbb Z$. Thus $A$ is the selfinjective Nakayama algebra $kC_3/J^3$.
  The Auslander-Reiten quiver of $A$ is

  $$
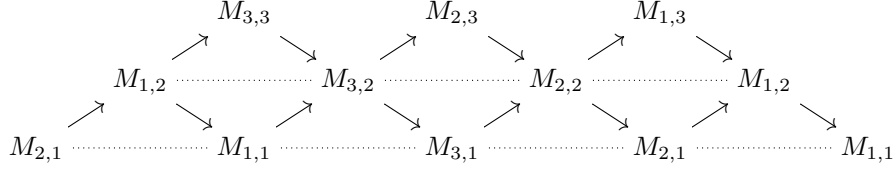
\xymatrix @R= 0.3cm @C=0.4cm {&&M_{3,3}\ar[dr] & & M_{2,3}\ar[dr]& & M_{1,3}\ar[dr] & \\
    & M_{1,2}\ar[dr]\ar[ur] & & M_{3,2}\ar[ur]\ar[dr]\ar@{.}[ll] & & M_{2,2}\ar[ur]\ar[dr]\ar@{.}[ll] & & M_{1,2}\ar[dr]\ar@{.}[ll]\\
    M_{2,1}\ar[ur] & &M_{1,1}\ar[ur]\ar@{.}[ll] & &M_{3,1}\ar[ur]\ar@{.}[ll] & & M_{2,1}\ar[ur]\ar@{.}[ll] && M_{1,1}\ar@{.}[ll]}$$

  \vskip5pt

  \noindent By Theorem \ref{type}(2),  $(\mathcal{C}, \mathcal{F}, \mathcal{W})$ is an exceptional \textup{Hovey} triple of \textup{Type II} in $A$-mod, namely,
  $(\mathcal{C}, \mathcal{F}, \mathcal{W})$ is a \textup{Hovey} triple such that $\mathcal{C} \cap \mathcal{F}$ is not a Frobenius category $($with respect to the canonical exact structure$)$, where

  \begin{align*}
    \mathcal{C} & = \operatorname{add}(M_{1,1}\oplus M_{1,2}\oplus M_{2,1}\oplus M_{2,2}\oplus M_{1,3}\oplus M_{2,3}\oplus M_{3,3}), \\
    \mathcal{F} & = \operatorname{add}(M_{1,1}\oplus M_{1,2}\oplus M_{2,1}\oplus M_{3,2}\oplus M_{1,3}\oplus M_{2,3}\oplus M_{3,3}), \\
    \mathcal{W} & = \operatorname{add}(M_{3,1}\oplus M_{1,2}\oplus M_{1,3}\oplus M_{2,3}\oplus M_{3,3}).
  \end{align*}

\end{example}

\section{\bf Exceptional \textup{Hovey} triples from Nakayama algebras with a unique maximal module}\label{sec_m_nak_hovey}

In this section, we deal with a class of non-selfinjective Nakayama algebras $A$ such that $A$ admits a unique maximal module.
Let $A$ be a Nakayama algebra without a simple projective module. An indecomposable $A$-module $M$ is {\it maximal} if its length is maximal, i.e.,  $l(M)\ge l(N)$ for any indecomposable $A$-module $N$.
Then the condition that $A$ admits a unique maximal module $M$ is equivalent to each of the following conditions:

(1) \ $M$ is indecomposable and projective, and each indecomposable projective $A$-module is a submodule of $M$;

(1') \ $M$ is indecomposable and injective, and each indecomposable injective $A$-module is a quotient of $M$;

(2) \ $M$ is unique indecomposable projective-injective;

(3) \ there is a unique {\it minimal projective} $A$-module $P$ (an indecomposable projective $A$-module $P$ is minimal if any non-zero proper submodule of $P$ is not projective);

(3') \ there is a unique {\it minimal injective} $A$-module $I$ (an indecomposable injective $A$-module $I$ is minimal if any non-zero proper quotient of $I$ is not injective);

(4) \ $A$ has a unique zero relation;

(5) \ The Kupisch series of $A$  has a unique maximal component.

\vskip5pt

Throughout this section, let $C_n$ be the cyclic quiver
\begin{equation*}
  C_n=\xymatrix{
  1 \ar@{->}[r]^{a_2} & 2 \ar@{->}[r]^{a_3} & 3 \ar@{->}[r]^{a_4} & \cdots \ar@{->}[r]^{a_n} & n \ar@/^1pc/@{->}[llll]^{a_1}
  }
\end{equation*}
where $n\ge2$, and the indices are taken in $\mathbb Z/n\mathbb Z$. Set
$A=kC_n/\langle a_t\cdots a_2a_1\rangle$, where $t$ may be greater than $n$.
Then $A$ is a Nakayama algebra without simple projective modules and admits a unique maximal module.
The Kupisch series of $A$ is $(c_1, \cdots, c_n)$ with $c_i=t+n-i$ for $1\le i\le n$.
The indecomposable $A$-modules are $M_{i,l}$, where $i\in \mathbb Z/n\mathbb Z$ and $1 \leq l \leq c_i$.
The indecomposable projective modules are $P_i=M_{i,c_i}$ for $1\le i\le n$,  and the indecomposable injective modules are $I_i=M_{1,t+i}$ for $1\leq i\leq n-1$ and $I_n=M_{1,t}$.
The unique indecomposable projective-injective module is $P_1 = I_{n-1} = M_{1,t+n-1}$,  which is the unique maximal module.
The Auslander-Reiten quiver of $A$ looks like (the first lower indices are taken in $\mathbb Z / n \mathbb Z$)
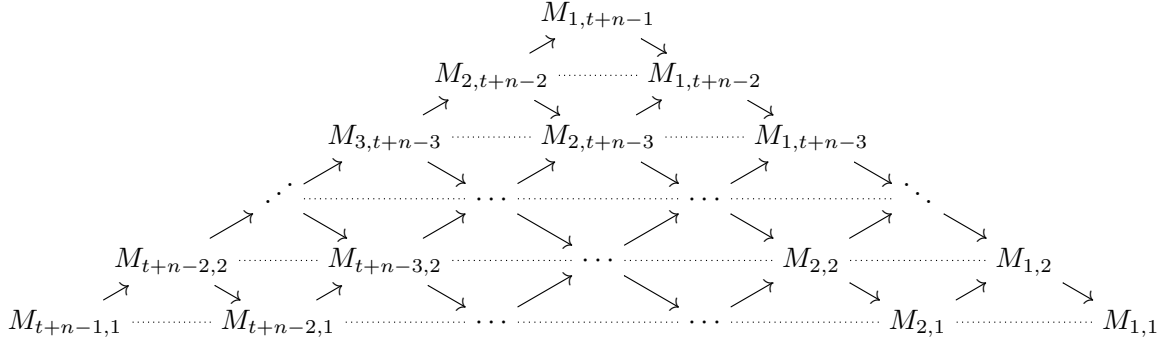
\begin{equation*}
  \resizebox{1\linewidth}{!}{$
    \xymatrix@!0@C=3.8em@R=2.2em{
    &  &  &  &  & M_{1,t+n-1} \ar@{->}[rd] &  &  &  &  &  \\
    &  &  &  & M_{2,t+n-2} \ar@{->}[ru] \ar@{->}[rd] &  & M_{1,t+n-2} \ar@{->}[rd] \ar@{.}[ll] &  &  &  &  \\
    &  &  & M_{3,t+n-3} \ar@{->}[ru] \ar@{->}[rd] &  & M_{2,t+n-3} \ar@{->}[ru] \ar@{->}[rd] \ar@{.}[ll] &  & M_{1,t+n-3} \ar@{.}[ll] \ar@{->}[rd] &  &  &  \\
    &  & \iddots \ar@{->}[ru] \ar@{->}[rd] &  & \cdots \ar@{->}[ru] \ar@{.}[ll] \ar@{->}[rd] &  & \cdots \ar@{->}[ru] \ar@{.}[ll] \ar@{->}[rd] &  & \ddots \ar@{.}[ll] \ar@{->}[rd] &  &  \\
    & M_{t+n-2,2} \ar@{->}[rd] \ar@{->}[ru] &  & M_{t+n-3,2} \ar@{->}[ru] \ar@{.}[ll] \ar@{->}[rd] &  & \cdots \ar@{->}[ru] \ar@{->}[rd] \ar@{.}[ll] &  & M_{2,2} \ar@{->}[rd] \ar@{->}[ru] \ar@{.}[ll] &  & M_{1,2} \ar@{->}[rd] \ar@{.}[ll] &  \\
    M_{t+n-1,1} \ar@{->}[ru] &  & M_{t+n-2,1} \ar@{.}[ll] \ar@{->}[ru] &  & \cdots \ar@{.}[ll] \ar@{->}[ru] &  & \cdots \ar@{->}[ru] \ar@{.}[ll] &  & M_{2,1} \ar@{->}[ru] \ar@{.}[ll] &  & M_{1,1} \ar@{.}[ll]
    }
  $}
\end{equation*}

\vskip10pt

\begin{theorem}\label{thm_m_nak_classification} \ Let $A$ be a Nakayama algebra of the form $A=kC_n/\langle a_t\cdots a_2a_1\rangle$.
  Then $A\mbox{-}{\rm mod}$ admits an exceptional \textup{Hovey} triple if and only if $t=dn$ with $d\geq2$.
\end{theorem}

\vskip5pt

When $t = dn$ for $d\geq2$, we will explicitly construct an exceptional \textup{Hovey} triple of \textup{Type I} in $A\mbox{-}{\rm mod}$.

\subsection{Thick subcategories of $A\mbox{-}{\rm mod}$}

\begin{proposition}\label{thm_m_nak_thick} \ Suppose that $n \nmid t$ or $t = n$. Then a thick subcategory $\mathcal W$ of $A\mbox{-}{\rm mod}$ containing $\mathcal P(A)\cup\mathcal I(A)$ is unique, i.e., $\mathcal W = A\mbox{-}{\rm mod}$.
\end{proposition}

\begin{proof} \ For each $1\leq i\leq n-1$, the  short exact sequence
  \[
    0\longrightarrow P_{i+1}\longrightarrow P_i\longrightarrow S_i\longrightarrow0.
  \]
  implies that $S_i\in \mathcal W$.
  On the other hand, since ${\rm soc}(M_{1, t+i}) = S_{t+i}$, there is a short exact sequence
  \[
    0\longrightarrow S_{t+i}\longrightarrow M_{1, t+i}\longrightarrow M_{1, t+i-1}\longrightarrow 0.\]
  For the Nakayama algebra $A$ under consideration, $M_{1,t+i}$ and $M_{1,t+i-1}$ are injective modules for $1\le i\le n-1$ and hence belong to $\mathcal W$.
  Since $\mathcal W$ is thick, $S_{t+i}\in\mathcal W$ for $1\le i\le n-1$, where the indices are taken in $\mathbb Z/n\mathbb Z$.

  \vskip5pt

  Suppose that $n \nmid t$.
  Then $\{S_i \ \mid  1\le i\le n-1\} \cup \{S_{t+i} \ \mid \ 1\le i \le n-1\}$ contains all the simple modules,
  and hence all the simple modules are in $\mathcal{W}$.
  Thus $\mathcal{W} = A\mbox{-}{\rm mod}$.
  \vskip5pt

  Suppose that $t=n$. By the argument above $S_1,\ldots,S_{n-1}$ belong to $\mathcal W$. Since $M_{1,n-1}$ has composition factors $S_1, \cdots, S_{n-1}$,
  $M_{1,n-1}\in\mathcal W$. Then the short exact sequence
  \[
    0\longrightarrow M_{1,n-1}\longrightarrow P_n=M_{n,n}\longrightarrow S_n\longrightarrow0
  \]
  implies that $S_n\in\mathcal W$. Therefore $\mathcal{W} = A\mbox{-}{\rm mod}$.
\end{proof}

By Fact \ref{Wcontainsp} and Proposition \ref{thm_m_nak_thick}, one gets

\vskip5pt

\begin{corollary}\label{cor_m_nak_hovey_triple} \ If $n \nmid t$ or $t = n$, then any \textup{Hovey} triple in $A\mbox{-}{\rm mod}$ is of the form $(\mathcal{C}, \mathcal{F}, A\mbox{-}{\rm mod})$,
  where  $(\mathcal{C}, \mathcal{F})$ is a complete cotorsion pair in $A\mbox{-}{\rm mod}$.
\end{corollary}

\subsection{Complete cotorsion pairs in $A\mbox{-}{\rm mod}$}
$\,$

\vskip5pt

For convenience,  denote by $I$ the direct sum of all the indecomposable injective $A$-modules, i.e., $I = \bigoplus\limits_{1\le i\le n} I_i = M_{1, t+1}\oplus \cdots \oplus M_{1, t+n-1}\oplus M_{1, t}$.

\vskip5pt

\begin{lemma}\label{lem_m_nak_injective_cx}
  Suppose $t = d n$ for some $d \geq 2$.
  Put
  $$\mathcal{C} = \operatorname{add}\left(\bigoplus_{1 \leq l \leq t}M_{n,l}\oplus\bigoplus_{\substack {1 \leq i \leq n-1 \\ l \notin [n-i, t-i]}} M_{i,l}
    \right), \ \ \ \ \ \  \mathcal{X} = \operatorname{add}(
    M_{n,t-n+1}\oplus I).$$

  \vskip5pt
  \noindent
  Then $(\mathcal{C}, \mathcal{X})$ is a non-hereditary complete cotorsion pair.
\end{lemma}

\begin{proof} \ Since $M_{i,l}$ satisfies $l\le c_i = t+n-i$,  one can explicitly rewrite

  $$\mathcal{C} = \operatorname{add}\left(\bigoplus_{1 \leq l \leq t}M_{n,l}\oplus\bigoplus_{\substack {1 \leq i \leq n-1 \\ 1\leq l \leq n-i-1}} M_{i,l}
    \oplus
    \bigoplus_{\substack {1 \leq i \leq n-1 \\ t+1-i\leq l\leq t+n-i}} M_{i,l}
    \right)$$

  \vskip10pt
  \noindent
  where if $i=n-1$, then $\bigoplus\limits_{1\leq l \leq n-i-1} M_{i,l}$ above is understood to be $0$.

  \vskip5pt

  Using the Auslander--Reiten formula, one has
  \begin{align*}
    {}^\perp\mathcal{X} = {}^\perp (M_{n,t-n+1}) = \{Y \mid \underline{\operatorname{Hom}}_A(\tau^{-1}M_{n,t-n+1},Y)=0 \}= \{Y \mid \underline{\operatorname{Hom}}_A(M_{n-1,t-n+1},Y)=0 \}.
  \end{align*}
  To show $\mathcal{C} = {}^\perp \mathcal{X}$, it suffices to show
  \vskip5pt

  \textbf{Claim 1} \
  For any non-projective indecomposable module $M_{i,l}$, one has
  \[
    \underline{\operatorname{Hom}}_A(M_{n-1,t-n+1},M_{i,l})=0
  \]
  if and only if either $i=n$, or $1\leq i\leq n-1$ with $l \notin [n-i, t-i]$.

  \vskip5pt

  \textup{Case 1.} \ Suppose that $i=n$.
  Let $f: M_{n-1,t-n+1}\longrightarrow M_{n,l}$ be a homomorphism.  If $l<n$, then $S_{n-1}$ is not a composition factor of $M_{n,l}$, and hence $f=0$.
  \vskip5pt
  Suppose that $n\leq l\leq t$ and $f\neq0$.
  Since $M_{n-1,t-n+1}$ is uniserial,  ${\rm Im} f$ is of the form $M_{n-1, j}$. Since ${\rm soc}({\rm Im} f) = {\rm soc}(M_{n, l}) = S_{l-1}$, one has $j-2 = l-1 \ ({\rm mod} \ n)$, i.e., $j = l+1 \ ({\rm mod} \ n)$.
  Thus there is a positive integer $s$ such that $f= e\circ p$, where
  $p: M_{n-1,t-n+1}\twoheadrightarrow {\rm Im} f = M_{n-1, l+1-sn}$ and  $e:  M_{n-1, l+1-sn}\hookrightarrow M_{n, l}$.
  Consider the following commutative square (cf. Fact \ref{pp}),  where $\pi$ is the projective cover
  \[
    \xymatrix@R=.7cm{
    & M_{n-1,t+1-sn} \ar@{->>}[d]_-{\pi'} \ar@{^{(}->}[r]^-{e'} & M_{n,t} \ar@{=}[r] \ar@{->>}[d]^-{\pi} & P_n \\
    M_{n-1,t-n+1} \ar@{->>}[r]^-{p} & M_{n-1,l+1-sn} \ar@{^{(}->}[r]^-{e} & M_{n,l} &
    }
  \]
  Since $t-n+1 \geq t+1-sn$, there is a canonical homomorphism $p' : M_{n-1,t-n+1} \twoheadrightarrow M_{n-1,t+1-sn}$ such that $p = \pi' \circ p'$ (cf. Lemma \ref{factorthrough}).  Hence $f = e \circ p = e \circ \pi'\circ p'$ factors through the projective module $P_n$.
  This proves $\underline{\operatorname{Hom}}_A(M_{n-1,t-n+1},M_{n,l})=0$.

  \vskip5pt

  \textup{Case 2.} \ Suppose that $1\leq i\leq n-1$ and $1\leq l\leq n-i-1$.
  Since $S_{n-1}$ is not a composition factor of $M_{i,l}$, one has $\operatorname{Hom}_A(M_{n-1,t-n+1},M_{i,l})=0$.

  \vskip5pt

  \textup{Case 3.} \ Suppose that $1\leq i\leq n-1$ and $t+1-i\leq l\leq t+n-i$.
  By the Auslander-Reiten formula
  \begin{align*}
    D\underline{\operatorname{Hom}}_A(M_{n-1,t-n+1},M_{i,l})
    \cong \operatorname{Ext}^1_A(M_{i,l},M_{n,t-n+1})\cong D\overline{\operatorname{Hom}}_A(M_{n,t-n+1},M_{i+1,l}),
  \end{align*}
  it suffices to show that every homomorphism $0\ne f: M_{n,t-n+1}\longrightarrow M_{i+1,l}$ factors through an injective module.
  Since ${\rm Im} f = M_{n, j}$ and ${\rm soc}({\rm Im} f) = {\rm soc}(M_{i+1, l})$, one has $j = l+i+1 \ ({\rm mod} \ n)$.
  Thus there is an integer $s\ge 0$ such that $f= e\circ p$,  where
  \[
    p:M_{n,t-n+1}\twoheadrightarrow
    M_{n,l+i+1-n-sn}, \ \ \ e: M_{n,l+i+1-n-sn}\hookrightarrow
    M_{i+1,l}.
  \]
  The case $s=0$ is impossible, since  $t-n+1 < l+i+1-n$.
  Hence $s\geq1$.
  Let $\sigma:M_{n,t-n+1}\hookrightarrow M_{1,t}$ be the injective envelope. Consider the following canonical commutative diagram (cf. Fact \ref{pp}):
  \[
    \xymatrix@R=.7cm{
    &M_{n,t-n+1} \ar@{->>}[r]^-{p} \ar@{^{(}->}[d]_{\sigma}
    & M_{n,l+i+1-n-sn} \ar@{^{(}->}[d]^{\sigma'} \ar@{^{(}->}[r]^-{e}
    & M_{i+1,l} \\
    I_n\ar@{=}[r]&M_{1,t}\ar@{>>}[r]^-{p'}
    & M_{1,l+i-sn}
    &
    }
  \]
  Since $\mathrm{soc}(M_{1,l+i-sn})=S_{i+l}=\mathrm{soc}(M_{i+1,l})$ and $l+i-sn\leq l$, there is $e':M_{1,l+i-sn}\longrightarrow M_{i+1,l}$ such that $e=e'\circ\sigma'$ (cf. Lemma \ref{factorthrough}).
  Thus $f = e \circ p = e' \circ \sigma' \circ p$ factors through the injective module $I_n$.
  Hence
  \[
    \underline{\operatorname{Hom}}_A(M_{n-1,t-n+1},M_{i,l})
    \cong \overline{\operatorname{Hom}}_A(M_{n,t-n+1},M_{i+1,l}) = 0.
  \]

  \vskip5pt

  \textup{Case 4.} \ Suppose that $1\leq i\leq n-1$ and $n-i\leq l\leq t-i$.
  Since $1\leq l+i-n+1\leq t-n+1$, one has the canonical homomorphisms
  \[M_{n-1,t-n+1}\stackrel p \twoheadrightarrow
    M_{n-1,l+i-n+1}
    =\operatorname{rad}^{n-i-1}(M_{i,l})\stackrel e
    \hookrightarrow M_{i,l}.
  \]
  We claim that $f = e\circ p$ does not factor through a projective module.  Otherwise, $f$ factors through the projective cover $q:P_i=M_{i,t+n-i}\twoheadrightarrow M_{i,l}$.
  Then there is $h:M_{n-1,t-n+1}\longrightarrow P_i$ such that $f=q\circ h$. Consider the following pullback square of $q$ and $e$ (cf. Fact \ref{pp}):
  \begin{equation*}
    \xymatrix@R=.7cm{
    M_{n-1,t-n+1} \ar@/_/@{->>}[rdd]_-{p} \ar@/^/@{.>}[rrd]^{h} \ar@{.>}[rd]^{\varphi} &  &  &  \\
    & M_{n-1,t+1} \ar@{.>>}[d]^{\pi'} \ar@{^{(}.>}[r]^{e'} & M_{i,t+n-i} \ar@{=}[r] \ar@{->>}[d]^{q} & P_i \\
    & M_{n-1,l+i-n+1} \ar@{^{(}->}[r]^-{e} & M_{i,l} &
    }
  \end{equation*}
  By the universal property of the pullback, there is a morphism $\varphi$ such that $e'\circ\varphi=h$ and $\pi'\circ\varphi=p$.
  Since  $\mathrm{top}(M_{n-1,t-n+1}) =\mathrm{top}(M_{n-1,t+1})$ and since $p$ and $\pi'$ are epimorphisms, it follows from Nakayama's lemma that $\varphi$ is an epimorphism, which contradicts $t-n+1 < t+1$.

  \vskip5pt

  This proves \textbf{Claim 1}. It remains to prove $\mathcal X=\mathcal C^\perp$, which is equivalent to

  \vskip5pt

  \textbf{Claim 2} \ For any indecomposable non-injective module $M_{i,l} \neq M_{n, t-n+1}$, there exists $C \in \mathcal{C}$ such that $\operatorname{Ext}^1_A(C, M_{i,l}) \neq 0$.

  \vskip5pt

  \textup{Case 1.} \ Suppose that $i\neq n$. Consider $S_{i-1}=M_{i-1,1}$, where $S_0=S_n$. Since $i-1\neq n-1$, one has $S_{i-1}\in\mathcal C$.
  We claim that $M_{i-1,l+1}$ is a well-defined indecomposable $A$-module, namely, $l+1 \le c_{i-1},$ where $c_{i-1} = t+n-(i-1)$ is the length of $P_{i-1}$. Indeed, if $2\leq i\leq n-1$, then $l+1\leq c_i+1=c_{i-1}.$
  If $i=1$, then by assumption $M_{1,l}$ is not injective. Since the indecomposable injective modules are $\{M_{1,t-1+i}\}_{1 \leq i\leq n}$, one has $l\leq t-1$, and hence $l+1\leq t=c_n$.
  Altogether, $M_{i-1,l+1}$ is well-defined. Then the canonical non-splitting short exact sequence
  \[
    0\longrightarrow M_{i,l}\longrightarrow M_{i-1,l+1}
    \longrightarrow S_{i-1}=M_{i-1,1}\longrightarrow0
  \]
  implies that $\operatorname{Ext}^1_A(S_{i-1},M_{i,l})\neq0$.

  \vskip5pt

  \textup{Case 2.} \ Suppose that $i=n$ and $l<t-n+1$.
  By construction  $M_{n,n}\in\mathcal{C}$.
  Since $l+n \leq t = c_n$, $M_{n,l+n}$ is a well-defined $A$-module.
  By the exact sequence
  \[
    0\longrightarrow M_{i,l}\longrightarrow M_{n,l+n}
    \longrightarrow M_{n,n}\longrightarrow0
  \]
  one sees that $\operatorname{Ext}^1_A(M_{n,n}, M_{i,l})\neq0$.

  \vskip5pt

  \textup{Case 3.} \ Suppose that $i=n$ and $t-n+2 \leq l \leq c_n = t$.
  Put $C = M_{n-1,l}$. Since $t-n+2 \leq l \leq t+1$, one has $C \in \mathcal{C}$ by construction.
  Since $l+1 \leq c_{n-1} = t+1$, $M_{n-1, l+1}$ is defined.
  By the non-splitting  short exact sequence (cf. Fact \ref{pp})
  \[
    0\longrightarrow M_{i,l}
    \longrightarrow M_{n,l-1}\oplus M_{n-1,l+1}
    \longrightarrow M_{n-1,l}\longrightarrow0
  \]
  one sees that $\operatorname{Ext}^1_A(M_{n-1,l},M_{i,l})\neq0$.
  \vskip5pt
  This proves {\bf Claim 2}. Thus $(\mathcal C,\mathcal X)$ is a cotorsion pair, and it is complete by \Cref{prop_ctp_1}. Finally, this cotorsion pair is not hereditary.  Indeed, $M_{n,1}\in\mathcal C$, but $\Omega M_{n,1}=M_{1,t-1}\notin\mathcal C$ by the construction of $\mathcal{C}$.  Thus $\mathcal C$ is not closed under syzygies, and hence $(\mathcal C,\mathcal X)$ is not hereditary by \Cref{lem_ctp_2}.
\end{proof}

\begin{lemma}\label{lem_m_nak_injective_wb} \ Suppose $t = d n$ for some $d \geq 2$.
  Put
  $$\mathcal W = \operatorname{add}\left(\bigoplus_{\substack {1 \leq i \leq n \\l \notin [n-i+1, t-i]}}M_{i,l} \right), \ \ \ \ \ \mathcal B = \operatorname{add}\left(
    \bigoplus_{1 \leq r \leq d-1}M_{1,rn}\oplus I \right).$$
  Then $(\mathcal W,\mathcal B)$ is an injective cotorsion pair.
  \vskip5pt
\end{lemma}

\begin{proof}  Since $M_{i,l}$ satisfies $l\le c_i = t+n-i$,  one can explicitly rewrite
  $$\mathcal W
    = \operatorname{add}\left(
    \bigoplus_{1\le i\le n}
    \left(
      \bigoplus_{1\le l\le n-i}M_{i,l}
      \oplus
      \bigoplus_{t-i+1\le l\le t+n-i}M_{i,l}
      \right)
    \right)$$

  \vskip5pt
  \noindent
  where if $i = n$ then $\bigoplus\limits_{1\le l\le n-i}M_{i,l}$ is understood to be $0$.

  \vskip5pt
  We claim \  ${}^\perp \mathcal{B} = {}^\perp M_{1,n}$. In fact, if $d = 2$ then we are done.
  Suppose that $d \geq 3$. For each $1 \leq i \leq d-2$ one has the exact sequence
  \[
    0 \longrightarrow M_{1,in} \longrightarrow M_{1,(i+1)n} \longrightarrow M_{1,n} \longrightarrow 0.
  \]
  Taking $i = 1$, one has ${}^\perp M_{1,n} \subseteq {}^\perp M_{1,2n}$.
  By induction one has \ ${}^\perp M_{1,n} \subseteq {}^\perp M_{1,2n} \subseteq \cdots \subseteq {}^\perp M_{1,(d-1)n},$
  and hence  ${}^\perp\mathcal B = \bigcap\limits_{1 \leq r \leq d-1}{}^\perp M_{1,rn} = {}^\perp M_{1,n}.$

  \vskip10pt

  Using the Auslander--Reiten formula, one has
  \[
    {}^\perp\mathcal B  = {}^\perp M_{1,n} = \{Y \mid \underline{\operatorname{Hom}}_A(\tau ^{-1} M_{1,n},Y)=0\}= \{Y \mid \underline{\operatorname{Hom}}_A(M_{n,n},Y)=0\}.
  \]
  We first prove ${}^\perp\mathcal B=\mathcal W$.  It suffices to prove

  \vskip5pt

  \textbf{Claim 1} \
  For an indecomposable module $M_{i,l}$ with $1\le i\le n$, $\underline{\operatorname{Hom}}_A(M_{n,n},M_{i,l})=0$
  if and only if
  $1\leq l\leq n-i$,  or \ \ $t-i+1\leq l\leq t+n-i.$

  \vskip5pt

  \textup{Case 1.} \ Suppose that $1\leq l\leq n-i$. Since  $S_n$ is not a composition factor of $M_{i,l}$,  $\operatorname{Hom}_A(M_{n,n},M_{i,l})=0$.

  \vskip5pt

  \textup{Case 2.} \ Suppose that $t-i+1\leq l\leq t+n-i$. By the Auslander--Reiten formula
  \[
    D\underline{\operatorname{Hom}}_A(M_{n,n},M_{i,l})
    \cong \operatorname{Ext}^1_A(M_{i,l},M_{1,n})
    \cong D\overline{\operatorname{Hom}}_A(M_{1,n},M_{i+1,l})
  \]
  it suffices to show that every morphism $0\ne f: M_{1,n}\longrightarrow M_{i+1,l}$ factors through an injective module.
  Since $1 \leq (l+i-t) \leq n$, it follows that $f = e\circ p$,  where
  \[M_{1,n}\stackrel p \twoheadrightarrow M_{1,l+i-t} \stackrel e
    \hookrightarrow M_{i+1,l}.
  \]
  Let $\sigma:M_{1,n}\hookrightarrow I_n=M_{1,t}$ be the injective envelope. Consider the following canonical commutative diagram (cf. Fact \ref{pp}):
  \[
    \xymatrix@R=.7cm{
    & M_{1,n} \ar@{->>}[r]^-{p} \ar@{^{(}->}[d]_{\sigma}
    & M_{1,l+i-t} \ar@{^{(}.>}[d]^{\sigma'}
    \ar@{^{(}->}[r]^-{e}
    & M_{i+1,l} \\
    I_n \ar@{=}[r]
    & M_{1,t} \ar@{.>>}[r]^-{p'}
    & M_{1, l+i-n}
    &
    }
  \]
  Since $\mathrm{soc}(M_{1,l+i-n})=S_{i+l}=\mathrm{soc}(M_{i+1,l})$ and $l+i-n\leq l$, there exists $e':M_{1,l+i-n}\longrightarrow M_{i+1,l}$ such that $e'\circ\sigma'=e$.
  Hence $f = e \circ p = e' \circ \sigma' \circ p$ factors through the injective module $I_n$.

  \vskip5pt

  \textup{Case 3.} \ Suppose that $n-i+1\leq l\leq t-i$, i.e., $1 \leq (l+i-n) \leq (d-1)n$.  Write
  $l+i-n=qn+s$, where $0\le q \le d-2$ and $1\leq s\leq n$.
  Consider the canonical homomorphisms
  \[M_{n,n}\stackrel p \twoheadrightarrow M_{n,s}
    =\operatorname{rad}^{n-i+qn}(M_{i,l})
    \stackrel e \hookrightarrow M_{i,l}
  \]
  We claim that $f = e\circ p$ does not factor through a projective module.
  Otherwise, $f$ factors through the projective cover $\pi:P_i=M_{i,t+n-i}\twoheadrightarrow M_{i,l}$,
  say,  $f=\pi \circ h$ for $h:M_{n,n}\longrightarrow P_i$.
  Consider the following pullback square (cf. Fact \ref{pp}):
  \[
    \xymatrix@R=.7cm{
    M_{n,n} \ar@/_/@{->>}[rdd]_{p} \ar@/^/@{.>}[rrd]^{h}
    \ar@{.>}[rd]^{\varphi} & & & \\
    & M_{n,t-qn} \ar@{.>>}[d]^{\pi'}
    \ar@{^{(}.>}[r]^{e'}
    & M_{i,t+n-i} \ar@{=}[r] \ar@{->>}[d]^{\pi}
    & P_i \\
    & M_{n,s} \ar@{^{(}->}[r]^{e}
    & M_{i,l} &
    }
  \]
  By the universal property of the pullback, there is a morphism $\varphi$ such that $e'\circ\varphi=h$ and $\pi'\circ\varphi=p$.
  By Nakayama's lemma, $\varphi$ is an epimorphism, and then  $n \geq t-qn$. However, by construction $0 \leq q \leq d-2$ and $t = dn$,
  one has  $t-qn \geq 2n > n$. This contradiction implies that
  $\underline{\operatorname{Hom}}_A(M_{n,n},M_{i,l})\neq0$.

  \vskip5pt
  This proves  ${}^\perp\mathcal B=\mathcal W$. To show that $(\mathcal W, \mathcal B)$ is a cotorsion pair, it remains to show  $\mathcal W^\perp \subseteq \mathcal B$, or equivalently, the following claim:

  \vskip5pt

  \textbf{Claim 2} \
  For any indecomposable non-injective module $M_{i,l} \notin \{M_{1,rn}\}_{1 \leq r \leq d-1}$, there exists some $W\in\mathcal W$ such that $\operatorname{Ext}^1_A(W,M_{i,l})\neq0$.

  \vskip5pt

  \textup{Case 1.} \ Suppose that $i\neq1$.
  By construction  $W = S_{i-1} = M_{i-1,1}\in \mathcal{W}$.
  Then the exact sequence
  \[
    0 \longrightarrow M_{i,l}\longrightarrow M_{i-1,l+1}
    \longrightarrow M_{i-1,1}\longrightarrow0.
  \]
  shows that $\operatorname{Ext}^1_A(M_{i-1,1},M_{i,l})\neq0$.

  \vskip5pt

  \textup{Case 2.} \ Suppose that $i=1$. By the assumption $M_{i,l} \notin \{M_{1,rn}\}_{1 \leq r \leq d-1}$ one knows that $l$ is not a multiple of $n$; and since $M_{1,l}$ is not injective,
  $l < t$. Thus one can write $l = qn + s$ with $0 \leq q \leq d-1$ and $1 \leq s \leq n-1$.
  Then  $W = M_{n,t-n+1}\in \mathcal W$.
  Since $t-n+s+1 \leq t \le c_n = t$,  $M_{n, t-n+s+1}$ is well-defined.
  By the canonical non-splitting short exact sequence (cf. Fact \ref{pp})
  \[
    0 \longrightarrow M_{1,l}\longrightarrow  M_{1,qn} \oplus M_{n, t-n+s+1}\longrightarrow M_{n, t-n+1}\longrightarrow0
  \]
  one sees that $\operatorname{Ext}^1_A(M_{n,t-n+1}, M_{1,l})\neq0$.

  \vskip5pt

  Up to now we have proved that  $(\mathcal W,\mathcal B)$ is a cotorsion pair. It is complete by \Cref{prop_ctp_1}.
  For each $1 \leq r \leq d-1$,  the exact sequence
  \[
    0 \longrightarrow M_{1,rn} \longrightarrow I_n = M_{1,t} \longrightarrow M_{1,(d-r)n} \longrightarrow 0.
  \]
  implies $\mathcal{B}$ is closed under cosyzygies. Thus $(\mathcal W,\mathcal B)$ is hereditary by \Cref{cor_ctp_2}.

  \vskip5pt

  To see that $(\mathcal W,\mathcal B)$ is an injective cotorsion pair,  by \Cref{def_inj_ctp} it remains to show $\mathcal W\cap\mathcal B=\mathcal I(A)$. In fact,
  \begin{align*}\mathcal I(A) = \operatorname{add}(I) & = \operatorname{add}\left(\bigoplus_{t\le l\le t+n-1}M_{1,l}\right) \\ &\subseteq
              \mathcal W
              = \operatorname{add}\left(
              \bigoplus_{1\le i\le n}
              \left(
                \bigoplus_{1\le l\le n-i}M_{i,l}
                \oplus
                \bigoplus_{t-i+1\le l\le t+n-i}M_{i,l}
                \right)
              \right).\end{align*}
  \vskip10pt
  \noindent Thus $\mathcal I(A)\subseteq\mathcal W\cap\mathcal B$. For any non-injective indecomposable object  $M_{1,rn} \in \mathcal B$, where $1\leq r\leq d-1$,
  since $rn\notin [dn, t+n-1]$,  $M_{1,rn}\notin\mathcal W$. Thus $\mathcal W\cap\mathcal B\subseteq \mathcal I(A)$.
  This completes the proof.
\end{proof}

\subsection{Exceptional \textup{Hovey} triples in $A\mbox{-}{\rm mod}$}

$\,$

\vskip5pt

Recall that $I = \bigoplus\limits_{1\le i\le n} I_i = M_{1, t+1}\oplus \cdots \oplus M_{1, t+n-1}\oplus M_{1, t}$.

\vskip10pt

\begin{lemma}\label{lem_m_nak_injective_hovey} \ Suppose $t = d n$ for some $d \geq 2$.
  Put
  \begin{align*}
    \mathcal C
    = & \operatorname{add}\left(\bigoplus_{1 \leq l \leq t}M_{n,l}\oplus\bigoplus_{\substack {1 \leq i \leq n-1 \\ 1\leq l \leq n-i-1}} M_{i,l}
    \oplus
    \bigoplus_{\substack {1 \leq i \leq n-1                                                                     \\ t+1-i\leq l\leq t+n-i}} M_{i,l}
    \right)                                                                                                     \\
    \mathcal F
    = & \operatorname{add}\left(
    I
    \oplus
    \bigoplus_{1\leq r\leq d-1}M_{1,rn}
    \oplus
    \bigoplus_{0\leq r\leq d-1}M_{n,rn+1}
    \right)                                                                                                     \\
    \mathcal W
    = & \operatorname{add}\left(
    \bigoplus_{1\le i\le n}
    \left(
      \bigoplus_{1\le l\le n-i}M_{i,l}
      \oplus
      \bigoplus_{t-i+1\le l\le t+n-i}M_{i,l}
      \right)
    \right).
  \end{align*}
  Then $(\mathcal C,\mathcal F,\mathcal W)$ is an  exceptional \textup{Hovey} triple in $A\mbox{-}{\rm mod}$.
\end{lemma}

\begin{proof} \ Put
  $$\mathcal{X} = \operatorname{add}(
    M_{n,t-n+1}\oplus I), \quad \ \ \ \mathcal B = \operatorname{add}\left(
    \bigoplus_{1 \leq r \leq d-1}M_{1,rn}\oplus I \right).$$
  By \Cref{lem_m_nak_injective_cx}, $(\mathcal{C}, \mathcal{X})$ is a non-hereditary complete cotorsion pair; and by \Cref{lem_m_nak_injective_wb}, $(\mathcal{W}, \mathcal{B})$ is an injective cotorsion pair.
  Moreover $\mathcal{X} \subseteq \mathcal{W}$.
  Hence, by \Cref{theorem_merge}, $(\mathcal{C}, \ (\mathcal{C} \cap \mathcal{W})^\perp, \ \mathcal{W})$ is a Hovey triple.  We will show  $(\mathcal{C} \cap \mathcal{W})^\perp = \mathcal{F}$.
  By construction
  \[
    \mathcal C\cap\mathcal W
    = \operatorname{add}\left(
    \bigoplus_{1\le i\le n}
    \left(
      \bigoplus_{1\le l\le n-i-1}M_{i,l}
      \oplus
      \bigoplus_{t-i+1\le l\le t+n-i}M_{i,l}
      \right)
    \right)
  \]

  \vskip10pt
  \noindent where $\bigoplus\limits_{1\le l\le n-i-1}M_{i,l}$ is understood to be $0$ when $(n-1) \leq i \leq n$.
  We claim that $(\mathcal C\cap\mathcal W)^\perp=\mathcal F$. Since every injective module belongs to $(\mathcal{C} \cap \mathcal{W})^\perp \cap \mathcal F$,
  we will only consider non-injective indecomposable modules.

  \vskip5pt

  \textbf{Claim 1} \ If $i\notin \{1,n\}$, then each non-injective object $M_{i,l} \notin (\mathcal{C} \cap \mathcal{W})^\perp$.
  \vskip5pt

  In fact, consider $S_{i-1} = M_{i-1,1}$. Since $1 \leq 1 \leq n-(i-1)-1$, one sees that $M_{i-1,1} \in \mathcal{C} \cap \mathcal{W}$. Then the exact sequence
  \[
    0\longrightarrow M_{i,l}\longrightarrow M_{i-1,l+1}\longrightarrow S_{i-1}\longrightarrow 0
  \]
  implies that  $\operatorname{Ext}^1_A(S_{i-1},M_{i,l})\neq0$, and thus $M_{i,l}\notin(\mathcal C\cap\mathcal W)^\perp$.

  \vskip5pt

  \textbf{Claim 2} \ If $i = 1$ and $l$ is a multiple of $n$, then $M_{1,l} \in (\mathcal{C} \cap \mathcal{W})^\perp$.

  \vskip5pt

  In fact, $M_{1,l} \in \mathcal{B} = \mathcal{W}^\perp \subseteq (\mathcal{C} \cap \mathcal{W})^\perp$.

  \vskip5pt

  \textbf{Claim 3} \ If $i = 1$, $l$ is not a multiple of $n$,  and $M_{i,l}$ is not injective, then $M_{1,l} \notin (\mathcal{C} \cap \mathcal{W})^\perp$.

  \vskip5pt

  In fact, in this case one can write $l=qn+r$, where $0\leq q\leq d-1$ and $1\leq r\leq n-1$.
  Note that $M_{1,t}=I_n\in\mathcal C\cap\mathcal W$.  The exact sequence (cf. Fact \ref{pp})
  \[
    0\longrightarrow M_{1,qn+r}
    \longrightarrow M_{1,qn}\oplus  M_{1,t+r}
    \longrightarrow M_{1,t}
    \longrightarrow 0
  \]
  (where $M_{1,0}=0$ if $q=0$) implies that $\operatorname{Ext}^1_A(M_{1,t},M_{1,l})\neq0$, and thus $M_{1,l}\notin(\mathcal C\cap\mathcal W)^\perp$.

  \vskip5pt

  \textbf{Claim 4} \ If $i = n$ and $l = 1 \pmod n$, then $M_{n,l} \in (\mathcal{C} \cap \mathcal{W})^\perp$.

  \vskip5pt

  In fact, it is clear $\mathcal{B} \cup \mathcal{X} \subseteq (\mathcal{C} \cap \mathcal{W})^\perp$.
  Write $l=qn+1$ with $0\leq q\leq d-1$.

  \vskip5pt

  If $1\leq q\leq d-1$, then there is a short exact sequence (cf. Fact \ref{pp})
  \[
    0\longrightarrow M_{n,t-n+1}
    \longrightarrow M_{n,qn+1}\oplus M_{1,t-n}
    \longrightarrow  M_{1,qn}
    \longrightarrow0
  \]
  where $M_{1,qn}\in\mathcal B$ and $M_{n,t-n+1}\in\mathcal X$ by construction.
  Thus $M_{n,qn+1}\in(\mathcal C\cap\mathcal W)^\perp$.

  \vskip5pt

  If $q=0$, then one has a short exact sequence (cf. Fact \ref{pp})
  \[
    0\longrightarrow M_{n,t-n+1}
    \longrightarrow M_{n,1}\oplus M_{1,t}
    \longrightarrow M_{1,n}
    \longrightarrow0
  \]
  where $M_{1,n}\in\mathcal B$ and $M_{n,t-n+1}\in\mathcal X$. The same argument shows $M_{n,1}\in(\mathcal C\cap\mathcal W)^\perp$.

  \vskip5pt

  \textbf{Claim 5} \ If $i = n$ and $l \neq 1 \pmod n$, then $M_{n,l} \notin (\mathcal{C} \cap \mathcal{W})^\perp$.

  \vskip5pt

  In fact, write $l = qn + r$ for $0 \leq q \leq d-1$ and $2 \leq r \leq n$. Consider $M_{r-1,t-r+2}$.
  Since $t-(r-1)+1 \leq t-r+2\leq t+n-(r-1)$,   $M_{r-1,t-r+2}\in\mathcal C\cap\mathcal W$.
  The exact sequence (cf. Fact \ref{pp})
  \[
    0\longrightarrow M_{n,qn+r}
    \longrightarrow  M_{n,qn+1}\oplus M_{r-1,t+1}
    \longrightarrow M_{r-1,t-r+2}
    \longrightarrow0
  \]
  implies that $\operatorname{Ext}^1_A(M_{r-1,t-r+2},M_{n,l}) \neq 0$, and thus $M_{n,l}\notin(\mathcal C\cap\mathcal W)^\perp$.

  \vskip5pt

  By \textbf{Claims 1--5} one sees that $(\mathcal C\cap\mathcal W)^\perp=\mathcal F$, and hence  $(\mathcal C,\mathcal F,\mathcal W)$ is a \textup{Hovey} triple in $A\mbox{-}{\rm mod}$.
  It remains to show that  $(\mathcal{C} \cap \mathcal{F}, \mathcal{C} \cap \mathcal{F} \cap \mathcal{W})$ is not a Frobenius pair.
  By construction
  $$\mathcal{C} \cap \mathcal{F} =\operatorname{add}\left(I\oplus \bigoplus_{0\leq r\leq d-1}M_{n,rn+1}\right),
    \ \ \ \ \mathcal{C} \cap \mathcal{F} \cap \mathcal{W} =\operatorname{add}\left(I\oplus M_{n,t-n+1}\right). $$
  We claim that $M_{n,1}$ is a projective object in $\mathcal{C} \cap \mathcal{F}$. For any non-injective indecomposable object $M_{n,rn+1}$ with $0\leq r\leq d-1$ in $\mathcal{C} \cap \mathcal{F}$,
  by the Auslander--Reiten formula, one has
  $$\mathrm{Ext}_A^1(M_{n,1}, M_{n,rn+1}) \cong D\underline{\mathrm{Hom}}_A(\tau^{-1}M_{n,rn+1}, M_{n,1}) = D\underline{\mathrm{Hom}}_A(M_{n-1,rn+1}, M_{n,1}).$$
  Since $M_{n,1} = S_n$ and  $\operatorname{top}(M_{n-1,rn+1})=S_{n-1}$, $\mathrm{Hom}_A(M_{n-1,rn+1}, M_{n,1}) = 0$, and hence

  \noindent $\operatorname{Ext}^1_A(M_{n,1},M_{n,rn+1})=0$. This proves the claim.

  \vskip5pt

  Since $M_{n,1} \notin \mathcal{C} \cap \mathcal{F} \cap \mathcal{W}$,  $\mathcal{P}(\mathcal{C} \cap \mathcal{F}) \neq \mathcal{C} \cap \mathcal{F} \cap \mathcal{W}$.
  By definition $(\mathcal{C} \cap \mathcal{F}, \mathcal{C} \cap \mathcal{F} \cap \mathcal{W})$ is not a Frobenius pair. Thus $(\mathcal C,\mathcal F,\mathcal W)$ is an exceptional Hovey triple.
\end{proof}

\subsection{Proof of \Cref{thm_m_nak_classification}}

If $t = dn$ with $d \geq 2$, then by \Cref{lem_m_nak_injective_hovey} one sees that $A\mbox{-}{\rm mod}$ admits an exceptional Hovey triple.

\vskip5pt

Conversely, suppose that $A\mbox{-}{\rm mod}$ admits an exceptional Hovey triple $(\mathcal C,\mathcal F,\mathcal W)$. Then
$(\mathcal C\cap\mathcal F,\mathcal C\cap\mathcal F\cap\mathcal W)$ is not a Frobenius pair. We claim that $t=dn$ for some $d\ge2$.
Otherwise, $t=n$ or $n\nmid t$. By \Cref{cor_m_nak_hovey_triple},
$(\mathcal C,\mathcal F,\mathcal W)=(\mathcal C,\mathcal F,A\mbox{-}{\rm mod})$
for some complete cotorsion pair $(\mathcal C,\mathcal F)$.
In this case, the pair
\[
(\mathcal C\cap\mathcal F,\mathcal C\cap\mathcal F\cap\mathcal W)
=(\mathcal C\cap\mathcal F,\mathcal C\cap\mathcal F)
\]
is a Frobenius pair, contradicting the assumption. \hfill$\square$

\subsection{The Type of the exceptional \textup{Hovey} triple in \Cref{lem_m_nak_injective_hovey}}

$\,$

\vskip5pt

\begin{theorem} \label{prop_m_nak_injective_type} \
  The exceptional \textup{Hovey} triple $(\mathcal C, \mathcal F, \mathcal W)$ in \Cref{lem_m_nak_injective_hovey} is of \textup{Type I}, i.e., $\mathcal{C} \cap \mathcal{F}$ is a Frobenius category, but $\mathcal{C} \cap \mathcal{F} \cap \mathcal{W} \ne \mathcal{P}(\mathcal{C} \cap \mathcal{F})$.
\end{theorem}

\begin{proof} \  By Corollary \ref{corollary_D_enough_PI},  $\mathcal C\cap\mathcal F$ has enough projective and enough injective objects.
  As seen in the proof of \Cref{lem_m_nak_injective_hovey},

  $$\mathcal{C} \cap \mathcal{F} =\operatorname{add}\left(I\oplus \bigoplus_{0\leq r\leq d-1}M_{n,rn+1}\right),
    \ \ \ \ \mathcal{C} \cap \mathcal{F} \cap \mathcal{W} =\operatorname{add}\left(I\oplus M_{n,t-n+1}\right). $$

  \vskip5pt

  {\bf Claim 1.}  \  $\mathcal{P}(\mathcal{C} \cap \mathcal{F})= \operatorname{add}\left(I \oplus M_{n,1} \oplus M_{n,t-n+1}\right)$.

  \vskip5pt

  Since $\mathcal{C} \cap \mathcal{F} \cap \mathcal{W} \subseteq \mathcal{P}(\mathcal{C} \cap \mathcal{F})$, it follows that
  $\operatorname{add}\left(I\oplus M_{n,t-n+1}\right)\subseteq \mathcal{P}(\mathcal{C} \cap \mathcal{F})$.
  In the proof of \Cref{lem_m_nak_injective_hovey} we have seen that $M_{n, 1}\in \mathcal{P}(\mathcal{C} \cap \mathcal{F})$.
  Thus $\operatorname{add}\left(I \oplus M_{n,1} \oplus M_{n,t-n+1}\right) \subseteq \mathcal{P}(\mathcal{C} \cap \mathcal{F})$.
  It remains to prove that if $1 \leq r \leq d-2$ then $M_{n, rn + 1} \notin \mathcal{P}(\mathcal{C} \cap \mathcal{F})$.
  Indeed, the exact sequence (cf. Fact \ref{pp})
  $$0 \longrightarrow M_{n, rn + 1} \longrightarrow M_{n, (r-1)n + 1} \oplus M_{n, (r+1)n + 1} \longrightarrow M_{n, rn + 1} \longrightarrow 0$$
  shows that $\mathrm{Ext}^1_A(M_{n,rn+1}, M_{n,rn+1}) \neq 0$, and hence $M_{n,rn+1} \notin \mathcal{P}(\mathcal{C} \cap \mathcal{F})$.

  \vskip5pt

  {\bf Claim 2.}  \  $M_{n,1} \in \mathcal{I}(\mathcal{C} \cap \mathcal{F})$.

  \vskip5pt

  For any non-injective indecomposable object $M_{n,rn+1}$ with $0\leq r\leq d-1$ in $\mathcal{C} \cap \mathcal{F}$,
  by the Auslander--Reiten formula, one has
  $$\mathrm{Ext}_A^1(M_{n,rn+1}, M_{n,1}) \cong D\overline{\mathrm{Hom}}_A(M_{n,1}, \tau M_{n,rn+1}) = D\overline{\mathrm{Hom}}_A(M_{n,1}, M_{1,rn+1}).$$
  Since $M_{n,1} = S_n$ and  $\operatorname{soc}(M_{1,rn+1})=S_1$, it follows that $\mathrm{Hom}_A(M_{n,1}, M_{1,rn+1})=0$, and hence
  $\mathrm{Ext}_A^1(M_{n,rn+1}, M_{n,1})=0$. Thus $M_{n,1} \in \mathcal{I}(\mathcal{C} \cap \mathcal{F})$.

  \vskip5pt

  {\bf Claim 3.}  \  $\mathcal{I}(\mathcal{C} \cap \mathcal{F}) = \operatorname{add}\left(\mathcal{I}(A) \oplus M_{n,1} \oplus M_{n,t-n+1}\right)$.

  \vskip5pt
  Since $\mathcal{C} \cap \mathcal{F} \cap \mathcal{W} \subseteq \mathcal{I}(\mathcal{C} \cap \mathcal{F})$, it follows that
  $\operatorname{add}\left(I\oplus M_{n,t-n+1}\right)\subseteq \mathcal{I}(\mathcal{C} \cap \mathcal{F})$.
  By {\bf Claim 2}, $M_{n, 1}\in \mathcal{I}(\mathcal{C} \cap \mathcal{F})$.
  Thus $\operatorname{add}\left(I \oplus M_{n,1} \oplus M_{n,t-n+1}\right) \subseteq \mathcal{I}(\mathcal{C} \cap \mathcal{F})$.
  It remains to prove that if $1\leq r\leq d-2$, then
  $M_{n,rn+1}\notin\mathcal I(\mathcal C\cap\mathcal F)$.
  This follows from the proof of {\bf Claim 1}, where we showed that
  $\mathrm{Ext}^1_A(M_{n,rn+1},M_{n,rn+1})\neq0$.

  \vskip5pt

  Thus,  $\mathcal{C} \cap \mathcal{F}$ is a Frobenius category, and $\mathcal{C} \cap \mathcal{F} \cap \mathcal{W} \ne \mathcal{P}(\mathcal{C} \cap \mathcal{F})$.
\end{proof}

\subsection{An example} \  Take $n=2 =d$ and $t=4$  in \Cref{lem_m_nak_injective_hovey}.  Then $A=kC_2/\langle a_2 a_1a_2 a_1\rangle$ is a Nakayama algebra with Kupisch series $(5, 4)$,
and $A$ is not selfinjective, with a unique maximal module $M_{1, 5}$. The Auslander-Reiten quiver of $A$ is
\begin{equation*}
  \xymatrix@R= 0.3cm @C=0.4cm{
  &  &  &  & M_{1,5} \ar@{->}[rd] &  &  &  &  \\
  &  &  & M_{2,4} \ar@{->}[ru] \ar@{->}[rd] &  & M_{1,4} \ar@{->}[rd] \ar@{.}[ll] &  &  &  \\
  &  & M_{1,3} \ar@{->}[ru] \ar@{->}[rd] &  & M_{2,3} \ar@{->}[ru] \ar@{->}[rd] \ar@{.}[ll] &  & M_{1,3} \ar@{->}[rd] \ar@{.}[ll] &  &  \\
  & M_{2,2} \ar@{->}[ru] \ar@{->}[rd] &  & M_{1,2} \ar@{->}[ru] \ar@{->}[rd] \ar@{.}[ll] &  & M_{2,2} \ar@{->}[rd] \ar@{->}[ru] \ar@{.}[ll] &  & M_{1,2} \ar@{->}[rd] \ar@{.}[ll] &  \\
  M_{1,1} \ar@{->}[ru] &  & M_{2,1} \ar@{->}[ru] \ar@{.}[ll] &  & M_{1,1} \ar@{->}[ru] \ar@{.}[ll] &  & M_{2,1} \ar@{->}[ru] \ar@{.}[ll] &  & M_{1,1} \ar@{.}[ll]
  }
\end{equation*}

\vskip10pt

\noindent
By \Cref{lem_m_nak_injective_hovey}, $(\mathcal{C}, \mathcal{F}, \mathcal{W})$ is an exceptional Hovey triple in $A\mbox{-}{\rm mod}$ of \textup{Type I},
where
\begin{align*}
  \mathcal{C} & =\operatorname{add}(M_{2,1}\oplus M_{2,2}\oplus M_{2,3}\oplus M_{2,4}\oplus M_{1,4}\oplus M_{1,5}) \\
  \mathcal{F} & =\operatorname{add}(M_{1,5}\oplus M_{1,4}\oplus M_{1,2} \oplus M_{2,1}\oplus M_{2,3})              \\
  \mathcal{W} & =\operatorname{add}(M_{1,1}\oplus M_{1,4}\oplus M_{1,5}\oplus M_{2,3}\oplus M_{2,4}).
\end{align*}

\vskip5pt

\section{\bf Homotopies in additive categories}

For a category $\mathcal A$ equipped with a model structure, let $\mathcal C$,  $\mathcal F$ and $\mathcal W$ denote the class of cofibrant objects, the class of fibrant objects, and the class of trivial objects, respectively.

\begin{defnlem} \label{lhtp} \ {\rm(\cite[Definition 3, page 1.4; Lemma 2, page 1.6]{Q1})} \ Let $(\mathsf{Cofib}, \mathsf{Fib}, \mathsf{Weq})$ be a model structure on a category $\mathcal A$, $f$ and $g$  in $\Hom_\mathcal A(A, B)$. Then the following are equivalent.

  \vskip5pt

  (1) \ \ There is a commutative square with $\sigma$  a weak equivalence:
  $$\xymatrix@R=0.5cm @C=1.5cm{A\oplus A\ar[d]_-{(1,1)}\ar[rd]^-{(\partial_0,\partial_1)} \ar[r]^-{(f,g)} & B \\
    A & \widetilde{A}\ar[l]_-\sigma\ar[u]_-h }$$

  \vskip5pt

  (2) \ There is a commutative square with $\sigma$  a weak equivalence and $(\partial_0, \partial_1)$ a cofibration:
  $$\xymatrix@R=0.5cm @C=1.5cm{A\oplus A\ar[d]_-{(1,1)}\ar[rd]^-{(\partial_0,\partial_1)} \ar[r]^-{(f,g)} & B \\
    A & \widetilde{A}\ar[l]_-\sigma\ar[u]_-h }$$

  \vskip10pt

  If the equivalent conditions above are satisfied, then $f$ is said to be {\it left-homotopic to} $g$, denoted by $f\stackrel{l}\sim g$.
  Moreover, if $A$ is a cofibrant object, then  $\partial_0$ and  $\partial_1$ in {\rm (2)} are trivial cofibrations.
\end{defnlem}

\vskip5pt

Dually, one has the right homotopy relation $\overset{r}{\sim}$.
It is known that $\overset{l}{\sim}$ and $\overset{r}{\sim}$ coincide on $\mathcal C\cap\mathcal F$, the full subcategory of cofibrant-fibrant objects.
Denote this equivalence relation on $\mathcal C\cap\mathcal F$ by $\sim$, and the corresponding quotient category by $(\mathcal C\cap\mathcal F)/_\sim$ (see \cite[Lemmas 4, 5, and their duals]{Q1}).

\vskip5pt

The following important theorem is known as the {\it Fundamental Theorem of Model Structures}.
See Hovey \cite[Theorem 1.2.10]{H1}, Gillespie \cite{Gil25}, and W. G. Dwyer and J. Spaliński \cite{DS}.
It was first proved by Quillen \cite[Theorem~1', p.~1.13]{Q1} for a model category. The version below is due to J. M. Egger \cite[Theorem 3.2]{Egg06}, who assumes only that
the category has a zero object, finite coproducts, finite products, and a model structure.

\vskip5pt

\begin{theorem}{\rm(Fundamental Theorem of Model Structures)}\label{thm} \ Let $(\mathsf{Cofib}, \mathsf{Fib}, \mathsf{Weq})$ be a model structure on a category $\mathcal A$ with a zero object, finite coproducts, and finite products. Then
  the composition of
  the embedding $(\mathcal{C}\cap \mathcal F)\hookrightarrow \mathcal{A}$  and the localization functor $\mathcal{A}\rightarrow \mathsf{Ho}(\mathcal{A})$ induces an equivalence
  $(\mathcal{C}\cap \mathcal F)/_\sim\cong\mathsf{Ho}(\mathcal{A})$ as categories.
\end{theorem}

\vskip5pt

\begin{lemma}\label{factor through CFW} \ Let $(\mathsf{Cofib}, \mathsf{Fib}, \mathsf{Weq})$ be a model structure on a category $\mathcal A$ with a zero object, and let $A$ and $B$ be objects in $\mathcal C\cap\mathcal F$.
  If $f:A\longrightarrow B$ factors through an object in $\mathcal W$, then $f$ factors through an object in $\mathcal C\cap\mathcal F\cap\mathcal W$.
\end{lemma}

\begin{proof} \ By assumption, there is a commutative diagram of morphisms
  \[\xymatrix@R=0.5cm{
    A\ar[rr]^-{f}\ar[dr]_-{u} & {} & B \\
    {} & W\ar[ur]_-{v} & {}
    }\]
  with $W\in \mathcal W$. Factorize $v$ as $v=p\circ i$ with $i\in \mathsf{TCofib}$ and $p\in \mathsf{Fib}$:
  \[\xymatrix@R=0.5cm{
    W\ar[rr]^-{v}\ar[dr]_-{i} & {} & B \\
    {} & W'\ar[ur]_-{p} & {}
    }\]
  Since $W\in \mathcal{W}$, $0: 0\longrightarrow W$ is in $\mathsf{Weq}$. Since  $i\in \mathsf{TCofib}$, it follows that
  $(0\longrightarrow W') = i\circ (0\longrightarrow W)$ is in $\mathsf{Weq}$, and hence $W'\in \mathcal{W}$. Similarly,
  since $p\in \mathsf{Fib}$ and $B\in \mathcal{F}$, $W'\in \mathcal{F}$. It follows that $W'\in \mathcal{F}\cap \mathcal{W}$.
  Factorize $i\circ u$ as $i\circ u=n\circ m$ with  $m\in \mathsf{Cofib}$ and $n\in \mathsf{TFib}$:
  \[\xymatrix@R=0.5cm{
    A\ar[rr]^-{iu}\ar[dr]_-{m} & {} & W' \\
    {} & W''\ar[ur]_-{n} & {}
    }\]
  Since $m\in \mathsf{Cofib}$ and $A\in \mathcal{C}$,  $W''\in \mathcal{C}$. Since $W'\in \mathcal{F}\cap \mathcal{W}$, $0: W'\longrightarrow 0$ is in $\mathsf{TFib}$. Since $n\in \mathsf{TFib}$,
  $(W''\longrightarrow 0) = (W'\longrightarrow 0)\circ n$ is in $\mathsf{TFib}.$
  It follows that  $W''\in \mathcal{F}\cap \mathcal{W}$, and hence $W''\in \mathcal{C}\cap \mathcal{F}\cap \mathcal{W}$. Thus $f$ factors through $W''\in \mathcal{C}\cap \mathcal{F}\cap \mathcal{W}$.
\end{proof}

\vskip5pt

The following result has been obtained for exact model structures on weakly idempotent complete exact categories in \cite[Proposition 4.4]{Gil11}.

\vskip5pt

\begin{theorem} \label{htcatisaddq}  \ Let $(\mathsf{Cofib}, \mathsf{Fib}, \mathsf{Weq})$ be an arbitrary model structure on a weakly idempotent complete additive category $\mathcal{A}$,
  and let $f,g:A\longrightarrow B$ be morphisms with $A,B\in\mathcal{C}\cap\mathcal{F}$. Then $f\sim g$ if and only if $f-g$ factors through an object in $\mathcal{C}\cap\mathcal{F}\cap\mathcal{W}$.
\end{theorem}
\begin{proof} \ Let $f, \ g: A \longrightarrow B$ with $A, B\in \mathcal{C}\cap \mathcal{F}$. Assume that $f-g$ factors through an object $W\in \mathcal{C}\cap \mathcal{F}\cap \mathcal{W}$, say with commutative diagram
  \[\xymatrix@R=0.5cm{
    A\ar[rr]^-{f-g}\ar[dr]_-{u} & {} & B \\
    {} & W\ar[ur]_-{v} & {}
    }\]
  Then one has a commutative diagram
  \[\xymatrix@R=0.7cm{
    A\oplus A\ar[rr]^-{(f,g)}\ar[d]_-{(1,1)}\ar[drr]^-{\left(\begin{smallmatrix} 1 & 1\\ u &0 \end{smallmatrix}\right) } & {} & B \\
    A & {} & A\oplus W\ar[u]_-{(g,v)}\ar[ll]_-{(1,0)}
    }\]
  where $(1, 0): A\oplus W \longrightarrow A$ is a coproduct of weak equivalences $1 = \mathrm{Id}_A : A\longrightarrow A$ and $0: W\longrightarrow 0$. Thus $(1, 0): A\oplus W \longrightarrow A$ is a weak equivalence.
  By definition $f \sim g$.

  \vskip5pt

  Conversely, if $f \sim g$, then one has a commutative diagram
  $$\xymatrix@R=0.5cm{A\oplus A \ar[rr]^-{(f,g)} \ar[d]_{(1,1)} \ar[rrd]^-{(\partial_0,\partial_1)} & & B\\ A & & C\ar[u]_{h}\ar[ll]_{\sigma}}
  $$
  with $\sigma$ a weak equivalence. Since $\sigma$ is a splitting epimorphism and $\mathcal{A}$ is a weakly idempotent complete additive category, it follows that $\sigma$ has a kernel
  $W$ and that there are morphisms $\mu: A \longrightarrow C$, $t: W\longrightarrow C$, $s:C\longrightarrow W$, such that $$\sigma\circ t = 0, \quad s\circ t=\mathrm{Id}_W, \quad \sigma \circ \mu=\mathrm{Id}_A, \quad \mu \circ \sigma + t\circ s=\mathrm{Id}_C.$$
  Since $f-g=h\circ (\partial_0-\partial_1)=h\circ (\mu \circ \sigma+t\circ s)\circ (\partial_0-\partial_1)=h\circ t\circ s\circ (\partial_0-\partial_1)$, one gets a commutative diagram

  \[\xymatrix @R=0.5cm{
    A\ar[rr]^-{f-g}\ar[dr]_-{s\circ (\partial_0-\partial_1)} & {} & B \\
    {} & W\ar[ur]_-{ht} & {}
    }\]
  By the commutative diagram
  \[\xymatrix@R=0.5cm{
    W\ar[r]^-{t}\ar[d] & C\ar[r]^-{s}\ar[d]^-{\sigma} & W\ar[d] \\
    0\ar[r] & A\ar[r] & 0
    }\]
  one sees that $W\longrightarrow 0$ is a retraction of the weak equivalence $\sigma: C\longrightarrow A$. Thus $W\longrightarrow 0$ is a weak equivalence, and hence  $W\in \mathcal{W}$.
  By Lemma \ref{factor through CFW},  $f-g$ factors through an object in $\mathcal{C}\cap \mathcal{F}\cap \mathcal{W}$.
\end{proof}

\vskip5pt

The following result has been proved in \cite[Theorem 1.1]{LZ}. Here we provide a different and direct proof using Theorem \ref{htcatisaddq}.

\vskip5pt

\begin{corollary} \label {hcatonadditve} \ Let $(\mathsf{Cofib}, \mathsf{Fib}, \mathsf{Weq})$ be a model structure on a weakly idempotent complete additive category  $\mathcal{A}$.
  Then the homotopy category $\mathsf{Ho}(\mathcal{A})$ is an additive category, and $\mathsf{Ho}(\mathcal{A}) \cong \frac{\mathcal{C}\cap \mathcal{F}}{\mathcal{C}\cap \mathcal{F}\cap \mathcal{W}}$ as additive categories.
\end{corollary}
\begin{proof} \ By Theorem \ref{thm}, $\mathsf{Ho}(\mathcal{A}) \cong (\mathcal C\cap \mathcal F)/_\sim$.
  Since $\mathcal{C}$,  $\mathcal{F}$ and $\mathcal{W}$ are closed under coproducts, $\mathcal{C}\cap \mathcal{F}$ and $\mathcal{C}\cap \mathcal{F}\cap \mathcal{W}$ are full additive categories of $\mathcal A$.
  We claim that the equivalence relation ideal $\sim$ on $\mathcal{C}\cap \mathcal{F}$ is additive, i.e.,
  $\{f\in \Hom_{\mathcal C}(A, B) \ | \ f\sim 0\}$ is a subgroup of $\Hom_{\mathcal A}(A, B)$. In fact, assume that $f\sim 0$ and $g\sim 0$. Then by
  Theorem \ref{htcatisaddq}, $f$ factors through an object in $\mathcal{C}\cap \mathcal{F}\cap \mathcal{W}$, and $g$ factors through an object in $\mathcal{C}\cap \mathcal{F}\cap \mathcal{W}$. Thus one has
  commutative diagrams
  \[\xymatrix@R=0.5cm{
    A\ar[rr]^-{f}\ar[dr]_-{u_1} & {} & B &&  A\ar[rr]^-{g}\ar[dr]_-{u_2} & {} & B \\
    {} & W_1\ar[ur]_-{v_1} & {} &&  {} & W_2\ar[ur]_-{v_2} & {}
    }\]
  where $W_1, W_2 \in \mathcal{C}\cap \mathcal{F}\cap \mathcal{W}$. It follows that there is a commutative diagram
  \[\xymatrix@R=0.5cm{
    A\ar[rr]^-{f-g}\ar[dr]_-{\binom{u_1}{u_2}} & {} & B\\
    {} & W_1\oplus W_2\ar[ur]_-{(v_1, -v_2)} & {}}\]
  with $W_1\oplus  W_2 \in \mathcal{C}\cap \mathcal{F}\cap \mathcal{W}$. Thus $f-g\sim 0$, by Theorem \ref{htcatisaddq}. This justifies the claim.
  Therefore $(\mathcal C\cap \mathcal F)/_\sim$ is an additive category, since it is obtained by quotienting out the additive equivalence relation ideal $\sim$.

  \vskip5pt

  By Theorem \ref{htcatisaddq}, the functor $$F: (\mathcal C\cap \mathcal F)/_\sim\longrightarrow \frac{\mathcal{C}\cap \mathcal{F}}{\mathcal{C}\cap \mathcal{F}\cap \mathcal{W}}, \qquad A \longmapsto A, \quad \widetilde{f} \longmapsto\overline{f}$$ is an equivalence of additive categories.
\end{proof}

\vskip5pt

\begin{remark} \ Corollary \ref{hcatonadditve} has been known for exact model structures on weakly idempotent complete exact categories in Gillespie {\rm \cite[Proposition 4.4]{Gil11}};
  for the $\omega$-model structures on weakly idempotent complete exact categories in A. Beligiannis and I. Reiten {\rm \cite[VIII, Theorem 4.2]{BR}}  and {\rm \cite[Theorem 1.1]{CLZ}};
  and for exact model structures on triangulated categories in Nakaoka {\rm \cite[Proposition 6.10]{N}}.\end{remark}

\section{\bf Homotopy categories of exceptional model structures}

This section characterizes exceptional exact model structures via their homotopy categories.

\vskip5pt

Let $(\mathcal C,\mathcal F,\mathcal W)$ be a Hovey triple in a weakly idempotent complete extriangulated category $(\mathcal A, \mathbb E, \mathfrak s)$.
Then $\mathsf{Ho}(\mathcal A)\cong \frac{\mathcal C\cap\mathcal F}{\mathcal C\cap\mathcal F\cap\mathcal W}$ as additive categories.
By Nakaoka and Palu \cite[Theorem 6.20]{NP19},
$\mathsf{Ho}(\mathcal A)$ has a triangulated structure, and hence
$\frac{\mathcal C\cap\mathcal F}{\mathcal C\cap\mathcal F\cap\mathcal W}$ has the induced triangulated structure.
However, the induced extriangulated category $(\frac{\mathcal C\cap\mathcal F}{\mathcal C\cap\mathcal F\cap\mathcal W}, \ \overline{\mathbb E}, \ \overline{\mathfrak s})$ is not necessarily {\it canonically triangulated} (see Definition \ref{cantri}): indeed,
Theorem \ref{thm_Frobenius_pair} claims that the Hovey triple $(\mathcal C,\mathcal F,\mathcal W)$ is exceptional if and only if
$(\frac{\mathcal C\cap\mathcal F}{\mathcal C\cap\mathcal F\cap\mathcal W}, \ \overline{\mathbb E}, \ \overline{\mathfrak s})$ is not canonically triangulated.

\vskip5pt

In other words,  $\frac{\mathcal C\cap\mathcal F}{\mathcal C\cap\mathcal F\cap\mathcal W}$ has another extriangulated structure, arising from its triangulated structure.
Comparing the two extriangulated structures on $\frac{\mathcal C\cap\mathcal F}{\mathcal C\cap\mathcal F\cap\mathcal W}$,
one has {\it an extriangle functor} between  them. Then Theorem \ref{thm_Frobenius_pair} claims that $(\mathcal C,\mathcal F,\mathcal W)$ is exceptional if and only if this
extriangle functor is not {\it an extriangle-equivalence}.

\vskip5pt

\subsection{Two extriangulated structures on $\frac{\mathcal C\cap\mathcal F}{\mathcal C\cap\mathcal F\cap\mathcal W}$} $\,$

\vskip5pt

The first extriangulated structure on $\frac{\mathcal C\cap\mathcal F}{\mathcal C\cap\mathcal F\cap\mathcal W}$ comes from the following fact.

\vskip5pt

\begin{lemma} \label{extriangulatedquotint} {\rm (\cite[Proposition 3.30]{NP19})} \ Let $(\mathcal{A}, \mathbb E, \mathfrak s)$ be an extriangulated category, and $\mathcal{B}$ a full additive subcategory with
  $\mathcal{B} \subseteq \mathcal P(\mathcal A) \cap \mathcal I(\mathcal A)$. Then there is an induced extriangulated category $(\mathcal{A}/\mathcal{B}, \ \overline{\mathbb E}, \ \overline{\mathfrak s})$.
\end{lemma}

Since $\mathcal C\cap\mathcal F$ is extension closed,   $\mathcal C\cap\mathcal F$ has an extriangulated structure, inherited from $(\mathcal{A}, \mathbb E, \mathfrak s)$. Since $(\mathcal C\cap\mathcal F\cap\mathcal W)\subseteq \mathcal P(\mathcal C\cap\mathcal F) \cap \mathcal I(\mathcal C\cap\mathcal F)$, by Lemma \ref{extriangulatedquotint} there is an induced extriangulated structure
$(\frac{\mathcal C\cap\mathcal F}{\mathcal C\cap\mathcal F\cap\mathcal W}, \ \overline{\mathbb E}, \ \overline{\mathfrak s})$.

\vskip5pt

To see the second extriangulated structure on $\frac{\mathcal C\cap\mathcal F}{\mathcal C\cap\mathcal F\cap\mathcal W}$, one has to recall Nakaoka-Palu's triangulated structure on the homotopy category $\mathsf{Ho}(\mathcal A)$ of the corresponding exact model structure via the Hovey correspondence.
For each $X\in\mathcal A$, since $(\mathcal C,\mathcal F\cap\mathcal W)$ is a complete cotorsion pair,
one can fix an $\mathbb E$-triangle $X\xlongrightarrow{\iota_X}V_X\xlongrightarrow{\pi_X}C_X \xdashrightarrow{\delta_X}$ with $V_X\in\mathcal F\cap\mathcal W$ and $C_X\in\mathcal C$.
For each $f: X\longrightarrow Y$, since $\mathbb E(C_X,V_Y)=0$, there is $g: V_X\longrightarrow V_Y$ such that
$g\circ\iota_X=\iota_Y\circ f$.
By {\bf ET3} one gets a morphism $(f,g,h)$ of $\mathbb E$-triangles:
\[
  \xymatrix@R=15pt{
  X \ar@{->}[r]^{\iota_X} \ar@{->}[d]_{f} & V_X \ar@{->}[r]^{\pi_X} \ar@{.>}[d]^{g} & C_X \ar@{-->}[r]^{\delta_X} \ar@{.>}[d]^{h} & {} \\
  Y \ar@{->}[r]^{\iota_Y} & V_Y \ar@{->}[r]^{\pi_Y} & C_Y \ar@{-->}[r]^{\delta_Y} & {}}\]
The assignments $X\longmapsto C_X$ and $f\longmapsto h$ induce an equivalence
$[1]:\mathsf{Ho}(\mathcal A)\longrightarrow \mathsf{Ho}(\mathcal A)$ (cf. \cite[Proposition 6.14]{NP19}).
For each $\mathbb E$-triangle $X\xlongrightarrow{f}Y\xlongrightarrow{g}Z\xdashrightarrow{\delta}$,  by \cite[the dual of Proposition 3.15]{NP19} one has a commutative diagram of $\mathbb E$-triangles
\[
  \xymatrix@R=15pt{
  X \ar@{->}[r]^{\iota_X} \ar@{->}[d]_-{f} & V_X \ar@{->}[r]^{\pi_X} \ar@{.>}[d] & C_X \ar@{-->}[r]^{\delta_X} \ar@{=}[d] & {} \\
  Y \ar@{->}[d]_-{g} \ar@{.>}[r] & E \ar@{.>}[r]^-{s} \ar@{.>}[d]^-{h} & C_X \ar@{-->}[r]^{\mu} & {} \\
  Z \ar@{-->}[d]_{\delta} \ar@{=}[r] & Z \ar@{-->}[d]^{\xi} & & \\
  {} & {} & &
  }\]
with $h\in\mathsf{TFib}\subseteq\mathsf{Weq}$.
Let $\ell:\mathcal A\to\mathsf{Ho}(\mathcal A)$ be the localization functor and put
\[
  \widetilde{\ell}(\delta)
  =-\ell(s)\circ\ell(h)^{-1}: Z\longrightarrow C_X=X[1].
\]
Then  \[\widetilde{\ell}_{Z,X}:
  \mathbb E(Z,X)\longrightarrow
  \operatorname{Hom}_{\mathsf{Ho}(\mathcal A)}(Z,X[1]),
  \qquad
  \delta\longmapsto\widetilde{\ell}(\delta)
\]
is a binatural transformation (cf. \cite[6.9, 6.10, 6.11, 6.13]{NP19}), i.e., for any $\delta \in \mathbb E(Z, X)$, $f : K \longrightarrow Z$ and $g : X \longrightarrow W$, one has
\begin{equation}\label{binatl}
  \widetilde{\ell}(f^\ast g_\ast \delta) = g [1] \circ \widetilde{\ell}(\delta) \circ f =  \widetilde{\ell}( g_\ast f^\ast\delta).
\end{equation}

\vskip5pt \noindent Moreover, $\mathsf{Ho}(\mathcal A)$ is a triangulated category (cf. \cite[Theorem 6.20]{NP19}), where the standard triangles are
$X\xlongrightarrow{\ell(f)}Y\xlongrightarrow{\ell(g)}Z \xlongrightarrow{\widetilde{\ell}(\delta)}X[1]$,  induced by $\mathbb E$-triangles $X\xlongrightarrow{f}Y\xlongrightarrow{g}Z\xdashrightarrow{\delta}$.

\vskip5pt

The process above uses the equivalence
$\mathsf{Ho}(\mathcal A)\cong \frac{\mathcal C\cap\mathcal F}{\mathcal C\cap\mathcal F\cap\mathcal W}$
of additive categories; see Theorem \ref{hcatonadditve}.
In the special case of exact model structures, this equivalence can also be obtained without Theorem \ref{hcatonadditve}: by Theorem \ref{thm}, one has
$(\mathcal C\cap\mathcal F)/_\sim\cong\mathsf{Ho}(\mathcal A)$,
and
$(\mathcal C\cap\mathcal F)/_\sim\cong
\frac{\mathcal C\cap\mathcal F}{\mathcal C\cap\mathcal F\cap\mathcal W}$
by \cite[Proposition 4.4]{Gil11} or \cite[Proposition 6.10]{N}.

\vskip5pt

The triangulated structure on  $\mathsf{Ho}(\mathcal A)$ can be translated  to the one on $\frac{\mathcal C\cap\mathcal F}{\mathcal C\cap\mathcal F\cap\mathcal W}$ as follows.
\vskip5pt
$(1)$ \ Fix a quasi-inverse $R : \mathsf{Ho}(\mathcal{A}) \longrightarrow \frac{\mathcal C\cap\mathcal F}{\mathcal C\cap\mathcal F\cap\mathcal W}$ which is identical on $\frac{\mathcal C\cap\mathcal F}{\mathcal C\cap\mathcal F\cap\mathcal W}$.
Then there are natural isomorphisms $\varepsilon_X: RX \xlongrightarrow{\sim} X$ in $\mathsf{Ho}(\mathcal{A})$ for each $X\in \mathcal{A}$. In particular, if $X\in (\mathcal{C} \cap \mathcal{F})$ one has
$\varepsilon_X = {\rm Id}_X.$

\vskip5pt

$(2)$ \ The shift functor $\Sigma$ is  $R \circ [1]$.

\vskip5pt

$(3)$ \ $L \xlongrightarrow{u} M \xlongrightarrow{v} N \xlongrightarrow{w} \Sigma L$ is a distinguished triangle in $\frac{\mathcal C\cap\mathcal F}{\mathcal C\cap\mathcal F\cap\mathcal W}$ if and only if
$L \xlongrightarrow{u} M \xlongrightarrow{v} N \xlongrightarrow{\varepsilon_{L[1]}\circ w} L[1]$ is a distinguished triangle in $\mathsf{Ho}(\mathcal{A})$.

\vskip5pt

In particular, any $\mathbb E$-triangle $X\xlongrightarrow{f}Y\xlongrightarrow{g}Z\xdashrightarrow{\delta}$ in $\mathcal{C} \cap \mathcal{F}$ gives a distinguished triangle $X\xlongrightarrow {f} Y\xlongrightarrow{g}Z \xlongrightarrow{\varepsilon_{X[1]}^{-1}\circ\widetilde{\ell}(\delta)}\Sigma X$ in $\frac{\mathcal C\cap\mathcal F}{\mathcal C\cap\mathcal F\cap\mathcal W}$.
Here $R(\widetilde{\ell}(\delta)) = \varepsilon_{X[1]}^{-1}\circ\widetilde{\ell}(\delta)$.

\vskip5pt

{\bf Terminology:} \ This triangulation is denoted by $(\frac{\mathcal C\cap\mathcal F}{\mathcal C\cap\mathcal F\cap\mathcal W}, \ \Sigma, \ \triangle)$.
It gives the second extriangulated structure on $\frac{\mathcal C\cap\mathcal F}{\mathcal C\cap\mathcal F\cap\mathcal W}$. The corresponding additive bifunctor is
\[\Hom_{\frac{\mathcal C\cap\mathcal F}{\mathcal C\cap\mathcal F\cap\mathcal W}}(-, \Sigma -): (\frac{\mathcal C\cap\mathcal F}{\mathcal C\cap\mathcal F\cap\mathcal W})^{\rm op}\times \frac{\mathcal C\cap\mathcal F}{\mathcal C\cap\mathcal F\cap\mathcal W} \longrightarrow {\rm Ab}.\]
This extriangulated structure will be denoted by
$(\frac{\mathcal C\cap\mathcal F}{\mathcal C\cap\mathcal F\cap\mathcal W}, \ \Hom_{\frac{\mathcal C\cap\mathcal F}{\mathcal C\cap\mathcal F\cap\mathcal W}}(-, \Sigma -), \ \mathfrak s_\triangle)$
when it needs to be specified.
For clarity, we call it {\it the extriangulated structure arising from Nakaoka-Palu's triangulated structure.}

\vskip5pt

We now compare the two  extriangulated structures on $\frac{\mathcal C\cap\mathcal F}{\mathcal C\cap\mathcal F\cap\mathcal W}$:
\[\xymatrix{(\frac{\mathcal C\cap\mathcal F}{\mathcal C\cap\mathcal F\cap\mathcal W}, \ \overline{\mathbb E}, \ \overline {\mathfrak s}), \ \ \ \mbox{and} \ \ (\frac{\mathcal C\cap\mathcal F}{\mathcal C\cap\mathcal F\cap\mathcal W}, \ \Hom_{\frac{\mathcal C\cap\mathcal F}{\mathcal C\cap\mathcal F\cap\mathcal W}}(-, \Sigma -), \ \mathfrak s_\triangle).}\]

\vskip5pt

\begin{definition} \label{extrifunctor} \ Let $(\mathcal A, \mathbb E, \mathfrak s)$ and $(\mathcal B,\mathbb F,\mathfrak t)$ be extriangulated categories.
  \vskip5pt
  $(1)$ \ {\rm (\cite[2.32]{BTS21}; \cite[2.11]{NOS22})} \ An \emph{extriangle functor} $(G, \eta) : (\mathcal A, \mathbb E,\mathfrak s) \longrightarrow(\mathcal B, \mathbb F,\mathfrak t)$ consists of an additive functor $G : \mathcal A\longrightarrow\mathcal B$,  and
  a natural transformation $\eta: \mathbb E\longrightarrow  \mathbb F\circ (G^{\rm op}\times G)$ of bifunctors, or explicitly, with bi-natural group homomorphisms:
  \[
    \eta_{Z, X} : \mathbb E (Z, X) \longrightarrow \mathbb F(GZ, GX),  \ \ \ \forall \ X,Z \in \mathcal{A},
  \]
  such that if  $X \xlongrightarrow{f} Y \xlongrightarrow{g} Z \xdashrightarrow {\delta}$ is an $\mathbb E$-triangle in $\mathcal A$, then
  $GX \xlongrightarrow{Gf} GY \xlongrightarrow{Gg} GZ \xdashrightarrow {\eta_{Z, X}(\delta)}$ is an $\mathbb F$-triangle in $\mathcal B$.

  \vskip5pt

  $(2)$ \ {\rm (\cite[2.11, 2.13]{NOS22})} \ An extriangle functor $(G,\eta) : (\mathcal A,\mathbb E,\mathfrak s) \longrightarrow(\mathcal B,\mathbb F,\mathfrak t)$ is an \emph{extriangle-equivalence}
  if $G$ is an equivalence of categories and $\eta$ is a natural isomorphism of bifunctors.

  \vskip5pt Extriangulated categories $(\mathcal A, \mathbb E, \mathfrak s)$ and $(\mathcal B,\mathbb F,\mathfrak t)$ are {\it extriangle-equivalent}
  if there exists  an extriangle-equivalence $(G,\eta) : (\mathcal A,\mathbb E,\mathfrak s) \longrightarrow(\mathcal B,\mathbb F,\mathfrak t)$.
\end{definition}

\vskip5pt

\begin{definition} \label{subfunctor}  \ $(1)$ \ Let $\mathcal A$ be a category, $\mathbb E: \mathcal A^{\rm op} \times \mathcal A \longrightarrow {\rm Ab}$ an additive  bifunctor.
  A bifunctor $\mathbb F: \mathcal A^{\rm op} \times \mathcal A \longrightarrow {\rm Ab}$  is {\it an additive subbifunctor} of $\mathbb E$, denoted by $\mathbb F\subseteq\mathbb E$, provided that it is additive as a bifunctor, and  there exists a natural transformation $\eta: \mathbb F\longrightarrow \mathbb E$ of bifunctors, such that each group homomorphism $\eta_{Z, X}: \mathbb F(Z, X) \longrightarrow \mathbb E(Z, X)  \ (\ \forall \ X,Z \in \mathcal{A})$  \ is injective.

  \vskip5pt

  An additive subbifunctor $\mathbb F: \mathcal A^{\rm op} \times \mathcal A \longrightarrow {\rm Ab}$  is {\it a proper additive subbifunctor} of $\mathbb E$,
  provided that $\mathbb F$ is not naturally isomorphic to $\mathbb E$, i.e.,
  any natural transformation $\eta: \mathbb F\longrightarrow \mathbb E$ of bifunctors is not a natural isomorphism.

  \vskip5pt

  $(2)$ \ (\cite[Definition~3.10, Proposition 3.16]{HLN}; also \cite[Proposition~5.5]{INP24}) \  Let $(\mathcal A,\mathbb E,\mathfrak s)$ be an extriangulated category. An additive subbifunctor $\mathbb F\subseteq\mathbb E$ is \emph{closed} if $(\mathcal A, \eta(\mathbb F), \mathfrak s|_{\eta(\mathbb F)})$ is an extriangulated category.
\end{definition}

\vskip5pt

\begin{lemma}\label{lem_extriangulated_functor} \ Let $(\mathcal C, \mathcal F,\mathcal W)$ be a Hovey triple in a weakly idempotent complete extriangulated category
  $(\mathcal A,\mathbb E,\mathfrak s)$, $(\frac{\mathcal C\cap\mathcal F}{\mathcal C\cap\mathcal F\cap\mathcal W}, \ \overline{\mathbb E}, \ \overline{\mathfrak s})$ the induced extriangulated structure,
  and  $(\frac{\mathcal C\cap\mathcal F}{\mathcal C\cap\mathcal F\cap\mathcal W},  \ \Hom_{\frac{\mathcal C\cap\mathcal F}{\mathcal C\cap\mathcal F\cap\mathcal W}}(-, \Sigma -), \ \mathfrak s_\triangle)$ the extriangulated structure arising from Nakaoka-Palu's triangulated structure.
  Then there is an extriangle functor $({\rm Id}_{\frac{\mathcal C\cap\mathcal F}{\mathcal C\cap\mathcal F\cap\mathcal W}}, \eta): (\frac{\mathcal C\cap\mathcal F}{\mathcal C\cap\mathcal F\cap\mathcal W}, \ \overline{\mathbb E}, \ \overline{\mathfrak s})\longrightarrow (\frac{\mathcal C\cap\mathcal F}{\mathcal C\cap\mathcal F\cap\mathcal W},  \ \Hom_{\frac{\mathcal C\cap\mathcal F}{\mathcal C\cap\mathcal F\cap\mathcal W}}(-, \Sigma -), \ \mathfrak s_\triangle)$, where
  \begin{equation*}\label{eq_extriangulated_functor}
    \eta_{Z,X}: \overline{\mathbb E}(Z,X)\longrightarrow \operatorname{Hom}_{\frac{\mathcal C\cap\mathcal F}{\mathcal C\cap\mathcal F\cap\mathcal W}}(Z,\Sigma X)
  \end{equation*}
  is defined by $\eta_{Z,X}(\delta) =\varepsilon_{X[1]}^{-1}\circ\widetilde{\ell}(\delta)$, and  each $\eta_{Z,X}$ is injective, and the image of this  extriangle functor is precisely the extriangulated category
  $(\frac{\mathcal C\cap\mathcal F}{\mathcal C\cap\mathcal F\cap\mathcal W}, \ \eta(\overline{\mathbb E}), \ (\mathfrak s_\triangle)|_{\eta(\overline{\mathbb E})})$.

  \vskip5pt

  In particular,  $\overline{\mathbb E}$ is a closed subbifunctor of $\Hom_{\frac{\mathcal C\cap\mathcal F}{\mathcal C\cap\mathcal F\cap\mathcal W}}(-, \Sigma -)$.

\end{lemma}

\begin{proof} \ For morphisms $\overline f:Z\longrightarrow Z'$ and $\overline g:X\longrightarrow X'$ in $\frac{\mathcal C\cap\mathcal F}{\mathcal C\cap\mathcal F\cap\mathcal W}$,
  with $f:Z\longrightarrow Z'$ and $g:X\longrightarrow X'$ in $\mathcal C\cap\mathcal F$, by the binaturality of $\widetilde{\ell}$ and the naturality of $\varepsilon$, one has
  \begin{align*}
    \eta_{Z,X}(f^*\delta) & =\eta_{Z',X}(\delta)\circ\ell(f), \ \ \ \forall \ \ \delta\in\mathbb E(Z',X)      \\
    \eta_{Z,X'}(g_*\rho)  & =(\Sigma\overline g)\circ\eta_{Z,X}(\rho) \ \ \ \forall \ \ \rho\in\mathbb E(Z,X)
  \end{align*}
  where the second equality uses $\varepsilon_{X'[1]} \circ (\Sigma\overline g) =\ell(g)[1]\circ\varepsilon_{X[1]}$.
  Thus $\eta$ is a natural transformation of bifunctors.

  \vskip5pt

  We show that each $\eta_{Z,X}$ is injective.
  Suppose that $\eta_{Z,X}(\delta)=\varepsilon_{X[1]}^{-1}\circ \widetilde{\ell}(\delta)=0$, and $\delta$ is realized by an $\mathbb E$-triangle $X\xlongrightarrow{f}Y\xlongrightarrow{g}Z\xdashrightarrow{\delta}$.
  Its image $X\xlongrightarrow{f}Y\xlongrightarrow{g}Z \xlongrightarrow{0}\Sigma X$ is a splitting distinguished triangle.
  Thus the sequence $X\xlongrightarrow{\overline f}Y\xlongrightarrow{\overline g}Z$ splits in $\frac{\mathcal C\cap\mathcal F}{\mathcal C\cap\mathcal F\cap\mathcal W}$.
  Since it realizes $\delta$ in the induced extriangulated structure, $\delta=0$ in $\overline{\mathbb E}(Z,X)=\mathbb E(Z,X)$, by \textbf{ET2}.

  \vskip5pt

  By definition $\overline{\mathbb E}$ is an  additive subbifunctor of $\operatorname{Hom}_{\frac{\mathcal C\cap\mathcal F}{\mathcal C\cap\mathcal F\cap\mathcal W}}(-, \Sigma -)$.
  For each $\delta \in \overline{\mathbb E}(Z,X) = \mathbb E(Z,X)$, there is an $\mathbb E$-triangle $X\xlongrightarrow{f}Y\xlongrightarrow{g}Z\xdashrightarrow{\delta}$ in $\mathcal C\cap\mathcal F$. It gives a distinguished triangle $X\xlongrightarrow{\overline f}Y\xlongrightarrow{\overline g}Z\xlongrightarrow{\widetilde{\ell} (\delta)} X[1]$ in $\mathsf{Ho}(\mathcal{A})$.
  By construction, $X \xlongrightarrow{\overline f} Y \xlongrightarrow{\overline g} Z \xlongrightarrow{\varepsilon_{X[1]}^{-1} \circ \widetilde{\ell} (\delta)} \Sigma X$ is a distinguished triangle in $\frac{\mathcal C\cap\mathcal F}{\mathcal C\cap\mathcal F\cap\mathcal W}$.
  Hence $X \xlongrightarrow{\overline f} Y \xlongrightarrow{\overline g} Z \xlongrightarrow{\varepsilon_{X[1]}^{-1} \circ \widetilde{\ell} (\delta)} \Sigma X$ is a $\mathfrak s_{\triangle}$-realization of $\eta_{Z,X}(\delta)= \varepsilon_{X[1]}^{-1} \circ \widetilde{\ell} (\delta)$.
  It follows that a $\overline{\mathfrak s}$-realisation of $\delta$ is sent to a $\mathfrak s_\triangle$-realisation of $\eta_{Z,X}(\delta)$. This proves that $({\rm Id}_{\frac{\mathcal C\cap\mathcal F}{\mathcal C\cap\mathcal F\cap\mathcal W}}, \eta): (\frac{\mathcal C\cap\mathcal F}{\mathcal C\cap\mathcal F\cap\mathcal W}, \ \overline{\mathbb E}, \ \overline{\mathfrak s})\longrightarrow (\frac{\mathcal C\cap\mathcal F}{\mathcal C\cap\mathcal F\cap\mathcal W},  \ \Hom_{\frac{\mathcal C\cap\mathcal F}{\mathcal C\cap\mathcal F\cap\mathcal W}}(-, \Sigma -), \ \mathfrak s_\triangle)$ is an extriangle functor, and that the image of this extriangle functor is precisely the extriangulated category
  $(\frac{\mathcal C\cap\mathcal F}{\mathcal C\cap\mathcal F\cap\mathcal W}, \ \eta(\overline{\mathbb E}), \ (\mathfrak s_\triangle)|_{\eta(\overline{\mathbb E})})$.
  Therefore $\overline{\mathbb E}$ is a closed subbifunctor of $\operatorname{Hom}_{\frac{\mathcal C\cap\mathcal F}{\mathcal C\cap\mathcal F\cap\mathcal W}}(-, \Sigma -)$.
\end{proof}

\vskip5pt

\subsection{When is $\bigl(\frac{\mathcal C\cap\mathcal F}{\mathcal C\cap\mathcal F\cap\mathcal W}, \ \overline{\mathbb E}, \ \overline{\mathfrak s}\bigr)$ canonically triangulated?} $\,$

\vskip5pt

\begin{definition} {\rm(\cite[Section 3.3]{NP19})} \label{cantri} \ An extriangulated category
  $(\mathcal A,\mathbb E,\mathfrak s)$ is \emph{canonically triangulated} if there is an auto-equivalence $[1]:\mathcal A\to\mathcal A$ and natural isomorphisms $\ell_{Z,X}:\mathbb E(Z,X)\xlongrightarrow{\sim}\operatorname{Hom}_{\mathcal A}(Z,X[1])$ for all objects $X,Z\in\mathcal A$, such that $(\mathcal A,[1],\triangle)$ is triangulated, where
  \[
    \triangle=
    \left\{
    X\xlongrightarrow{f}Y\xlongrightarrow{g}Z
    \xlongrightarrow{\ell_{Z,X}(\delta)}X[1]
    \ \middle|\
    X\xlongrightarrow{f}Y\xlongrightarrow{g}Z\xdashrightarrow{\delta}
    \text{ is an $\mathbb E$-triangle}
    \right\}.
  \]\end{definition}

\vskip5pt

The following lists some equivalent characterizations.

\begin{remark} \ Let  $(\mathcal{A}, \mathbb E, \mathfrak s)$ be an extriangulated category.
  Then the following are equivalent.

  \vskip5pt
  $(1)$ \ $(\mathcal{A}, \mathbb E, \mathfrak s)$ is canonically triangulated.
  \vskip5pt
  $(2)$ \ {\rm (\cite[3.3]{Msa25})} \ $0\longrightarrow X$ and $X\longrightarrow0$ are both $\mathbb E$-inflations and $\mathbb E$-deflations for every $X\in\mathcal A$.
  \vskip5pt
  $(3)$ \ All morphisms are both $\mathbb E$-inflations and $\mathbb E$-deflations.
  \vskip5pt
  $(4)$ \ {\rm (\cite[7.4]{NP19})} \ $\mathcal A$ is Frobenius with $\mathcal P(\mathcal A)=\mathcal I(\mathcal A)=\{0\}$.
  \vskip5pt
\end{remark}

\vskip5pt

Note that if $\mathcal{A}$ is a Frobenius extriangulated category, then the stable category $\mathcal{A}/\mathcal{P}(\mathcal{A})$ has the induced extriangulated structure \textup{(cf. \Cref{extriangulatedquotint})}.

\vskip5pt

\begin{proposition}\label{prop_Frobenius_ideal_quotient} \   Let $(\mathcal A,\mathbb E,\mathfrak s)$ be a weakly idempotent complete extriangulated category.

  \vskip5pt

  $(1)$ \ {\rm (\cite[7.5]{NP19}; \cite[3.3]{Msa25})} If $\mathcal A$ is Frobenius, then the induced extriangulated category $(\mathcal A/\mathcal P(\mathcal A), \ \overline{\mathbb E}, \ \overline{\mathfrak s})$ is canonically triangulated.

  \vskip5pt

  $(2)$ \ Conversely, let $\mathcal B$ be a full additive subcategory of $\mathcal A$ with $\mathcal B\subseteq\mathcal P(\mathcal A)\cap\mathcal I(\mathcal A)$.
  If the induced extriangulated category $(\mathcal A/\mathcal B, \ \overline{\mathbb E}, \ \overline{\mathfrak s})$ is canonically triangulated, then $\mathcal A$ is Frobenius, and $P\in\mathcal P(\mathcal A)$ if and only if $P$ is a direct summand of an object of $\mathcal B$.
\end{proposition}

\begin{proof}
  $(1)$ \ See \cite[Theorem 3.3]{Msa25} for details.

  \vskip5pt

  $(2)$ \ By \Cref{extriangulatedquotint}, $\mathcal A/\mathcal B$ has the induced extriangulated structure. If it is canonically triangulated, then for any $X\in\mathcal A$ there are distinguished triangles in $\mathcal A/\mathcal B$
  \begin{equation*}
    X[-1]\longrightarrow0\longrightarrow X\xlongrightarrow{\operatorname{Id}_X}X,
    \qquad
    X\longrightarrow0\longrightarrow X[1]\xlongrightarrow{\operatorname{Id}_{X[1]}}X[1].
  \end{equation*}
  By definition, these arise from $\mathbb E$-triangles in $\mathcal A$
  \begin{equation}\label{eq_dist}
    X[-1]\longrightarrow P\longrightarrow X\dashrightarrow,
    \qquad
    X\longrightarrow I\longrightarrow X[1]\dashrightarrow,
  \end{equation}
  where $P\cong0\cong I$ in $\mathcal A/\mathcal B$. Since $\mathcal A$ is weakly idempotent complete,  $P$ and $I$ are direct summands of objects of $\mathcal B$. Since $\mathcal B\subseteq\mathcal P(\mathcal A)\cap\mathcal I(\mathcal A)$, both $P$ and $I$ are projective-injective. Thus $\mathcal A$ has enough projective objects and enough injective objects.

  \vskip5pt

  For $Q\in\mathcal P(\mathcal A)$, by \Cref{eq_dist} there is an $\mathbb E$-triangle $K\longrightarrow P\longrightarrow Q\dashrightarrow$ in $\mathcal A$ with $P\in\mathcal P(\mathcal A)\cap\mathcal I(\mathcal A)$ and $P\cong0$ in $\mathcal A/\mathcal B$. Since $\mathbb E(Q,K)=0$, the $\mathbb E$-triangle splits. Hence $Q$ is a direct summand of $P$, so $Q\cong0$ in $\mathcal A/\mathcal B$ and therefore is a direct summand of an object of $\mathcal B$. Dually, every object of $\mathcal I(\mathcal A)$ is a direct summand of an object of $\mathcal B$. Thus $\mathcal P(\mathcal A)=\mathcal I(\mathcal A)$.
\end{proof}

\vskip5pt

\begin{remark} \ The assumption of weakly idempotent completeness in Proposition \ref{prop_Frobenius_ideal_quotient} can be removed by replacing ``direct summand'' in $(2)$ with {\it retract}.
  Recall that an object $X$ of an additive category $\mathcal A$ is {\it a retract} of an object $Y$ if there are morphisms $i:X\longrightarrow Y$
  and $p:Y\longrightarrow X$ such that $p\circ i={\rm Id}_X$.
  For a full additive subcategory $\mathcal B$ of $\mathcal A$, one has $X\cong0$ in $\mathcal A/\mathcal B$ if and only if $X$ is a retract of an object of $\mathcal B$.
  In fact, $X\cong 0$ in $\mathcal A/\mathcal B$ if and only if ${\rm Id}_X$ factors through an object of $\mathcal B$.
\end{remark}

\vskip5pt

\begin{theorem}\label{thm_Frobenius_pair} \ Suppose that $(\mathcal C,\mathcal F,\mathcal W)$ is a Hovey triple in a weakly idempotent complete extriangulated category $(\mathcal A,\mathbb E,\mathfrak s)$.
  Let $(\frac{\mathcal C\cap\mathcal F}{\mathcal C\cap\mathcal F\cap\mathcal W}, \ \overline{\mathbb E}, \ \overline{\mathfrak s})$ be the induced extriangulated structure,   and let $(\frac{\mathcal C\cap\mathcal F}{\mathcal C\cap\mathcal F\cap\mathcal W},  \ \Hom_{\frac{\mathcal C\cap\mathcal F}{\mathcal C\cap\mathcal F\cap\mathcal W}}(-, \Sigma -), \ \mathfrak s_\triangle)$ be the extriangulated structure arising from Nakaoka-Palu's triangulated structure.
  Then the following are equivalent.

  \vskip5pt

  $(1)$ \ $(\mathcal C,\mathcal F,\mathcal W)$ is not exceptional, i.e., $(\mathcal C\cap\mathcal F, \ \mathcal C\cap\mathcal F\cap\mathcal W)$ is a Frobenius pair.

  \vskip5pt

  $(2)$ \ $\mathcal C\cap\mathcal F$ has enough projective objects and $\mathcal P(\mathcal C\cap\mathcal F)=\mathcal C\cap\mathcal F\cap\mathcal W$.

  \vskip5pt

  $(2')$ \ $\mathcal C\cap\mathcal F$ has enough injective objects and $\mathcal I(\mathcal C\cap\mathcal F)=\mathcal C\cap\mathcal F\cap\mathcal W$.

  \vskip5pt

  $(3)$ \ The extriangle functor $({\rm Id}_{\frac{\mathcal C\cap\mathcal F}{\mathcal C\cap\mathcal F\cap\mathcal W}}, \eta): (\frac{\mathcal C\cap\mathcal F}{\mathcal C\cap\mathcal F\cap\mathcal W}, \ \overline{\mathbb E}, \ \overline{\mathfrak s})\longrightarrow (\frac{\mathcal C\cap\mathcal F}{\mathcal C\cap\mathcal F\cap\mathcal W},  \ \Hom_{\frac{\mathcal C\cap\mathcal F}{\mathcal C\cap\mathcal F\cap\mathcal W}}(-, \Sigma -), \ \mathfrak s_\triangle)$ in {\rm \Cref{eq_extriangulated_functor}} is an extriangle-equivalence.

  \vskip5pt

  $(4)$ \ Functors  $\overline{\mathbb E}$ and $\Hom_{\frac{\mathcal C\cap\mathcal F}{\mathcal C\cap\mathcal F\cap\mathcal W}}(-, \Sigma -)$  are
  naturally isomorphic.

  \vskip5pt

  $(5)$ \  $\bigl(\frac{\mathcal C\cap\mathcal F}{\mathcal C\cap\mathcal F\cap\mathcal W}, \ \overline{\mathbb E}, \ \overline{\mathfrak s}\bigr)$ is canonically triangulated.
\end{theorem}

\begin{proof} \ $(1)\Longrightarrow(2)$: This follows from the definition of a Frobenius pair.

  \vskip5pt

  $(2)\Longrightarrow (3)$:  \ By \Cref{lem_extriangulated_functor}, the natural transformation $\eta_{Z,X}:\overline{\mathbb E}(Z,X)\longrightarrow\operatorname{Hom}_{\frac{\mathcal C\cap\mathcal F}{\mathcal C\cap\mathcal F\cap\mathcal W}}(Z,\Sigma X)$ is injective for all $X, Z\in\mathcal C\cap\mathcal F$.
  It remains to prove its surjectivity.
  For each $Z\in\mathcal C\cap\mathcal F$, choose an $\mathbb E$-triangle $K\longrightarrow P\longrightarrow Z\xdashrightarrow{\rho_Z}$ with $P\in\mathcal C\cap\mathcal F\cap\mathcal W$ and $K\in\mathcal C\cap\mathcal F$.
  Then its image in $\frac{\mathcal C\cap\mathcal F}{\mathcal C\cap\mathcal F\cap\mathcal W}$ is a distinguished triangle $K\longrightarrow0\longrightarrow Z \xlongrightarrow{c_Z}\Sigma K$,  where $c_Z=\eta_{Z,K}(\rho_Z)$ is an isomorphism.
  For any $X\in\mathcal C\cap\mathcal F$ and any $h: Z\longrightarrow\Sigma X$ in $\frac{\mathcal C\cap\mathcal F}{\mathcal C\cap\mathcal F\cap\mathcal W},$
  since $\Sigma$ is an auto-equivalence, there exists a morphism $\overline g: K\longrightarrow X$ in $\frac{\mathcal C\cap\mathcal F}{\mathcal C\cap\mathcal F\cap\mathcal W}$ such that $\Sigma\overline g=h\circ c_Z^{-1}$.
  The naturality of $\eta$ yields that
  \[
    \eta_{Z,X}(g_*\rho_Z)
    =(\Sigma\overline g)\circ\eta_{Z,K}(\rho_Z)
    =(h\circ c_Z^{-1})\circ c_Z
    =h.
  \]
  It follows that $\eta_{Z,X}$ is surjective.

  \vskip5pt

  $(3)\Longleftrightarrow (4)$: This follows directly from the definition.

  \vskip5pt

  $(3)\Longrightarrow(5)$: This follows since the target of \Cref{eq_extriangulated_functor} is canonically triangulated.

  \vskip5pt

  $(5)\Longrightarrow (1)$: \  By \Cref{prop_Frobenius_ideal_quotient}, $\mathcal C\cap\mathcal F$ is Frobenius and its projective-injective objects are precisely  $\mathcal C\cap\mathcal F\cap\mathcal W$, i.e., $(\mathcal C\cap\mathcal F, \ \mathcal C\cap\mathcal F\cap\mathcal W)$ is a Frobenius pair. \end{proof}

\subsection{Examples of homotopy categories} $\, $ \vskip5pt

Let $(\mathcal{A}, \mathbb E, \mathfrak s)$ be a weakly idempotent complete extriangulated category, and $(\mathcal C,\mathcal F,\mathcal W)$ a Hovey triple in $\mathcal A$.
The following two examples of exceptional Hovey triples illustrate how the extriangle functor in \Cref{lem_extriangulated_functor}
\[\xymatrix{
  (\mathrm{Id}_{\frac{\mathcal C\cap\mathcal F}{\mathcal C\cap\mathcal F\cap\mathcal W}}, \eta) : (\frac{\mathcal C\cap\mathcal F}{\mathcal C\cap\mathcal F\cap\mathcal W}, \ \overline{\mathbb E}, \ \overline{\mathfrak s}) \longrightarrow (\frac{\mathcal C\cap\mathcal F}{\mathcal C\cap\mathcal F\cap\mathcal W}, \ \Hom_{\frac{\mathcal C\cap\mathcal F}{\mathcal C\cap\mathcal F\cap\mathcal W}}(-, \Sigma -), \ \ \mathfrak s_\triangle)}
\]
fails to be an extriangle-equivalence, or equivalently, how the natural monomorphism $\eta_{Z,X}:\overline{\mathbb E}(Z,X) \hookrightarrow \operatorname{Hom}_{\frac{\mathcal C\cap\mathcal F}{\mathcal C\cap\mathcal F\cap\mathcal W}}(Z,\Sigma X)$ fails to be an isomorphism for some $X,Z\in\mathcal C\cap\mathcal F$.

\vskip5pt

\begin{example}\label{exmple_ideal} \ Let $A=kC_2/J^4$ and $\mathcal A=A\text{-}\mathrm{mod}$, as in \Cref{exm1}.
  One has an exceptional \textup{Hovey} triple $(\mathcal{C}, \mathcal{F}, \mathcal{W})$ of Type I with
  \begin{align*}
    \mathcal C
     & =\operatorname{add}(M_{1,1}\oplus M_{1,2}\oplus M_{1,3}
    \oplus M_{1,4}\oplus M_{2,4}),                                             \\
    \mathcal F
     & =\operatorname{add}(M_{1,1}\oplus M_{1,3}\oplus M_{2,2}
    \oplus M_{1,4}\oplus M_{2,4}),                                             \\
    \mathcal W
     & =\operatorname{add}(M_{2,1}\oplus M_{1,3}\oplus M_{1,4}\oplus M_{2,4}).
  \end{align*}
  By \Cref{lem_sinj_hovey_exceptional} and the formulas in
  \Cref{eq_projectives_sinj_exceptional,eq_injectives_sinj_exceptional}, one has
  \begin{align*}
    \mathcal C\cap\mathcal F
                                            & =\operatorname{add}(M_{1,1}\oplus M_{1,3}\oplus M_{1,4}\oplus M_{2,4}), \\
    \mathcal C\cap\mathcal F\cap\mathcal W
                                            & =\operatorname{add}(M_{1,3}\oplus M_{1,4}\oplus M_{2,4}),               \\
    \mathcal P(\mathcal C \cap \mathcal{F}) & =\mathcal I(\mathcal C \cap \mathcal{F})
    =\mathcal C \cap \mathcal{F}.
  \end{align*}
  Thus
  $\frac{\mathcal C\cap\mathcal F}{\mathcal C\cap\mathcal F\cap\mathcal W}= \operatorname{add}(M_{1,1})$.
  We compute $\Sigma M_{1,1}$ by the definition $\Sigma X=RC_X$.
  Take
  \begin{equation*}\label{eq_type_I_shift}
    0\longrightarrow M_{1,1}\longrightarrow M_{1,3}\longrightarrow M_{1,2}
    \longrightarrow0
  \end{equation*}
  with $V_X=M_{1,3}\in\mathcal F\cap\mathcal W$ and
  $C_X=M_{1,2}\in\mathcal C$.
  Then consider the exact sequence
  \begin{equation*}\label{eq_type_I_replacement}
    0\longrightarrow M_{1,2}\longrightarrow M_{1,1}\oplus M_{1,4}
    \longrightarrow M_{1,3}\longrightarrow0,
  \end{equation*}
  with $M_{1,1}\oplus M_{1,4}\in\mathcal C \cap \mathcal{F}$ and $M_{1,3}\in\mathcal C\cap\mathcal W$.
  Hence $RC_X \cong M_{1,1}\oplus M_{1,4}$ in $\frac{\mathcal C\cap\mathcal F}{\mathcal C\cap\mathcal F\cap\mathcal W}$.
  Since $M_{1,4}\in\mathcal{C} \cap \mathcal{F} \cap \mathcal{W}$, it follows that $\Sigma M_{1,1}= M_{1,1}$.

  \vskip5pt

  Since $\mathcal P(\mathcal C \cap \mathcal{F}) =\mathcal C \cap \mathcal{F}$, $\overline{\mathbb E}(M_{1,1},M_{1,1})= 0$.
  However
  $$\operatorname{Hom}_{\frac{\mathcal C\cap\mathcal F}{\mathcal C\cap\mathcal F\cap\mathcal W}} (M_{1,1}, \Sigma M_{1,1}) = \operatorname{Hom}_{\frac{\mathcal C\cap\mathcal F}{\mathcal C\cap\mathcal F\cap\mathcal W}}(M_{1,1},M_{1,1}) \cong k.$$
  Hence the induced extriangulated category $(\frac{\mathcal C\cap\mathcal F}{\mathcal C\cap\mathcal F\cap\mathcal W}, \ \overline{\mathbb E}, \ \overline{\mathfrak s})$ is not canonically triangulated, and the extriangle functor
  \[\xymatrix{
    (\mathrm{Id}_{\frac{\mathcal C\cap\mathcal F}{\mathcal C\cap\mathcal F\cap\mathcal W}}, \eta) : (\frac{\mathcal C\cap\mathcal F}{\mathcal C\cap\mathcal F\cap\mathcal W}, \ \overline{\mathbb E}, \ \overline{\mathfrak s})
    \longrightarrow (\frac{\mathcal C\cap\mathcal F}{\mathcal C\cap\mathcal F\cap\mathcal W}, \ \Hom_{\frac{\mathcal C\cap\mathcal F}{\mathcal C\cap\mathcal F\cap\mathcal W}}(-, \Sigma -), \ \ \mathfrak s_\triangle)}
  \]
  is not an extriangle-equivalence.
\end{example}

\vskip5pt

\begin{example}\label{exm4} \ Let $A=kC_3/J^3$ and $\mathcal A=A\text{-}\mathrm{mod}$, as in \Cref{exm2}.
  One has an exceptional \textup{Hovey} triple $(\mathcal{C}, \mathcal{F}, \mathcal{W})$ of Type II with
  \begin{align*}
    \mathcal C
     & =\operatorname{add}(M_{1,1}\oplus M_{1,2}\oplus M_{2,1}\oplus M_{2,2}
    \oplus M_{1,3}\oplus M_{2,3}\oplus M_{3,3}),                             \\
    \mathcal F
     & =\operatorname{add}(M_{1,1}\oplus M_{1,2}\oplus M_{2,1}\oplus M_{3,2}
    \oplus M_{1,3}\oplus M_{2,3}\oplus M_{3,3}),                             \\
    \mathcal W
     & =\operatorname{add}(M_{3,1}\oplus M_{1,2}\oplus M_{1,3}
    \oplus M_{2,3}\oplus M_{3,3}).
  \end{align*}
  By \Cref{lem_sinj_hovey_exceptional} and the formulas in
  \Cref{eq_projectives_sinj_exceptional,eq_injectives_sinj_exceptional}, one has
  \begin{align*}
    \mathcal C\cap\mathcal F
     & =\operatorname{add}(M_{1,1}\oplus M_{1,2}\oplus M_{2,1} \oplus M_{1,3}\oplus M_{2,3}\oplus M_{3,3}), \\
    \mathcal C\cap\mathcal F\cap\mathcal W
     & =\operatorname{add}(M_{1,2}\oplus M_{1,3}\oplus M_{2,3}\oplus M_{3,3}),                              \\
    \mathcal P(\mathcal C \cap \mathcal{F})
     & =\operatorname{add}(M_{1,2}\oplus M_{2,1}\oplus M_{1,3} \oplus M_{2,3}\oplus M_{3,3}),               \\
    \mathcal I(\mathcal C \cap \mathcal{F})
     & =\operatorname{add}(M_{1,1}\oplus M_{1,2}\oplus M_{1,3} \oplus M_{2,3}\oplus M_{3,3}).
  \end{align*}
  Thus  $\frac{\mathcal C\cap\mathcal F}{\mathcal C\cap\mathcal F\cap\mathcal W}= \operatorname{add}(M_{1,1}\oplus M_{2,1})$.
  By the same argument one sees  $\Sigma M_{1,1} = M_{2,1}$.

  \vskip5pt

  Since $M_{1,1}\in\mathcal I(\mathcal C \cap \mathcal F)$,  it follows that $\overline{\mathbb E}(M_{2,1},M_{1,1})= 0$.
  However
  $$\operatorname{Hom}_{\frac{\mathcal C\cap\mathcal F}{\mathcal C\cap\mathcal F\cap\mathcal W}} (M_{2,1}, \Sigma M_{1,1}) = \operatorname{Hom}_{\frac{\mathcal C\cap\mathcal F}{\mathcal C\cap\mathcal F\cap\mathcal W}}(M_{2,1},M_{2,1}) \cong k.$$
  Hence the induced extriangulated category $(\frac{\mathcal C\cap\mathcal F}{\mathcal C\cap\mathcal F\cap\mathcal W}, \ \overline{\mathbb E}, \ \overline{\mathfrak s})$ is not canonically triangulated, and the extriangle functor
  \[\xymatrix{
    (\mathrm{Id}_{\frac{\mathcal C\cap\mathcal F}{\mathcal C\cap\mathcal F\cap\mathcal W}}, \eta) : (\frac{\mathcal C\cap\mathcal F}{\mathcal C\cap\mathcal F\cap\mathcal W}, \ \overline{\mathbb E}, \ \overline{\mathfrak s})
    \longrightarrow (\frac{\mathcal C\cap\mathcal F}{\mathcal C\cap\mathcal F\cap\mathcal W}, \ \Hom_{\frac{\mathcal C\cap\mathcal F}{\mathcal C\cap\mathcal F\cap\mathcal W}}(-, \Sigma -), \ \mathfrak s_\triangle)}
  \]
  is not an extriangle-equivalence.
\end{example}

\subsection{A remark on exceptional Hovey triples of Type I} $\,$

\vskip5pt

Let $(\mathcal{C}, \mathcal{F}, \mathcal{W})$ be an exceptional Hovey triple of \textup{Type I} in a weakly idempotent complete extriangulated category $(\mathcal A, \mathbb E, \mathfrak s)$, namely, $(\mathcal{C} \cap \mathcal{F})$ is a Frobenius extriangulated category with $\mathcal{C} \cap \mathcal{F} \cap \mathcal{W} \subsetneqq \mathcal{P}(\mathcal{C} \cap \mathcal{F})$.
Then there are two triangulated categories: one is $(\frac{\mathcal{C} \cap \mathcal{F}}{\mathcal{C} \cap \mathcal{F}\cap \mathcal W}, \ \Sigma, \ \triangle)$ inherited from
Nakaoka-Palu's triangulated structure on the homotopy category $\mathsf{Ho}(\mathcal A)$;
the other is the stable category $\frac{\mathcal{C} \cap \mathcal{F}}{\mathcal P(\mathcal{C} \cap \mathcal{F})}$ by \Cref{prop_Frobenius_ideal_quotient}.
Note that there is a canonical additive quotient functor
\[\xymatrix{\Pi : \frac{\mathcal{C} \cap \mathcal{F}}{\mathcal{C} \cap \mathcal{F}\cap \mathcal W}\longrightarrow \frac{\mathcal{C} \cap \mathcal{F}}{\mathcal P(\mathcal{C} \cap \mathcal{F})}}
\].
We will see that $\Pi$ is not necessarily a triangle functor.

\begin{example}
  Take $A=kC_2/J^6$.
  By \Cref{lem_sinj_hovey_exceptional} there is a \textup{Hovey} triple $(\mathcal{C}, \mathcal{F}, \mathcal{W})$
  \begin{align*}
    \mathcal{C} & = \operatorname{add}(M_{1,1} \oplus M_{1,2} \oplus M_{1,3} \oplus M_{1,4} \oplus M_{1,5} \oplus M_{1,6} \oplus M_{2,6} ), \\
    \mathcal{F} & = \operatorname{add}(M_{1,1} \oplus M_{2,2} \oplus M_{1,3} \oplus M_{2,4} \oplus M_{1,5} \oplus M_{1,6} \oplus M_{2,6} ), \\
    \mathcal{W} & = \operatorname{add}(M_{2,1} \oplus M_{1,5} \oplus M_{1,6} \oplus M_{2,6}).
  \end{align*}
  By \Cref{type} this Hovey triple is of Type I, where
  \begin{align*}
    \mathcal{C} \cap \mathcal{F}                  & = \operatorname{add}(M_{1,1} \oplus M_{1,3} \oplus M_{1,5} \oplus M_{1,6} \oplus M_{2,6}), \\
    \mathcal{C} \cap \mathcal{F} \cap \mathcal{W} & = \operatorname{add}(M_{1,5} \oplus M_{1,6} \oplus M_{2,6}),                               \\
    \mathcal{P}(\mathcal{C} \cap \mathcal{F})     & = \operatorname{add}(M_{1,1} \oplus M_{1,5} \oplus M_{1,6} \oplus M_{2,6}).
  \end{align*}

  Consider the exact sequence $0 \longrightarrow M_{1,1} \longrightarrow M_{1,5}\longrightarrow M_{1,4} \longrightarrow 0$ in $\mathcal{A}$ with $M_{1,1} \in \mathcal{P}(\mathcal{C} \cap \mathcal{F})$, $M_{1,5} \in \mathcal{C} \cap \mathcal{F} \cap \mathcal{W}$ and $M_{1,4} \in \mathcal{C}$.
  Then $M_{1,1}[1] \cong M_{1,4}$ in $\mathsf{Ho}(\mathcal{A})$.
  Consider the exact sequence $0 \longrightarrow M_{2,1} \longrightarrow M_{1,4}\longrightarrow M_{1,3} \longrightarrow 0$ in $\mathcal{A}$ with $M_{2,1} \in \mathcal{W}$ and $M_{1,3} \in \mathcal{C} \cap \mathcal{F}$.
  It follows that $M_{1,4} \cong M_{1,3}$ in $\mathsf{Ho}(\mathcal{A})$,
  and hence $\Sigma M_{1,1} \cong M_{1,3}$ in $\frac{\mathcal C\cap\mathcal F}{\mathcal C\cap\mathcal F\cap\mathcal W}$.
  Since $M_{1,1} \in \mathcal{P}(\mathcal{C} \cap \mathcal{F})$ and $M_{1,3} \notin \mathcal{P}(\mathcal{C} \cap \mathcal{F})$, one sees that $\Pi(M_{1,1}) = 0$ and $\Pi (\Sigma M_{1,1}) \neq 0$.
  It follows that $\Pi (\Sigma M_{1,1})$ is not isomorphic to the suspension of $\Pi M_{1,1}$ in $\frac{\mathcal{C} \cap \mathcal{F}}{\mathcal P(\mathcal{C} \cap \mathcal{F})}$.
  Hence the canonical quotient functor $\Pi : \frac{\mathcal C\cap\mathcal F}{\mathcal C\cap\mathcal F\cap\mathcal W}\longrightarrow \frac{\mathcal{C} \cap \mathcal{F}}{\mathcal P(\mathcal{C} \cap \mathcal{F})}$ is not a triangle functor.
\end{example}

\section{\bf Extriangulated structures arising from an exceptional model structure}

\vskip5pt

Let $(\mathcal C, \mathcal F,\mathcal W)$ be a Hovey triple in a weakly idempotent complete extriangulated category. By Theorem \ref{thm_Frobenius_pair},
if this Hovey triple is exceptional, then there are two different extriangulated structures on $\frac{\mathcal C\cap\mathcal F}{\mathcal C\cap\mathcal F\cap\mathcal W}$,  and an extriangle functor between them:
\[\xymatrix{
  (\mathrm{Id}_{\frac{\mathcal C\cap\mathcal F}{\mathcal C\cap\mathcal F\cap\mathcal W}}, \eta) : (\frac{\mathcal C\cap\mathcal F}{\mathcal C\cap\mathcal F\cap\mathcal W}, \ \overline{\mathbb E}, \ \overline{\mathfrak s}) \longrightarrow (\frac{\mathcal C\cap\mathcal F}{\mathcal C\cap\mathcal F\cap\mathcal W}, \ \Hom_{\frac{\mathcal C\cap\mathcal F}{\mathcal C\cap\mathcal F\cap\mathcal W}}(-, \Sigma -), \ \ \mathfrak s_\triangle)}
\]
which is not an extriangle-equivalence. We will relate ghost maps (cf. \cite{Chr98} and \cite{Bel15}) to the natural transformation $\eta$.

\vskip5pt

Also, if this Hovey triple is exceptional, then  $\overline{\mathbb E}$ is a closed proper subbifunctor of $\Hom_{\frac{\mathcal C\cap\mathcal F}{\mathcal C\cap\mathcal F\cap\mathcal W}}(-, \Sigma -)$.
Using Y. Ogawa's work \cite{Oga21} on  the defect category and H. Enomoto's work  \cite{Eno21} on closed additive subbifunctors,
Theorem \ref{thm_enomoto_ogawa} classifies all the closed additive subbifunctors $\mathbb F$ with
$$\overline{\mathbb E}\subseteq\mathbb F\subseteq\operatorname{Hom}_{\frac{\mathcal C\cap\mathcal F}{\mathcal C\cap\mathcal F\cap\mathcal W}}(-, \Sigma-)$$
when $\mathcal C\cap\mathcal F$ is essentially small and  has enough projective objects.
These  intermediate closed additive subbifunctors $\mathbb F$ are in one-to-one correspondence with Serre subcategories of  $\mathrm{fp}((\frac{\mathcal P(\mathcal C\cap\mathcal F)}{\mathcal C\cap\mathcal F\cap\mathcal W})^{\mathrm{op}}, \mathrm{Ab})$,
and each such $\mathbb F$  produces  an intermediate  extriangulated structure $(\frac{\mathcal C\cap\mathcal F}{\mathcal C\cap\mathcal F\cap\mathcal W}, \ \mathbb F, \ (\mathfrak s_\triangle)|_\mathbb F)$.

\vskip5pt

\subsection {Ghost ideals} $\,$ \vskip5pt

Let $(\mathcal C, \mathcal F,\mathcal W)$ be a Hovey triple in a weakly idempotent complete extriangulated category. The aim of this subsection is to relate
ghost maps to the natural transformation $\eta$ between the two extriangulated structures on $\frac{\mathcal C\cap \mathcal F}{\mathcal C\cap \mathcal F\cap \mathcal W}$, as given in \Cref{lem_extriangulated_functor}.

\vskip5pt

Let $\mathcal{A}$ be an additive category, and $\mathcal{X}$ a full additive subcategory of $\mathcal{A}$.
For each $U,V \in \mathcal{A}$, define
\begin{align*}
  \mathrm{Gh}_\mathcal{X}(U,V)   & = \{ f : U \longrightarrow V \mid \text{$f \circ a = 0 $ for any $a : X \longrightarrow U$ and $X \in \mathcal{X}$}\}, \\
  \mathrm{CoGh}_\mathcal{X}(U,V) & = \{ f : U \longrightarrow V \mid \text{$b \circ f = 0 $ for any $b : V \longrightarrow X$ and $X \in \mathcal{X}$}\}.
\end{align*}

In particular,  $\mathcal{X} = 0$ if and only if  $\mathrm{Gh}_\mathcal{X}(U,V)  = \Hom_\mathcal A (U,V) =\mathrm{CoGh}_\mathcal{X}(U,V)$ for all $U,V \in \mathcal{A}$.
It is clear that $\mathrm{Gh}_\mathcal{X}$ and $\mathrm{CoGh}_\mathcal{X}$ are ideals of $\mathcal{A}$.

\vskip5pt

J. D. Christensen \cite[Definition~2.5]{Chr98} introduced maps of this kind and called them {\it null maps}.
Here we use Beligiannis' terminology in \cite[Definition~2.1]{Bel15}, and call the maps in $\mathrm{Gh}_\mathcal{X}(U,V)$ {\it ghost maps}.

\vskip5pt

\begin{proposition}\label{prop_ghosts} \ Let $(\mathcal{C}, \mathcal{F}, \mathcal{W})$ be a Hovey triple in a weakly idempotent complete extriangulated category $(\mathcal A, \mathbb E, \mathfrak s)$,
  and $(\frac{\mathcal{C} \cap \mathcal{F}}{\mathcal{C} \cap \mathcal{F}\cap \mathcal W}, \ \overline{\mathbb E}, \ \overline{\mathfrak s})$ and  $(\frac{\mathcal C\cap\mathcal F}{\mathcal C\cap\mathcal F\cap\mathcal W},  \ \Hom_{\frac{\mathcal C\cap\mathcal F}{\mathcal C\cap\mathcal F\cap\mathcal W}}(-, \Sigma -), \ \mathfrak s_\triangle)$ the induced extriangulated structure and the extriangulated structure arising from Nakaoka-Palu's triangulated structure, respectively.
  Then for $X, Z\in\mathcal C\cap\mathcal F$ one has

  \vskip5pt

  $(1)$ \ If $\mathcal C\cap\mathcal F$ has enough projective objects, then
  \[
    \overline{\mathbb E}(Z,X)=\mathbb E(Z,X)\stackrel {\eta_{Z, X}}\cong
    \mathrm{Gh}_{\frac{\mathcal P(\mathcal C\cap\mathcal F)}{\mathcal C\cap\mathcal F\cap\mathcal W}}(Z, \Sigma X)
    \subseteq \operatorname{Hom}_{\frac{\mathcal{C} \cap \mathcal{F}}{\mathcal{C} \cap \mathcal{F}\cap \mathcal W}}(Z, \Sigma X),\]
  where $\eta_{Z, X}$ is given in {\rm \Cref{eq_extriangulated_functor}}.
  \vskip5pt

  $(1')$ \ If $\mathcal C\cap\mathcal F$ has enough injective objects, then
  \[
    \overline{\mathbb E}(Z,X)=\mathbb E(Z,X)
    \stackrel {\eta_{Z, X}}\cong
    \mathrm{CoGh}_{\Sigma(\frac{\mathcal I(\mathcal C\cap\mathcal F)}{\mathcal C\cap\mathcal F\cap\mathcal W})}(Z,\Sigma X)
    \subseteq \operatorname{Hom}_{\frac{\mathcal{C} \cap \mathcal{F}}{\mathcal{C} \cap \mathcal{F}\cap \mathcal W}}(Z, \Sigma X).
  \]
\end{proposition}

\begin{proof} $(1)$ \ By \Cref{eq_extriangulated_functor} one has the injective natural transformation
  \[\eta_{Z,X}: \overline{\mathbb E}(Z,X)\longrightarrow   \mathrm{Hom}_{\frac{\mathcal{C} \cap \mathcal{F}}{\mathcal{C} \cap \mathcal{F}\cap \mathcal W}}(Z,\Sigma X), \qquad \delta \longmapsto \varepsilon_{X[1]}^{-1}\circ \widetilde{\ell}(\delta).
  \]
  To see $\eta_{Z,X} (\mathbb E(Z,X)) \subseteq \mathrm{Gh}_{\frac{\mathcal P(\mathcal C\cap\mathcal F)}{\mathcal C\cap\mathcal F\cap\mathcal W}}(Z, \Sigma X)$, take $\delta \in \mathbb E(Z,X)$, $P \in \mathcal{P}(\mathcal C\cap\mathcal F)$,
  and a morphism $a:P\longrightarrow Z$ in $\mathcal C\cap\mathcal F$. Since $P$ is a projective object of $\mathcal C\cap\mathcal F$, $\overline{\mathbb E}(P, X) =0$.
  Thus $a^* \delta\in \overline{\mathbb E}(P, X) =0$, and then  by the identity (\ref{binatl}) one has
  \[\eta_{Z,X}(\delta) \circ \overline a
    = \varepsilon_{X[1]}^{-1}\circ \widetilde{\ell}(\delta) \circ a
    = \varepsilon_{X[1]}^{-1}\circ \widetilde{\ell}(a^* \delta) = \eta_{Z,X}(a^* \delta)  = 0.
  \]
  Conversely, for  any $h \in \mathrm{Gh}_{\frac{\mathcal P(\mathcal C\cap\mathcal F)}{\mathcal C\cap\mathcal F\cap\mathcal W}}(Z, \Sigma X)$, since $\mathcal C\cap\mathcal F$ has enough projective objects, there is an $\mathbb E$-triangle $K \longrightarrow P \xlongrightarrow{p} Z \xdashrightarrow{\delta}$ with $P \in \mathcal{P}(\mathcal C\cap\mathcal F)$ and $K \in \mathcal C\cap\mathcal F$.
  By definition $h \circ p = 0$ in $\frac{\mathcal{C} \cap \mathcal{F}}{\mathcal{C} \cap \mathcal{F}\cap \mathcal W}$.
  Hence there is a morphism $s : \Sigma K\longrightarrow\Sigma X$ such that $h = s \circ \varepsilon^{-1}_{K[1]} \circ \widetilde{\ell}(\delta)$:
  \[
    \xymatrix@R=15pt{
    K \ar@{->}[r] & P \ar@{->}[r]^-{p} & Z \ar@{->}[r]^-{\varepsilon^{-1}_{K[1]} \circ \widetilde{\ell}(\delta)} \ar@{->}[d]_-{h} & \Sigma K \ar@{.>}[ld]^-{s} \\
    &  & \Sigma X. &
    }
  \]
  Thus $s = \Sigma \bar t$ for a morphism $t: K\longrightarrow X$ in $\mathcal C\cap\mathcal F$, by the commutative square and \Cref{binatl}
  \[
    \xymatrix@R=15pt{
    K[1] \ar@{->}[r]^-{\varepsilon_{K[1]}^{-1}}\ar@{->}[d]_-{\overline t[1]} & R(K[1]) \ar@{=}[r] \ar@{->}[d]_-{\Sigma \bar t}  & \Sigma K \ar@{->}[d]^-{\Sigma \bar t = s} \\
    X[1] \ar@{->}[r]^-{\varepsilon_{X[1]}^{-1}} & R(X[1]) \ar@{=}[r] & \Sigma X\\
    }
  \]
  one has
  $$h = \Sigma \bar t \circ \varepsilon^{-1}_{K[1]} \circ \widetilde{\ell}(\delta) = \varepsilon^{-1}_{X[1]}\circ \bar t[1]\circ \widetilde{\ell}(\delta)
    = \varepsilon^{-1}_{X[1]}\circ \widetilde{\ell}(t_*\delta) = \eta_{Z,X}(t_* \delta)$$
  Since $t_* \delta \in \mathbb E(Z,X)$,  it follows that $h \in \eta_{Z,X}(\mathbb E(Z,X))$.

  \vskip5pt

  $(1')$ \ To see $\eta_{Z,X}(\mathbb E(Z,X))\subseteq
    \mathrm{CoGh}_{\Sigma(\frac{\mathcal I(\mathcal C\cap\mathcal F)}{\mathcal C\cap\mathcal F\cap\mathcal W})}(Z,\Sigma X)$,
  take $\delta\in\mathbb E(Z,X)$, $I\in\mathcal I(\mathcal C\cap\mathcal F)$,
  and a morphism $b:\Sigma X\to\Sigma I$ in
  $\frac{\mathcal{C} \cap \mathcal{F}}{\mathcal{C} \cap \mathcal{F}\cap \mathcal W}$.
  Since $\Sigma$ is an equivalence, $b=\Sigma\overline a$ for some morphism
  $a:X\to I$ in $\mathcal C\cap\mathcal F$.
  Since $I$ is injective, $a_*\delta\in\overline{\mathbb E}(Z,I)=0$.
  Hence, by the naturality of $\varepsilon$ and the identity \eqref{binatl}, one has
  \[
    b\circ\eta_{Z,X}(\delta)
    =(\Sigma\overline a)\circ\varepsilon_{X[1]}^{-1}
    \circ\widetilde{\ell}(\delta)
    =\varepsilon_{I[1]}^{-1}\circ\overline a[1]
    \circ\widetilde{\ell}(\delta)
    =\varepsilon_{I[1]}^{-1}\circ\widetilde{\ell}(a_*\delta)
    =\eta_{Z,I}(a_*\delta)=0.
  \]

  Conversely, for any
  $h\in\mathrm{CoGh}_{\Sigma(\frac{\mathcal I(\mathcal C\cap\mathcal F)}{\mathcal C\cap\mathcal F\cap\mathcal W})}(Z,\Sigma X)$,
  since $\mathcal C\cap\mathcal F$ has enough injective objects, there is an
  $\mathbb E$-triangle
  $X\xlongrightarrow{i}I\longrightarrow K\xdashrightarrow{\delta}$
  with $I\in\mathcal I(\mathcal C\cap\mathcal F)$ and
  $K\in\mathcal C\cap\mathcal F$.
  By definition, $(\Sigma\overline i)\circ h=0$.
  Hence, by the exactness of $\operatorname{Hom}(Z,-)$ for the associated
  distinguished triangle, there is a morphism $\overline s:Z\longrightarrow K$ such that
  $h=\varepsilon_{X[1]}^{-1}\circ\widetilde{\ell}(\delta)\circ\overline s$
  \[
    \xymatrix@R=15pt{
    & & Z \ar@{.>}[d]_-{\overline s} \ar@{->}[dr]^-h & \\
    X \ar@{->}[r]^-{\overline i} & I \ar@{->}[r] & K
    \ar@{->}[r]_-{\varepsilon_{X[1]}^{-1}\circ\widetilde{\ell}(\delta)}
    & \Sigma X .
    }
  \]
  Choosing a representative $s:Z\longrightarrow K$ in $\mathcal C\cap\mathcal F$, by
  \eqref{binatl} one has
  \[
    h=\varepsilon_{X[1]}^{-1}\circ\widetilde{\ell}(\delta)\circ\overline s
    =\varepsilon_{X[1]}^{-1}\circ\widetilde{\ell}(s^*\delta)
    =\eta_{Z,X}(s^*\delta).
  \]
  Since $s^*\delta\in\mathbb E(Z,X)$, it follows that
  $h\in\eta_{Z,X}(\mathbb E(Z,X))$.\end{proof}


\begin{remark}
  The subcategory $\overline{\mathcal I}$ need not be closed under $\Sigma$, even when $\mathcal C\cap\mathcal F$ has enough injective objects.
  Indeed, in \Cref{exm4}, $\mathcal C\cap\mathcal F$ has enough injective objects by \Cref{corollary_D_enough_PI}.
  Then one has $\overline{\mathcal I}=\operatorname{add}M_{1,1}$, but $\Sigma M_{1,1}=M_{2,1} \notin \overline{\mathcal{I}}$.
\end{remark}

\vskip5pt

\subsection{Defects and intermediate extriangulated structures} $\,$

\vskip5pt

Let $(\mathcal C, \mathcal F,\mathcal W)$ be an exceptional  Hovey triple in a weakly idempotent complete extriangulated category
$(\mathcal A,\mathbb E,\mathfrak s)$. Then one has two different extriangulated structures on $\frac{\mathcal C\cap\mathcal F}{\mathcal C\cap\mathcal F\cap\mathcal W}$:
\[\xymatrix{(\frac{\mathcal C\cap\mathcal F}{\mathcal C\cap\mathcal F\cap\mathcal W}, \ \overline{\mathbb E}, \ \overline{\mathfrak s}) \ \ \ \mbox{and} \ \ \
    (\frac{\mathcal C\cap\mathcal F}{\mathcal C\cap\mathcal F\cap\mathcal W},  \ \Hom_{\frac{\mathcal C\cap\mathcal F}{\mathcal C\cap\mathcal F\cap\mathcal W}}(-, \Sigma -), \ \mathfrak s_\triangle)}\]
and  $\overline{\mathbb E}$ is a {\it closed} subbifunctor of $\Hom_{\frac{\mathcal C\cap\mathcal F}{\mathcal C\cap\mathcal F\cap\mathcal W}}(-, \Sigma -)$ (cf. Lemma \ref{lem_extriangulated_functor}).
Using defect categories in \cite{Oga21} and the isomorphism of posets in \cite{Eno21},
Theorem \ref{thm_enomoto_ogawa} classifies all the closed additive subbifunctors $\mathbb F$ with
$\overline{\mathbb E}\subseteq\mathbb F\subseteq\operatorname{Hom}_{\frac{\mathcal C\cap\mathcal F}{\mathcal C\cap\mathcal F\cap\mathcal W}}(-, \Sigma-)$,
in case $\mathcal C\cap\mathcal F$ is essentially small and  has enough projective objects.
These  intermediate closed additive subbifunctors $\mathbb F$ are in one-to-one correspondence with Serre subcategories of  $\mathrm{fp}((\frac{\mathcal P(\mathcal C\cap\mathcal F)}{\mathcal C\cap\mathcal F\cap\mathcal W})^{\mathrm{op}}, \mathrm{Ab})$,
and each such $\mathbb F$  gives an  extriangulated structure $(\frac{\mathcal C\cap\mathcal F}{\mathcal C\cap\mathcal F\cap\mathcal W}, \ \mathbb F, \ (\mathfrak s_\triangle)|_\mathbb F)$ between
the two known extriangulated structures.

\vskip5pt

Let $\mathcal A$ be an essentially small additive category, and $\mathrm{Fun}(\mathcal A^{\mathrm{op}}, \mathrm{Ab})$ the category of contravariant additive functors from $\mathcal A$ to ${\rm Ab}$.
Then $\mathrm{Fun}(\mathcal A^{\mathrm{op}},\mathrm{Ab})$ is an abelian category.
Denote by $\mathrm{fp}(\mathcal A^{\mathrm{op}}, \mathrm{Ab})$ the full subcategory of $\mathrm{Fun}(\mathcal A^{\mathrm{op}},\mathrm{Ab})$ consisting of finitely presented functors, i.e., $F\in\mathrm{fp}(\mathcal A^{\mathrm{op}},\mathrm{Ab})$ if and only if there is a morphism $g: Y\longrightarrow Z$ in $\mathcal A$ such that there is  an exact sequence in  $\mathrm{Fun}(\mathcal A^{\mathrm{op}},\mathrm{Ab})$:
\[\Hom_\mathcal A(-,Y)\xlongrightarrow{\Hom_\mathcal A(-,g)}
  \Hom_\mathcal A(-,Z)\longrightarrow F\longrightarrow 0.
\]
Denote by ${\rm coh}\mathcal A$ the full subcategory of $\mathrm{Fun}(\mathcal A^{\mathrm{op}},\mathrm{Ab})$ consisting of finitely presented functors $F$ such that each finitely generated subfunctor of $F$ is
finitely presented.

\vskip5pt

Let $\mathcal{A}$ be an additive category, and $g : Y \longrightarrow Z$ a morphism in $\mathcal{A}$. By definition, a \emph{weak kernel} of $g$ is a morphism $f:X\longrightarrow Y$ such that $g\circ f=0$ and, for any $h:W\longrightarrow Y$ with $g\circ h=0$, there is a morphism $h':W\longrightarrow X$ satisfying $f\circ h'=h$.
The category $\mathcal{A}$ has weak kernels, provided that each morphism in $\mathcal{A}$ has a weak kernel.

\vskip5pt

\begin{lemma} \label{lem_weak_kernels} \ Let $\mathcal{A}$ be an essentially small additive category. Then

  \vskip5pt

  $(1)$ \ {\rm (\cite[Theorem 1.4]{Fre66})} \ $\mathrm{fp}(\mathcal A^{\mathrm{op}},\mathrm{Ab})$
  is  an exact abelian subcategory of $\mathrm{Fun}(\mathcal A^{\mathrm{op}},\mathrm{Ab})$ if and only if $\mathcal A$ has weak kernels.

  \vskip5pt

  $(2)$ \ {\rm (\cite[Proposition 1.5]{Her})} \ ${\rm coh}\mathcal{A}$ is  an exact abelian subcategory of $\mathrm{Fun}(\mathcal A^{\mathrm{op}},\mathrm{Ab})$.
\end{lemma}

\vskip5pt

Note that if  $\mathcal A$ has weak kernels then  $\mathrm{fp}(\mathcal A^{\mathrm{op}},\mathrm{Ab}) = {\rm coh}\mathcal{A}$.

\vskip5pt

Let $(\mathcal A, \mathbb E, \mathfrak s)$ be an essentially small extriangulated category. For $\delta\in\mathbb E(Z,X)$ and an $\mathbb E$-triangle $X\xlongrightarrow{f}Y\xlongrightarrow{g}Z\xdashrightarrow{\delta}$,
\emph{the contravariant defect} (or simply, \emph{the defect}) of $\delta$ (\cite[p.128]{ARS}; \cite[Definition~2.4]{Oga21}) is the finitely presented functor
\[
  \delta^*
  :=\operatorname{Coker}\bigl(
  \Hom_\mathcal A(-,Y)
  \xlongrightarrow{\Hom_\mathcal A(-,g)}
  \Hom_\mathcal A(-,Z)
  \bigr)
  \in\mathrm{fp}(\mathcal A^{\mathrm{op}},\mathrm{Ab}).
\]
Denote by $\operatorname{def}\mathbb E$ the full subcategory of $\mathrm{fp}(\mathcal A^{\mathrm{op}},\mathrm{Ab})$ consisting of functors isomorphic to defects.

\vskip5pt

\begin{theorem}\label{prop_ogawa_colocalisation} \ Let $(\mathcal A,\mathbb E,\mathfrak s)$ be an essentially small extriangulated category. Then

  \vskip5pt

  $(1)$ \ {\rm (\cite[Proposition 2.9]{Eno21}; \cite[Proposition 2.5]{Oga21})} \ $\operatorname{def}\mathbb E$ is a Serre subcategory of ${\rm coh}\mathcal{A}$.
  In particular, if $\mathcal A$ has weak kernels then $\operatorname{def}\mathbb E$ is a Serre subcategory of $\mathrm{fp}(\mathcal A^{\mathrm{op}},\mathrm{Ab})$.

  \vskip5pt

  $(2)$ \ {\rm (\cite[Proposition 2.16]{Oga21})} \ If $(\mathcal A,\mathbb E,\mathfrak s)$ has weak kernels and enough projective objects, then
  $\operatorname{def}\mathbb E = \mathrm{fp}((\mathcal A/\mathcal P(\mathcal A))^{\mathrm{op}},\mathrm{Ab})$, and the Serre quotient
  $\mathrm{fp}(\mathcal A^{\mathrm{op}},\mathrm{Ab})/\operatorname{def}\mathbb E \cong  \mathrm{fp}(\mathcal P(\mathcal A)^{\mathrm{op}},\mathrm{Ab}).$
\end{theorem}

\vskip5pt

Applying Theorem \ref{prop_ogawa_colocalisation} to exact model structures,  one has

\vskip5pt

\begin{corollary}\label{defect} \ Let $(\mathcal C,\mathcal F,\mathcal W)$ be a Hovey triple in a weakly idempotent complete extriangulated category $(\mathcal A,\mathbb E,\mathfrak s)$,
  and $(\frac{\mathcal C\cap\mathcal F}{\mathcal C\cap\mathcal F\cap\mathcal W}, \ \overline{\mathbb E}, \ \overline{\mathfrak s})$ the induced extriangulated category. Assume that $\mathcal C\cap\mathcal F$ is essentially small and has enough projective objects. Then

  \vskip10pt

  $(1)$ \ $\operatorname{def}\overline{\mathbb E}$ is a Serre subcategory of $\mathrm{fp}((\frac{\mathcal C\cap\mathcal F}{\mathcal C\cap\mathcal F\cap\mathcal W})^{\mathrm{op}},\mathrm{Ab});$ and
  $\operatorname{def}\overline{\mathbb E} = \mathrm{fp}((\frac{\mathcal C\cap\mathcal F}{\mathcal P(\mathcal C\cap\mathcal F)})^{\!\mathrm{op}}, \ \mathrm{Ab}).$

  \vskip10pt

  $(2)$ \ The Serre quotient
  $\mathrm{fp}((\frac{\mathcal C\cap\mathcal F}{\mathcal C\cap\mathcal F\cap\mathcal W})^{\mathrm{op}},\mathrm{Ab})/\operatorname{def}\overline{\mathbb E} \cong  \mathrm{fp}((\frac{\mathcal P(\mathcal C\cap\mathcal F)}{\mathcal C\cap\mathcal F\cap\mathcal W})^{\mathrm{op}},\mathrm{Ab}).$
\end{corollary}

\begin{proof} \ Apply Theorem \ref{prop_ogawa_colocalisation} to the induced extriangulated category $(\frac{\mathcal C\cap\mathcal F}{\mathcal C\cap\mathcal F\cap\mathcal W}, \ \overline{\mathbb E}, \ \overline{\mathfrak s})$.
  Since $\frac{\mathcal C\cap\mathcal F}{\mathcal C\cap\mathcal F\cap\mathcal W}$ admits a triangulated structure, it has weak kernels.
  Since by the assumption $\mathcal C\cap\mathcal F$ is essentially small and has enough projective objects, it follows that
  $\frac{\mathcal C\cap\mathcal F}{\mathcal C\cap\mathcal F\cap\mathcal W}$ is essentially small and has enough projective objects, with
  $\mathcal P(\frac{\mathcal C\cap\mathcal F}{\mathcal C\cap\mathcal F\cap\mathcal W}) = \frac{\mathcal P(\mathcal C\cap\mathcal F)}{\mathcal C\cap\mathcal F\cap\mathcal W}$.
  Then the assertion follows directly from Theorem \ref{prop_ogawa_colocalisation}.
\end{proof}

\vskip5pt

Let $(\mathcal A,\mathbb E,\mathfrak s)$ be an extriangulated category. Enomoto \cite{Eno21} classified all the closed additive subbifunctors of $\mathbb E$ by the Serre subcategories of $\operatorname{def}\mathbb E$.

\vskip5pt

\begin{theorem} \label{prop_enomoto_correspondence} {\rm (\cite[Theorem~B]{Eno21})} \  Let $(\mathcal A,\mathbb E,\mathfrak s)$ be an essentially small extriangulated category.
  Then there is an isomorphism of posets$:$
  \[
    \left\{
    \textup{closed additive subbifunctors of $\mathbb E$}
    \right\}
    \longrightarrow
    \left\{
    \textup{Serre subcategories of }\operatorname{def}\mathbb E
    \right\},
    \qquad
    \mathbb F\longmapsto\operatorname{def}\mathbb F
  \]
  with the inverse given by $\mathcal S\mapsto \mathbb F_{\mathcal S}$, where  $\mathbb F_{\mathcal S}(Z,X)
    =
    \{\delta\in\mathbb E(Z,X)\mid\delta^*\in\mathcal S\}$.
\end{theorem}

\vskip5pt

Combining Corollary \ref{defect} and \Cref{prop_enomoto_correspondence},
one obtains the following result which classifies the closed extriangulated substructures lying between $\overline{\mathbb E}$ and $\operatorname{Ext}_{\frac{\mathcal C\cap\mathcal F}{\mathcal C\cap\mathcal F\cap\mathcal W}}^1$ in terms of Serre subcategories of $\mathrm{fp}((\frac{\mathcal P(\mathcal C\cap\mathcal F)}{\mathcal C\cap\mathcal F\cap\mathcal W})^{\!\mathrm{op}}, \ \mathrm{Ab})$.

\vskip10pt

\begin{theorem} \label{thm_enomoto_ogawa} \ Let $(\mathcal C,\mathcal F,\mathcal W)$ be a  Hovey triple in a weakly idempotent complete extriangulated category $(\mathcal A,\mathbb E,\mathfrak s)$,
  and $(\frac{\mathcal C\cap\mathcal F}{\mathcal C\cap\mathcal F\cap\mathcal W}, \ \overline{\mathbb E}, \ \overline{\mathfrak s})$
  the induced extriangulated structure.
  Assume that $\mathcal C\cap\mathcal F$ is essentially small.

  \vskip5pt

  $(1)$  \ If $\mathcal C\cap\mathcal F$ has enough projective objects, then the map $\mathbb F\mapsto {\rm def}\mathbb F$ gives an isomorphism of posets$:$

  \[\xymatrix{\left\{\mathbb F \ \middle| \ \overline{\mathbb E}\subseteq\mathbb F,
    \ \mathbb F \ \textup{is a closed additive subbifunctor of} \operatorname{Hom}_{\frac{\mathcal C\cap\mathcal F}{\mathcal C\cap\mathcal F\cap\mathcal W}}(-, \Sigma-)\right\}\ar[d] \\
    \left\{\textup{Serre subcategories of} \ \mathrm{fp}((\frac{\mathcal P(\mathcal C\cap\mathcal F)}{\mathcal C\cap\mathcal F\cap\mathcal W})^{\mathrm{op}}, \mathrm{Ab}). \right\}}\]

  \vskip5pt

  $(1')$ \ If  $\mathcal C\cap\mathcal F$ has enough injective objects, then there is an isomorphism of posets from
  $\{\mathbb F \ |\
    \overline{\mathbb E}\subseteq\mathbb F,
    \ \mathbb F\text{ is a closed additive subbifunctor of} \ \operatorname{Hom}_{\frac{\mathcal C\cap\mathcal F}{\mathcal C\cap\mathcal F\cap\mathcal W}}(-, \Sigma-)\}$ to $\{\textup{Serre subcategories of} \ \mathrm{fp}(\frac{\mathcal I(\mathcal C\cap\mathcal F)}{\mathcal C\cap\mathcal F\cap\mathcal W}, \mathrm{Ab})\}$.

  \vskip5pt

  $(2)$ \ If $\mathcal C\cap\mathcal F$ has enough projective objects and enough injective objects, then there is an isomorphism of posets
  from $\{\textup{Serre subcategories of }
    \ \mathrm{fp}((\frac{\mathcal P(\mathcal C\cap\mathcal F)}{\mathcal C\cap\mathcal F\cap\mathcal W})^{\mathrm{op}}, \mathrm{Ab})\}$ to $\{
    \textup{Serre subcategories of }
    \ \mathrm{fp}(\frac{\mathcal I(\mathcal C\cap\mathcal F)}{\mathcal C\cap\mathcal F\cap\mathcal W}, \mathrm{Ab})\}$.
\end{theorem}

\begin{proof} \ Consider Nakaoka-Palu's triangulated structure $(\frac{\mathcal C\cap\mathcal F}{\mathcal C\cap\mathcal F\cap\mathcal W}, \ \Sigma, \ \triangle)$, which is an extriangulated category with the additive bifunctor
  $\operatorname{Hom}_{\frac{\mathcal C\cap\mathcal F}{\mathcal C\cap\mathcal F\cap\mathcal W}}(-, \Sigma-)$.
  Since $\frac{\mathcal C\cap\mathcal F}{\mathcal C\cap\mathcal F\cap\mathcal W}$ is triangulated, it has weak kernels. Thus
  $\mathrm{fp}((\frac{\mathcal C\cap\mathcal F}{\mathcal C\cap\mathcal F\cap\mathcal W})^{\mathrm{op}},\mathrm{Ab})$ is abelian, by \Cref{lem_weak_kernels}.
  Since every morphism of $\frac{\mathcal C\cap\mathcal F}{\mathcal C\cap\mathcal F\cap\mathcal W}$ can be completed to a distinguished triangle, every finitely presented functor is isomorphic to $\delta^*$ for some
  $\delta\in\operatorname{Hom}_{\frac{\mathcal C\cap\mathcal F}{\mathcal C\cap\mathcal F\cap\mathcal W}}(-, \Sigma-)$. Hence
  \[\xymatrix{\operatorname{def}\operatorname{Hom}_{\frac{\mathcal C\cap\mathcal F}{\mathcal C\cap\mathcal F\cap\mathcal W}}(-, \Sigma-)=\mathrm{fp}((\frac{\mathcal C\cap\mathcal F}{\mathcal C\cap\mathcal F\cap\mathcal W})^{\mathrm{op}},\mathrm{Ab}).}\]
  Applying \Cref{prop_enomoto_correspondence} to the extriangulated category
  $(\frac{\mathcal C\cap\mathcal F}{\mathcal C\cap\mathcal F\cap\mathcal W}, \operatorname{Hom}_{\frac{\mathcal C\cap\mathcal F}{\mathcal C\cap\mathcal F\cap\mathcal W}}(-, \Sigma-), \mathfrak s_\triangle)$,
  one has an isomorphism of posets:
  \[\xymatrix{\left\{\textup{closed additive subbifunctors of} \ \operatorname{Hom}_{\frac{\mathcal C\cap\mathcal F}{\mathcal C\cap\mathcal F\cap\mathcal W}}(-, \Sigma-)\right\}\ar[d] \\  \left\{
    \textup{Serre subcategories of } \mathrm{fp}((
    \frac{\mathcal C\cap\mathcal F}{\mathcal C\cap\mathcal F\cap\mathcal W})^{\mathrm{op}}, \mathrm {Ab})\right\}}\]
  sending $\mathbb F$ to $\operatorname{def}\mathbb F$.
  Restricting the poset in the upper row to the interval above $\overline{\mathbb E}$,  the corresponding image is
  \[\xymatrix{\left\{\textup{Serre subcategory} \ \mathcal V \ \mbox{of } \mathrm{fp}((\frac{\mathcal C\cap\mathcal F}{\mathcal C\cap\mathcal F\cap\mathcal W})^{\mathrm{op}},\mathrm{Ab}) \ \middle|\
    \operatorname{def}\overline{\mathbb E}\subseteq\mathcal V  \right\}.}\]
  This is isomorphic to the poset
  \[\xymatrix{\left\{\textup{Serre subcategory of} \ \mathrm{fp}((\frac{\mathcal C\cap\mathcal F}{\mathcal C\cap\mathcal F\cap\mathcal W})^{\mathrm{op}}, \mathrm{Ab})/ \operatorname{def}\overline{\mathbb E}\right\}}\]
  By \Cref{defect} one has $\mathrm{fp}((\frac{\mathcal C\cap\mathcal F}{\mathcal C\cap\mathcal F\cap\mathcal W})^{\mathrm{op}},\mathrm{Ab})/\operatorname{def}\overline{\mathbb E} \cong  \mathrm{fp}((\frac{\mathcal P(\mathcal C\cap\mathcal F)}{\mathcal C\cap\mathcal F\cap\mathcal W})^{\mathrm{op}},\mathrm{Ab})$, and hence one gets
  the desired  isomorphism of posets.

  \vskip5pt

  The assertion $(1')$ is the dual of $(1)$. Note that $(\mathcal F^{\mathrm{op}}, \mathcal C^{\mathrm{op}},\mathcal W^{\mathrm{op}})$ is a Hovey triple in a weakly idempotent complete extriangulated category $(\mathcal A^{\mathrm{op}}, \mathbb E^{\mathrm{op}},\mathfrak s^{\mathrm{op}})$ (\cite[p.~1.6, Remark~2]{Q1}, \cite[Caution~2.20]{NP19}). The assertion $(2)$ follows from $(1)$ and $(1')$.
\end{proof}

\vskip5pt

\subsection{An application}

$\,$

\vskip5pt

Let $A$ be a representation-finite finite-dimensional algebra over a field $k$, and let $(\mathcal C,\mathcal F,\mathcal W)$ be a Hovey triple in $A\text{-}\mathrm{mod}$.
Viewing $A\text{-}\mathrm{mod}$ as an extriangulated category,  one writes it as $(A\text{-}\mathrm{mod}, \mathbb E, \mathfrak s)$ with $\mathbb E = {\rm Ext}^1_A(-, -): (A\text{-}\mathrm{mod})^{\rm op}\times A\text{-}\mathrm{mod} \longrightarrow k\mbox{-}{\rm mod}$,
where $k\mbox{-}{\rm mod}$ is the category of finite-dimensional $k$-linear spaces. Then one has the induced extriangulated category
$(\frac{\mathcal C\cap\mathcal F}{\mathcal C\cap\mathcal F\cap\mathcal W}, \overline {\mathbb E}, \overline {\mathfrak s})$. Also, Nakaoka--Palu's
triangulated structure $(\frac{\mathcal C\cap\mathcal F}{\mathcal C\cap\mathcal F\cap\mathcal W}, \Sigma, \triangle)$ induces an extriangulated category, denoted by
$(\frac{\mathcal C\cap\mathcal F}{\mathcal C\cap\mathcal F\cap\mathcal W}, \Hom_{\frac{\mathcal C\cap\mathcal F}{\mathcal C\cap\mathcal F\cap\mathcal W}}(-, \Sigma -), \mathfrak s_\triangle)$.
We will study the closed $k$-linear subbifunctors $\mathbb F$ between $\overline {\mathbb E}$ and $\Hom_{\frac{\mathcal C\cap\mathcal F}{\mathcal C\cap\mathcal F\cap\mathcal W}}(-, \Sigma -)$.

\vskip10pt

\begin{corollary} \label{number} \ Let $A$ be a representation-finite finite-dimensional algebra over a field $k$, and $(\mathcal C,\mathcal F,\mathcal W)$ a Hovey triple in $A\text{-}\mathrm{mod}$. Then
  $\mathcal C\cap\mathcal F$ has enough projective and enough injective objects, with $\left|\operatorname{Ind}\mathcal P(\mathcal C\cap\mathcal F)\right| = \left|\operatorname{Ind}\mathcal I(\mathcal C\cap\mathcal F)\right|;$ there are $2^n$ closed $k$-linear subbifunctors $\mathbb F$ satisfying
  $\overline{\mathbb E}\subseteq\mathbb F\subseteq
    \Hom_{\frac{\mathcal C\cap\mathcal F}{\mathcal C\cap\mathcal F\cap\mathcal W}}(-, \Sigma -)$,
  and there are $2^n$ Serre subcategories of \ $\mathrm{fp}((\frac{\mathcal P(\mathcal C\cap\mathcal F)}{\mathcal C\cap\mathcal F\cap\mathcal W})^{\mathrm{op}}, k\mbox{-}{\rm mod})$, where $n = \left|
    \operatorname{Ind}\mathcal P(\mathcal C\cap\mathcal F) \right| - \left|
    \operatorname{Ind}(\mathcal C\cap\mathcal F\cap\mathcal W)
    \right|$.
\end{corollary}

\begin{proof} \ By \Cref{corollary_D_enough_PI}, $\mathcal C\cap\mathcal F$ has enough projective and enough injective objects. To apply \Cref{thm_enomoto_ogawa}
  in this case,  one should replace $\mathrm{fp}((\frac{\mathcal P(\mathcal C\cap\mathcal F)}{\mathcal C\cap\mathcal F\cap\mathcal W})^{\mathrm{op}}, \mathrm{Ab})$ by
  $\mathrm{fp}((\frac{\mathcal P(\mathcal C\cap\mathcal F)}{\mathcal C\cap\mathcal F\cap\mathcal W})^{\mathrm{op}}, k\mbox{-}{\rm mod})$, the category of  finitely presented contravariant $k$-linear functors from
  $\frac{\mathcal P(\mathcal C\cap\mathcal F)}{\mathcal C\cap\mathcal F\cap\mathcal W}$ to $k\mbox{-}{\rm mod}$.
  Thus the closed $k$-linear subbifunctors between $\overline {\mathbb E}$ and $\Hom_{\frac{\mathcal C\cap\mathcal F}{\mathcal C\cap\mathcal F\cap\mathcal W}}(-, \Sigma -)$ are in bijection with the Serre subcategories of $\mathrm{fp}((\frac{\mathcal P(\mathcal C\cap\mathcal F)}{\mathcal C\cap\mathcal F\cap\mathcal W})^{\mathrm{op}}, k\mbox{-}{\rm mod})$.

  \vskip5pt

  Since $A$ is representation-finite, there is an object $P$ such that $\frac{\mathcal P(\mathcal C\cap\mathcal F)}{\mathcal C\cap\mathcal F\cap\mathcal W} = {\rm add} P$. Hence
  \[\xymatrix{\mathrm{fp}((\frac{\mathcal P(\mathcal C\cap\mathcal F)}{\mathcal C\cap\mathcal F\cap\mathcal W})^{\mathrm{op}}, k\mbox{-}{\rm mod})\cong \big(\operatorname{End}_{\frac{\mathcal P(\mathcal C\cap\mathcal F)}{\mathcal C\cap\mathcal F\cap\mathcal W}}(P)\big)^{\mathrm{op}}\text{-}\mathrm{mod}.}\] The module category $\big(\operatorname{End}_{\frac{\mathcal P(\mathcal C\cap\mathcal F)}{\mathcal C\cap\mathcal F\cap\mathcal W}}(P)\big)^{\mathrm{op}}\text{-}\mathrm{mod}$ has $n_{\mathcal P}$ simple modules,
  where $n_{\mathcal P}$ is the number of pairwise non-isomorphic indecomposable direct summands of $P$, and hence $n_{\mathcal P} = \left|
    \operatorname{Ind}\mathcal P(\mathcal C\cap\mathcal F) \right| - \left|
    \operatorname{Ind}(\mathcal C\cap\mathcal F\cap\mathcal W)
    \right|$.
  Thus $\mathrm{fp}((\frac{\mathcal P(\mathcal C\cap\mathcal F)}{\mathcal C\cap\mathcal F\cap\mathcal W})^{\mathrm{op}}, k\mbox{-}{\rm mod})$ is a length abelian category with $n_{\mathcal P}$ simple objects.
  A Serre subcategory of a length abelian category is uniquely determined by the simple objects it contains, and every subset of the simple objects uniquely determines a Serre subcategory. In this way one sees that $\mathrm{fp}((\frac{\mathcal P(\mathcal C\cap\mathcal F)}{\mathcal C\cap\mathcal F\cap\mathcal W})^{\mathrm{op}}, k\mbox{-}{\rm mod})$ has $2^{n_{\mathcal P}}$ Serre subcategories.

  \vskip5pt

  Dually,  $\mathrm{fp}(\frac{\mathcal I(\mathcal C\cap\mathcal F)}{\mathcal C\cap\mathcal F\cap\mathcal W}, k\mbox{-}{\rm mod})$ has $2^{n_{\mathcal I}}$ Serre subcategories, where $n_{\mathcal I} = \left|
    \operatorname{Ind}\mathcal I(\mathcal C\cap\mathcal F) \right| - \left|
    \operatorname{Ind}(\mathcal C\cap\mathcal F\cap\mathcal W)
    \right|$.
  By Theorem \ref{thm_enomoto_ogawa}$(2)$ one has
  $2^{n_{\mathcal P}}=2^{n_{\mathcal I}}$, so $n_{\mathcal P}=n_{\mathcal I}$. Denote this common number by $n$. Then we are done by Theorem \ref{thm_enomoto_ogawa}$(1)$. \end{proof}

\vskip5pt

Note that the corollary above does not give the number of equivalence classes of Serre subcategories of \ $\mathrm{fp}((\frac{\mathcal P(\mathcal C\cap\mathcal F)}{\mathcal C\cap\mathcal F\cap\mathcal W})^{\mathrm{op}}, k\mbox{-}{\rm mod})$, since different
Serre subcategories are possibly equivalent as abelian categories.

\vskip5pt

\section{\bf Appendix}

$\,$

This Appendix contains the technique facts which have been used in $\S 3$ and $\S 4$.

\vskip5pt

\subsection{Morphisms factoring through projective objects} $\,$

\vskip5pt

Let $A$ be a finite-dimensional algebra over a field $k$, and $A\text{-}\mathrm{mod}$ the category of finitely generated $A$-modules.
Recall that for $M, N\in A\text{-}\mathrm{mod}$,  one has the Auslander-Reiten formula (see e.g. {\rm \cite[IV, Theorem 2.13]{ASS06}})
$$\mathrm{Ext}^1_A(M,N)\cong D\underline{\mathrm{Hom}}_A(\tau^{-1}N, M) \cong D \overline{\mathrm{Hom}}_A(N, \tau M)$$
where $\tau = D\circ \mathrm{Tr}$,  $\tau^{-1} = \mathrm{Tr} \circ D$,  $\mathrm{Tr}$ is the Auslander--Reiten transpose, and $D ={\rm Hom}_k(-, k)$.

\vskip5pt

\begin{lemma} \label{radical} \ Let $\Lambda$ be any finite-dimensional algebra.
  \vskip5pt
  $(1)$ \ Let $i: M \longrightarrow N$ be a $\Lambda$-monomorphism between non-injective indecomposable modules.
  Then $i$ does not factor through any injective module.
  \vskip5pt
  $(1')$ \ Let $p: M \longrightarrow N$ be a $\Lambda$-epimorphism between non-projective indecomposable modules.
  Then $p$ does not factor through any projective module.
\end{lemma}

\begin{proof} \ We only prove $(1)$.  Assume that $i$ factors through an injective module $I$, i.e., there is a factorization $i = g\circ f$ with $f :M \longrightarrow I$ and $g : I \longrightarrow N$.
  Then $f$ factors through the injective envelope $j : M \longrightarrow I(M)$. Thus $f = \varphi \circ j$ for some $\varphi: I(M)\longrightarrow I$.
  Hence $i = g \circ \varphi \circ j$. Then $j = h \circ i$ for some $h: N \longrightarrow I(M)$, and thus $j = h \circ g \circ \varphi \circ j$.
  Since $j$ is left minimal, $h \circ g \circ \varphi$ is an isomorphism, i.e., $I(M)$ is a direct summand of $N$, a contradiction.
\end{proof}

\vskip5pt

Let $\mathcal A$ be an extriangulated category, $M, N$ objects of $\mathcal A$. We fix the notation $\Omega M$ and $\mho N$. If $\mathcal A$ has enough projective objects,  then
there is an $\mathbb E$-triangle \ $\Omega M \xlongrightarrow{\kappa} P \xlongrightarrow {p} M \xdashrightarrow \delta $ with $P\in\mathcal P(\mathcal A)$.
If $\mathcal A$ has enough injective objects, then
there is an $\mathbb E$-triangle \  $N\xlongrightarrow{i} I \xlongrightarrow {c} \mho N \xdashrightarrow \eta$ with $I\in\mathcal I(\mathcal A)$.

\vskip5pt

\begin{lemma} \label{extinfrobenius} \ Let $\mathcal A$ be a Frobenius extriangulated category. Then \ $$\mathbb{E}(M,N)\cong \underline{\mathrm{Hom}}_\mathcal A(M,\mho N) \cong \underline{\mathrm{Hom}}_\mathcal A(\Omega M, N).$$
\end{lemma}

\begin{proof} \ We only show $\mathbb{E}(M,N)\cong \underline{\mathrm{Hom}}_\mathcal A(M,\mho N)$.
  The isomorphism $\mathbb{E}(M,N)$ $\cong \underline{\mathrm{Hom}}_\mathcal A(\Omega M,N)$ can be proved dually.
  One has the exact sequence
  \begin{equation*}
    \resizebox{1\linewidth}{!}{$
        \mathrm{Hom}_{\mathcal A}(M, N) \xlongrightarrow{\mathrm{Hom}_{\mathcal A}(M, i)} \mathrm{Hom}_{\mathcal A}(M, I) \xlongrightarrow{\mathrm{Hom}_{\mathcal A}(M, c)}  \mathrm{Hom}_{\mathcal A}(M, \mho N) \xlongrightarrow{\eta_\sharp}  \mathbb{E}(M, N) \longrightarrow 0
      $}
  \end{equation*}
  Since $\mathcal{A}$ is Frobenius, it is clear that $\mathrm{Im}(\mathrm{Hom}_{\mathcal A}(M, c)) = \mathcal{P}(M, \mho N)$.
  Then one has

  \vskip8pt

  \hskip15pt $\mathbb{E}(M, N) \cong \frac{\mathrm{Hom}_{\mathcal A}(M, \mho N)}{\operatorname{Im}(\mathrm{Hom}_{\mathcal A}(M, c))} = \frac{\mathrm{Hom}_{\mathcal A}(M, \mho N)}{\mathcal{P}(M, \mho N)} = \underline{\mathrm{Hom}}_\mathcal A(M, \mho N).
  $
\end{proof}

\subsection{Facts on Nakayama algebras}

In this subsection  $A$ is a Nakayama algebra without simple projective modules.

\vskip5pt

\begin{lemma}\label{simphom} \ For any simple $A$-module $S_{i} = M_{i, 1}$ and $A$-module $M_{j,l}$, one has
  \vskip5pt
  $(1)$ \ If $M_{j,l}$ is non-projective, then $\underline{\mathrm{Hom}}_A(M_{j,l}, M_{i, 1}) = 0$ if and only if $j \ne i$.
  \vskip5pt
  $(2)$ \ If $M_{j,l}$ is non-injective, then $\overline{\mathrm{Hom}}_A(M_{i, 1}, M_{j,l}) = 0$ if and only if $j \ne i - l + 1$.
\end{lemma}

\begin{proof} \  We show $(1)$.
  In fact, if $j \neq i$, then $\mathrm{Hom}_A(M_{j,l}, M_{i, 1}) = 0$.
  Conversely, it suffices to show that if $j = i$ then $\underline{\mathrm{Hom}}_A(M_{j,l}, M_{i, 1}) \neq 0$. By assumption, $M_{j,l}$ is not projective.
  One has the canonical  epimorphism $p : M_{i, l} \twoheadrightarrow \mathrm{top}(M_{i, l}) =  M_{i, 1}$ between non-projective indecomposable modules.
  By \Cref{radical}, $p$ does not factor through a  projective module. Thus $\underline{\mathrm{Hom}}_A(M_{i,l}, M_{i, 1}) \neq 0$.

  \vskip5pt

  To justify the index,  we show $(2)$. If $j \neq i-l+1$, then $\mathrm{Hom}_A(M_{i,1}, M_{j, l}) = 0$, since $\mathrm{soc}(M_{j,l}) = S_{j+l-1} \ncong S_i$.
  Conversely, it suffices to show that if $j = i-l+1$ then $\overline{\mathrm{Hom}}_A(M_{i, 1}, M_{j,l}) \neq 0$.
  By assumption, $M_{j,l}$ is not injective. Then the monomorphism $M_{i, 1} = \mathrm{soc}(M_{j, l}) \hookrightarrow M_{j,l}$ between non-injective modules does not factor through injective modules (cf. \Cref{radical}). Thus $\overline{\mathrm{Hom}}_A(M_{i,1}, M_{j,l}) \neq 0$. \end{proof}

\vskip5pt

\begin{lemma}\label{factorthrough} \ Let  $s \le r\le l\le c_i$ be positive integers.

  \vskip5pt

  $(1)$ \  For any $A$-epimorphisms $p$ and $\pi$, there is an epimorphism $q$ such that $\pi\circ q=p:$
  \[\xymatrix@R = 0.5cm{& M_{i,r}\ar@{->>}[d]^\pi \\
        M_{i,l}\ar@{->>}[r]^p\ar@{.>>}[ur]^q & M_{i,s}}\]

  \vskip5pt

  $(1')$ \ For any $A$-monomorphisms $e$ and $\sigma$, there is a monomorphism $e'$ such that $e'\circ\sigma=e:$
  \[\xymatrix@R = 0.5cm{M_{i+l-s,s}\ar@{^(->}[r]^-e \ar@{^(->}[d]_\sigma & M_{i,l}\\
        M_{i+l-r,r}\ar@{^(.>}[ur]_{e'}}\]

\end{lemma}

\begin{proof} \ We prove $(1)$. \ For $x\in\{l,r,s\}$, since $M_{i, x}$ is a quotient of indecomposable projective module $P_i = Ae_i$, one can choose a path basis $\{\mathbf e^{(x)}_0,\mathbf e^{(x)}_1,\ldots,\mathbf e^{(x)}_{x-1}\}$ of $M_{i,x}$,
  where $\mathbf e^{(x)}_j$ starts at vertex $i$ and ends at vertex $i+j, \ 0\le j\le x-1$. In particular, $\mathbf e^{(x)}_0$ is the image of $e_i$ in  $M_{i, x}$.
  Consider the right shift endomorphism
  \[
    R_x : M_{i,x} \longrightarrow M_{i,x}, \qquad \mathbf e^{(x)}_j \mapsto
    \begin{cases}
      \mathbf e^{(x)}_{j+n}, & j+n<x,     \\
      0,                     & j+n\geq x.
    \end{cases}
  \]
  Clearly $R_x$ is nilpotent.
  For $x, y \in \{l,r,s\}$ with $x\geq y$, the canonical quotient is given by
  \[
    \rho_{x,y}: M_{i,x}\twoheadrightarrow M_{i,y}, \qquad \mathbf e^{(x)}_j \mapsto
    \begin{cases}
      \mathbf e^{(y)}_j, & 0\leq j<y, \\
      0,                 & y\leq j<x.
    \end{cases}
  \]
  Thus $\rho_{l,s}=\rho_{r,s}\circ\rho_{l,r}$ and $\rho_{r,s}\circ R_r=R_s\circ\rho_{r,s}$.
  For $x\geq y$, since $M_{i,x} = A\mathbf e^{(x)}_0$, any epimorphism
  $h:M_{i,x}\twoheadrightarrow M_{i,y}$ is uniquely determined by
  \[
    h(\mathbf e^{(x)}_0)
    =a_0\mathbf e^{(y)}_0+a_1\mathbf e^{(y)}_n+\cdots
    +a_m\mathbf e^{(y)}_{mn},
    \qquad
    m=\lfloor (y-1) / n\rfloor,
    \quad a_0\neq0.
  \]
  Hence there are polynomials $a(T),b(T)\in k[T]$ with $a(0),b(0)\neq0$ such that
  \[
    p=a(R_s)\circ\rho_{l,s},
    \qquad
    \pi=b(R_s)\circ\rho_{r,s}.
  \]
  Note that $R_x$ is nilpotent and $a(0), b(0) \neq 0$.
  Then $a(R_s)$ and $b(R_s)$ are automorphisms.
  Similarly $b(R_r)^{-1}$ exists, which is again a polynomial in $R_r$. Put
  \[
    c(R_r)=b(R_r)^{-1} \circ a(R_r), \qquad q=c(R_r)\circ\rho_{l,r}.
  \]
  Since $c(0)=b(0)^{-1} \cdot a(0)\neq0$, $c(R_r)$ is an automorphism.
  Thus $q$ is an epimorphism, and
  \[
    \pi\circ q =b(R_s)\circ\rho_{r,s}\circ c(R_r)\circ\rho_{l,r} =b(R_s)\circ c(R_s)\circ\rho_{r,s}\circ\rho_{l,r}=a(R_s)\circ\rho_{l,s}=p.
  \]
  This proves $(1)$.
  The proof of $(1')$ is dual.
\end{proof}

\vskip5pt

\begin{fact}\label{pp} \ Let $i\in\mathbb Z/n\mathbb Z$ and let $l$ be a positive integer. Let $r,s$ be non-negative integers with $r\le l$, and adopt the convention $M_{a,0}=0$.
  Suppose that $l+s\le c_i$ and $l-r+s\le c_{i+r}$.
  Then there is the following diagram,  which is a pullback square and a pushout square:
  \[
    \xymatrix{
    M_{i+r,l-r+s} \ar@{->>}[d]_-{p} \ar@{^{(}->}[r]^-{e} & M_{i,l+s} \ar@{->>}[d]^-{p'} \\
    M_{i+r,l-r} \ar@{^{(}->}[r]^-{e'} & M_{i,l}
    }
  \]
  In particular, there is an exact sequence
  \[
    0 \longrightarrow M_{i+r,l-r+s} \xrightarrow{\binom pe}  M_{i+r,l-r}\oplus M_{i,l+s}  \xrightarrow{(-e', p')} M_{i,l} \longrightarrow 0.
  \]
\end{fact}

\begin{proof} \ Since $\operatorname{soc} M_{i+r, l-r+s} = S_{i+l+s-1} = \operatorname{soc} M_{i,l+s}$, one has $M_{i+r,l-r+s}=\operatorname{rad}^{r}(M_{i,l+s})$, by the structure of uniserial modules.  Similarly, if $r < l$, then
$\operatorname{soc} M_{i+r, l-r} = S_{i+l-1} = \operatorname{soc} M_{i,l}$, and hence $M_{i+r,l-r}=\operatorname{rad}^{r}(M_{i,l})$, and this also holds for $r= l$.
  Let $e,e'$ be the corresponding inclusions, and let $p,p'$ be the canonical quotient maps obtained by removing the last $s$ composition factors.
  Since $p$ is the restriction of $p'$ on $\operatorname{rad}^{r}(M_{i,l+s}) = M_{i+r,l-r+s}$, it follows that $p'\circ e=e'\circ p$.

  \vskip5pt

  Since ${\rm Ker} \ p = M_{i+l, s} = {\rm Ker} \ p'$, the commutative square is a  pullback square of $e'$ and $p'$. Since ${\rm Coker}\ e = M_{i, r} = {\rm Coker}\ e'$, the commutative square is a  pushout square of $e$ and $p$.
  This completes the proof.
\end{proof}

\vskip5pt

{\bf AI Disclosure} \ No AI tools were used throughout this paper.

\end{document}